\documentclass[12pt, draft]{amsart}
\usepackage{amssymb, amstext, amscd, amsmath, amssymb}
\usepackage{mathtools, xypic, paralist, color, dsfont, rotating}
\usepackage{verbatim}
\usepackage{enumerate}
\usepackage{enumitem}
\usepackage{calc}
\usepackage{float}

\numberwithin{equation}{section}

\usepackage[a4paper]{geometry}
\usepackage{quoting}
\quotingsetup{vskip=.1in}
\quotingsetup{leftmargin=.17in}
\quotingsetup{rightmargin=.17in}

\makeatletter
\def\@cite#1#2{{\m@th\upshape\bfseries%
[{#1\if@tempswa{\m@th\upshape\mdseries, #2}\fi}]}}
\makeatother
\theoremstyle{plain}
\newtheorem{theorem}{Theorem}[section]
\newtheorem{corollary}[theorem]{Corollary}
\newtheorem{proposition}[theorem]{Proposition}

\theoremstyle{definition}
\newtheorem{definition}[theorem]{Definition}
\newtheorem{example}[theorem]{Example}
\newtheorem{examples}[theorem]{Examples}

\newtheorem{remark}[theorem]{Remark}

\newtheorem{question}[theorem]{Question}
\newtheorem*{acknow}{Acknowledgements}
\theoremstyle{remark}

\mathtoolsset{centercolon}
  \newcommand{\A}{{\mathcal{A}}}
  \newcommand{\B}{{\mathcal{B}}}

  \newcommand{\F}{{\mathcal{F}}}

  \newcommand{\I}{{\mathcal{I}}}
  \newcommand{\J}{{\mathcal{J}}}
  \newcommand{\K}{{\mathcal{K}}}
\renewcommand{\L}{{\mathcal{L}}}
  
  \newcommand{\N}{{\mathcal{N}}}
\renewcommand{\O}{{\mathcal{O}}}

\renewcommand{\S}{{\mathcal{S}}}
  \newcommand{\T}{{\mathcal{T}}}

\def\al{\alpha}
\def\be{\beta}
\def\ga{\gamma}
\def\de{\delta}

\def\ka{\kappa}
\def\la{\lambda}
\def\La{\Lambda}

\def\Om{\Omega}
\def\si{\sigma}
\def\Si{\Sigma}

\newcommand{\bC}{\mathbb{C}}

\newcommand{\bF}{\mathbb{F}}
\newcommand{\bN}{\mathbb{N}}
\newcommand{\bT}{\mathbb{T}}
\newcommand{\bZ}{\mathbb{Z}}

\newcommand{\fI}{{\mathfrak{I}}}

\newcommand{\fL}{{\mathfrak{L}}}

\newcommand{\Bi}{{\mathbf{i}}}

\newcommand{\AND}{\text{ and }}

\newcommand{\foral}{\text{ for all }}
\newcommand{\qand}{\quad\text{and}\quad}

\newcommand{\qiff}{\quad\text{if and only if}\quad}
\newcommand{\qfor}{\quad\text{for}\; }

\newcommand{\ca}{\mathrm{C}^*}

\newcommand{\ol}{\overline}

\newcommand{\wh}{\widehat}

\newcommand{\Aut}{\operatorname{Aut}}

\newcommand{\End}{\operatorname{End}}

\newcommand{\id}{{\operatorname{id}}}

\newcommand{\mt}{\emptyset}

\newcommand{\spn}{\operatorname{span}}

\newcommand{\sumoplus}{\operatornamewithlimits{\sum \strut^\oplus}}
\newcommand{\supp}{\operatorname{supp}}

\newcommand{\sca}[1]{\left\langle#1\right\rangle} 
\newcommand{\nor}[1]{\left\Vert #1\right\Vert} 
\newcommand{\bo}[1]{\mathbf{#1}} 
\newcommand{\un}[1]{{\underline{#1}}} 
\newcommand{\umu}{\underline\mu}
\newcommand{\unu}{\underline\nu}
\newcommand{\uw}{\underline{w}}
\newcommand{\uq}{\underline{q}}
\newcommand{\umt}{\underline\mt}
\newcommand{\fdn}{(\mathbb{F}_+^d)^N}

\addtocontents{toc}{\protect\setcounter{tocdepth}{1}}

\begin{document}

\title[Operator algebras for higher rank analysis]{Operator algebras for higher rank analysis and their application to factorial languages}

\author[A. Dor-On]{Adam Dor-On}
\address{Department of Mathematics\\ University of Illinois at Urbana-Champaign\\ Urbana\\ IL 61801\\ USA}
\email{adoron@illinois.edu}

\author[E.T.A. Kakariadis]{Evgenios T.A. Kakariadis}
\address{School of Mathematics, Statistics and Physics\\ Newcastle University\\ Newcastle upon Tyne\\ NE1 7RU\\ UK}
\email{evgenios.kakariadis@ncl.ac.uk}



\thanks{The first author was partially supported by an Azrieli international postdoctoral fellowship, and an Ontario trillium scholarship. 
This work was supported by EPSRC grant no EP/K032208/1.
Work on this project has been undertaken during a visit of the first author at Newcastle University, funded by the School of Mathematics and Statistics.
}

\maketitle

\begin{abstract}
We study strong compactly aligned product systems of $\mathbb{Z}_+^N$ over a C*-algebra $A$.
We provide a description of their Cuntz-Nica-Pimsner algebra in terms of tractable relations coming from ideals of $A$.
This approach encompasses product systems where the left action is given by compacts, as well as a wide class of higher rank graphs (beyond row-finite).

Moreover we analyze higher rank factorial languages and their C*-algebras.
Many of the rank one results in the literature find here their higher rank analogues.
In particular, we show that the Cuntz-Nica-Pimsner algebra of a higher rank sofic language coincides with the Cuntz-Krieger algebra of its unlabeled follower set higher rank graph.
However there are also differences.
For example, the Cuntz-Nica-Pimsner can lie in-between the first quantization and its quotient by the compactly supported operators.
\end{abstract}



\section{Introduction}

\subsection{Rank one initiative}

A major trend in Operator Theory is the examination of geometric structures through Hilbertian operators.
The aim is to achieve a rigid representation so that properties of the underlying structure have equivalent C*-algebraic formulations.
This idea goes back to the work of Murray and von Neumann. Since then, a rich theory has been developed for group actions on C*-algebras.
Furthermore, in the past 20 years there has been a growing interest in understanding semigroup actions, taking motivation from the well known Cuntz algebras.

It is instructive to highlight the case of directed graphs in this endeavor. Every vertex in a directed graph has a set of edges that it receives. The Toeplitz-Cuntz-Krieger algebra generated by the Fock operators provides an effective model to encode this behaviour.
However the set of received edges should also relate to some vertex and the Fock operators do not reflect this additional information.
One can always quotient out the corresponding relations.
But the aim is to find an appropriate quotient that still preserves a faithful copy of the vertices, and thus the relationships between vertices and edges are not lost.
Cuntz and Krieger \cite{CK80} achieved this initially for $0$-$1$ matrices with no non-zero rows or columns.
Their seminal work was pushed forward in a series of papers by the Australian School and graph C*-algebras reached their final form in \cite{BHRS02, RS04}.
This representation is quite rigid.
For example Bates-Hong-Raeburn-Szymanski \cite{BHRS02} established a correspondence between the gauge-invariant ideals of the Cuntz-Krieger algebra and the hereditary and saturated vertex sets.

Graph C*-algebras fall in the general class of Pimsner algebras coming from a C*-corres\-pon\-dence, introduced by Pimsner \cite{Pim97}.
By now there is a well developed theory that unifies effectively Cuntz-type C*-algebras arising from many examples of transformations in discrete time.
One of the main tools for verifying faithfulness of their representations is the \emph{Gauge Invariant Uniqueness Theorems} (GIUT).
This type of result is motivated by the GIUT for C*-crossed products of abelian groups and was pioneered by an Huef and Raeburn \cite{HR97} for $\bZ_+$.
Essentially, it shows faithfulness of representations of the Cuntz-type C*-algebra from faithfulness of the coefficients, as long as they admit a gauge action.
Pimsner's \cite{Pim97} assumptions were subsequently removed by Fowler, Muhly and Raeburn \cite{FR99, FMR03}, and finally Katsura \cite{Kat04} proved a GIUT in full generality.
The key idea in \cite{Kat04} is to realize the ideal of solutions for certain polynomial equations, as exhibited in \cite{Kak14}.

\subsection{Higher rank impetus}

In the meantime, Fowler \cite{Fow02} pushed further the theory of Pimsner to encounter semigroups of C*-correspondences.
In a loose sense a product system $X$ over a semigroup $P$ is a family of C*-correspondences $\{X_p \mid  p \in P\}$ over the same C*-algebra $A$ such that an associative multiplication identifies $X_{p q}$ with $X_{p} \otimes_A X_{q}$ (and hence takes into consideration the semigroup structure).
C*-correspondences and product systems over $\bZ_+$ are essentially the same; but product systems are not just an artifact for generalizations.
Earlier versions had been examined by Robertson-Steger \cite{RS96, RS99} for operator algebras related to buildings.
They are the higher rank analogues of Cuntz-Krieger algebras and form the cornerstone of the higher rank analysis.
Two necessary restrictions are identified by Fowler \cite{Fow02} to get a tractable model.
First there must be an alignment of compacts that re-captures the semigroup quasi-lattice order structure.
Secondly this has to be reflected through Nica's representations \cite{Nic92} so as to induce universality of the Fock representation.
For example, Nica-covariant representations of $\bZ_+^N$ are by doubly commuting isometries instead of just commuting.

Kumjian-Pask \cite{KP00, KP03} extended the theory of Cuntz-Krieger algebras to row-finite sourceless higher rank graphs.
In their breakthrough work, Raeburn, Sims and Yeend \cite{RS03, RSY03, RSY04} established the theory for the entire class of finitely aligned higher rank graphs.
Higher rank operator algebras have been henceforth studied in depth by various authors, e.g., Davidson-Yang \cite{DY08}, Deaconu-Kumjian-Pask-Sims \cite{DKPS10}, Evans \cite{Eva08}, Farthing-Muhly-Yeend, \cite{FMY05}, Hopenwasser \cite{Hop05}, Kribs-Power \cite{KP06} and Popescu-Zacharias \cite{PZ05}.
In particular, Yeend \cite{Yee07} established the GIUT for higher rank topological graphs, a class that includes both higher rank graphs and $\bZ_+^N$-actions on commutative C*-algebras, and Davidson-Fuller-Kakariadis \cite[Section 4]{DFK14} proved the GIUT for any C*-dynamical system using a tail-adding technique.

Sims-Yeend \cite{SY11} initiated the study of Cuntz-Nica-Pimsner relations with the aim of producing a higher rank analogue of Pimsner algebras and covers very general semigroups (and all $\bZ_+^N$-product systems).
Later Carlsen-Larsen-Sims-Vittadello \cite{CLSV11} established an effective theory that produces co-universal Cuntz-Nica-Pimsner algebras.
In contrast to the $\bZ_+$-case, the description of the covariant representations is tied to the specific representation on the \emph{augmented Fock space}.
These covariant representations are characterized by equalities on finite subsets of the semigroup rather than just on $A$.
This is not a surprise since in general a semigroup can be too pathological compared to $\bZ_+$.
The problem is that, not just $A$, but also the compacts $\K X_p$ act on each $X_q$.
One then has to deal with the combinatorics involved with their injective part, and the augmented Fock representation accounts for all possible mixings.

A second point of comparison is that Katsura's approach \cite{Kat04} allowed for the further exploitation of C*-correspondences.
This includes the C*-envelope, adding tails, minimal extensions and stability of shift relations studied respectively by Kat\-soulis-Kribs \cite{KK06}, Muhly-Tomforde \cite{MT04}, Muhly-Solel \cite{MS98, MS00} and Muhly-Pask-Tomforde \cite{MPT08}, and by the second author with Katsoulis \cite{KK11, KK12}.
Katsura's ideal and the simple algebraic formulation of covariant representations in terms of $A$ has been proven essential to get these results.
Compared to the $\bZ_+$-case, there is a lack of similar higher rank developments.
The necessary, yet complex, augmented Fock space formulation seems to be an obstruction to this end.

At this point the theory seems to be amenable when focusing on specific sub-classes.
Sims-Yeend \cite{SY11} have provided a characterization of the Cuntz-Nica-Pims\-ner algebra in terms of $A$ when $X$ is a \emph{regular product system}.
The covariant representations just need to be Katsura-covariant fiber-wise in this case.
Likewise a description of the covariant representations is achieved for Cuntz-Nica-Pimsner algebras of C*-dynamical systems in \cite{DFK14} without a reference to the augmented Fock space.
In both \cite{SY11, DFK14} the action of $A$ is by compacts and the question is whether a handy description can be achieved at least at that level. In the current paper we do more than that.

\subsection{Main goal}

We will be denoting the generators of $\bZ_+^N$ by bold face $\bo{1}, \dots, \Bi, \dots, \bo{N}$.
We consider the class of \emph{strong compactly aligned product systems over $\bZ_+^N$}, i.e., $X = \{X_{\un{n}}\}_{\un{n} \in \bZ_+^N}$ is a compactly aligned product system over $A$ such that
\begin{equation}\label{eq:con 1}
\K X_{\un{n}} \otimes \id_{X_\bo{i}} \subseteq \K(X_{\un{n}} \otimes_A X_{\bo{i}}) 
\quad 
\textup{ whenever }
\quad
\un{n} \perp \bo{i},
\end{equation}
(see Definition \ref{D:sca}).
No other conditions are imposed on the left actions of $A$.
This class includes examples where all actions of $A$ are by compacts, e.g. regular product systems, C*-dynamical systems and row-finite graphs, but goes well beyond that point.
For example, the product system of a higher rank graph is strongly compactly aligned if and only if for every edge there are finitely many edges of perpendicular colour with non-void minimal common extension set (Section \ref{S:hrg}).
Mainly we undertake the following task:

\begin{quoting}
\noindent
{\bf Task.} \emph{Describe the Cuntz-Nica-Pimsner relations in terms of ideals in $A$ when condition (\ref{eq:con 1}) is satisfied.}
\end{quoting}

\noindent
Our motivation is to unlock the algebraic aspects of the GIUT for this class and pave the way for subsequent developments, in parallel to the rank one achievements.
We outline a series of questions in Section \ref{S:que} to be pursued in the future.

\subsection{Structure of the paper}

Apart from Section \ref{S:pre} on preliminaries and Section \ref{S:que} on future directions, the paper is split in three parts.
In Sections \ref{S:pol eq}--\ref{S:CNP-rep} we give the description of the CNP-relations.
Sections \ref{S:regular}--\ref{S:hrg} concern connections with the literature.
In Section \ref{S:mfl} we apply our results to a new class of C*-algebras arising from factorial languages.
We proceed to the discussion of our results.
The reader is addressed to Section \ref{Ss:not} for notation on $\bZ_+^N$ that we will be using.

\subsection{Cuntz-Nica-Pimsner-relations}

Henceforth let $X = \{X_\un{n} \mid \un{n} \in \bZ_+^N\}$ be a strong compactly aligned product system over $A$.
The ideals we use are parametrized over subsets $\{1, \dots, N\}$.
For every such $F \subseteq \{1, \dots, N\}$ first set
\begin{equation}
\J_F:= (\bigcap_{i \in F} \ker\phi_{\Bi})^\perp \cap (\bigcap_{\un{n} \leq (1, \dots, 1)} \phi_{\un{n}}(\K X_{\un{n}})^{-1} ),
\end{equation}
and then isolate the maximal $F^\perp$-invariant sub-ideals therein, i.e.,
\begin{equation}
\I_F:= \{a \in \J_F \mid \sca{X_{\un{m}}, a X_{\un{m}}} \in \J_F \foral \un{m} \text{ with } \un{m} \perp F\}.
\end{equation}
To fix notation, if $(\pi,t)$ is a Nica-covariant representation of $X$ then $(\pi, t_{\un{n}})$ is the representation for each component $X_{\un{n}}$ where $t_{\un{0}} := \pi$, and we write $\psi_{\un{n}}$ for the induced representation on $\K X_{\un{n}}$.
We define the \emph{Cuntz-Nica-Pimsner representations} (CNP) through the alternating sums condition
\begin{equation}
\sum \{ (-1)^{|\un{n}|} \psi_{\un{n}}(\phi_{\un{n}}(a)) \mid \un{n} \leq \un{1}_F\} = 0 \foral a \in \I_F,
\end{equation}
where $|\un{n}| = \sum_i n_i$ and $\un{1}_F = \sum_{i \in F} \Bi$.
The Cuntz-Nica-Pimsner algebra $\N\O(X)$ is the universal C*-algebra with respect to the CNP-representations.
It is instructive to compare this with the augmented Fock space description of Sims-Yeend \cite{SY11}, and we do this in Section \ref{Ss:SY}.
Recall that a representation $(\pi,t)$ of $X$ \emph{admits a gauge action} if there is a strongly continuous group action $\beta \colon \bT^N \rightarrow \Aut(\ca(\pi,t))$ such that $\beta_{\un{z}}(t_{\un{n}}(\xi)) = \un{z}^{\un{n}}t_{\un{n}}(\xi)$ for all $\xi \in X_{\un{n}}$ and $\un{n}\in \bZ_+^N$. 
The GIUT then reads as follows.

\vspace{8pt}

\noindent {\bf Theorem \ref{T:GIUT}.} ($\bZ_+^N$-GIUT)
\emph{Let $(\pi,t)$ be a CNP-representation of a strong compactly aligned $\bZ_+^N$-product system $X$.
Then it defines a faithful representation on $\N\O(X)$ if and only if $\pi$ is injective and $(\pi,t)$ admits a gauge action.}

\vspace{8pt}

Our approach is independent from \cite{CLSV11}; in fact this form of the GIUT is required to identify the the Cuntz-Nica-Pimsner algebra here with that of \cite{CLSV11, SY11} through co-universality.
For the proof of Theorem \ref{T:GIUT} we follow a \emph{Gauge-Invariant-Uniqueness-Principle}.
First set
\begin{equation}
\B_{F}^{-} := \spn \{ \psi_{\un{n}}(k_{\un{n}}) \mid k_{\un{n}} \in \K X_{\un{n}}, \un{0} \neq \supp \un{n} \subseteq F \} \qfor \ \mt \neq F\subseteq \{1, \dots, N\},
\end{equation}
and follow the steps:
\begin{enumerate}[labelindent=24pt,labelwidth=\widthof{\ref{last-item}},leftmargin=!]
\item[(i)] Solve the polynomial equations $\pi(a) \in \B_F^{-}$ for injective $\pi$;

\item[(ii)] Define $\I_F$ as the ideal of their solutions;

\item[(iii)] Define $\N\O(X)$ with respect to the $(\pi,t)$ for which $\pi(a) \in \B_{F}^{-}$ when $a \in \I_F$;

\item[(iv)] Establish such a $(\pi,t)$ that is faithful on $A$ and admits a gauge action. \label{last-item}
\end{enumerate}
Then a standard argument using the conditional expectation should yield a GIUT.
This scheme may be transferable to other settings that satisfy the higher rank axioms, without necessarily using the product systems language, e.g., as in \cite{DFK14}.

Step (i) for strong compactly aligned product systems is tackled in Section \ref{S:pol eq}.
The GIUT follows in Section \ref{S:CNP-rep} by a technique introduced in \cite{Kak14}, modulo the faithful embedding $A \hookrightarrow \N\O(X)$. 
For the latter we verify that the CNP-representations of Sims-Yeend \cite{SY11} initially form a subclass of representations of $\N\O(X)$.
Then the GIUT for $\N\O(X)$ gives that the Sims--Yeend CNP-representations of \cite{SY11} actually coincide with the ones herein (see Corollary \ref{C:CNP is CNP}).
We note that we use the augmented Fock representation only for step (iv), and that all of our arguments are independent from \cite{CLSV11}.

Our approach gives the ideal of CNP-relations as an algebraic sum of ideals indexed by the subsets $F$ (see Proposition \ref{P:K ideal}).
As $\N\O(X)$ is the quotient of the Fock space representation by this ideal, we henceforth free the CNP-representations from the augmented Fock space description.
This description triggers a second point of interaction with \cite{CLSV11}.
In this impressive work the authors first explore a C*-algebra $\N\O_X^r$ that is co-universal with respect to the GIUT.
To do so the authors use the theory of Fell bundles of Exel \cite{Exe97} and co-actions, and identify $\N\O_X^r$ with $\N\O(X)$ in many cases, including $\bZ_+^N$.
Here we can avoid the machinery of \cite{CLSV11, Exe97} and instead rely on the algebraic formulation of the CNP-relations to obtain co-universality of $\N\O(X)$ for strong compactly aligned product systems (see Corollary \ref{C:co-un}).

\subsection{Connections with literature}

We already commented on that $\N\O(X)$ coincides with $\N\O_X$ of \cite{CLSV11, SY11}.
As further applications, we use three motivating examples to showcase our approach: regular product systems, higher rank graphs and C*-dynamical systems.
In Section \ref{S:regular} we show that $\N\O(X)$ coincides with $\O_X$ of \cite{Fow02} when $X$ is regular, and that it suffices to check the covariance relations just on the generating fibers.
In Section \ref{S:ds} we present the connection with $\N\O(A,\al)$ of a C*-dynamical system $\al \colon \bZ_+^N \to \End(A)$ studied in \cite{DFK14, Kak15}.
In this case all left actions are by compacts and fall automatically in our setting so that $\N\O(X)$ fits in $\N\O(A,\al)$.
As shown in \cite{DFK14}, if $\al$ is by injective endomorphisms then $\N\O(A,\al)$ coincides with the C*-crossed product of the minimal automorphic extension of $\al$.
In particular, in this case $\N\O(A,\al)$ is $\N\O_X$, and recaptures the surjective $*$-commuting maps considered by Afsar-an Huef-Raeburn \cite{AHR17}.
In Section \ref{S:hrg} we characterize strong compact alignment for product systems arising from higher rank graphs.
To allow comparisons, in Example \ref{E:non-sca} we give a finitely aligned graph that does not fit in this class.
The $F^\perp$-invariance of $\I_F$ is translated to $F$-tracing vertices and we verify that the CNP-representations coincide with the Cuntz-Krieger $\La$-families of \cite{RSY04}.

\subsection{Applications to Factorial Languages}

In Section \ref{S:mfl} we introduce product systems of higher rank \emph{factorial languages} (FL) to which we apply our results.
An FL may be seen as the language of parallel performing automata; so apart from forbidden words within each automaton we may have forbidden entangled operations.
C*-algebras of rank one factorial languages had been introduced by Matsumoto \cite{Mat97}.
Recently the second author with Shalit \cite{KS15} and with Barrett \cite{BK17} studied their Pimsner algebras.
Let us review some of the related results, to allow comparisons.
A rank one factorial language $\La^*$ is a subset of the free semigroup $\bF_+^d$ on $d$ generators such that if a word $\mu \in \La^*$ then every subword of $\mu$ is also in $\La^*$.
The basic example is the language of a subshift.
In \cite{KS15, BK17} it has been shown that there is a dynamical system, namely the \emph{quantized dynamics}, which coincides with the follower set graph when $\La^*$ is sofic.
In addition, the Cuntz-Pimsner algebra $\O(\La^*)$ is either Matsumoto's first quantization $\ca(T)$ or the quotient by the compacts.
This dichotomy depends on whether for every generator $k \in [d]$ there is a word $\mu_k$ such that $\mu_k k \notin \La^*$.

Most of the structural data pass to the higher rank context naturally, along the same lines of \cite{LM95, KS15, BK17}.
For example there is an analogue of a subshift construction and follower set graphs, and soficity amounts to the follower set graph having finitely many vertices. 
As in \cite{KS15}, one can naturally associate a product system and similar algebras to a higher-rank FL.
In Theorem \ref{T: fsg} we prove that $\N\O(\La^*)$ coincides with the Cuntz-Krieger algebra of the ambient unlabeled follower set (higher rank) graph. 
This result is in analogy with that of Carlsen \cite{Car03} and \cite{KS15} in the single variable case.
However, there are several important differences:
\begin{enumerate}[labelindent=10pt,labelwidth=\widthof{\ref{last-item}},leftmargin=!]
\item[(a)] With respect to \cite{Car03} we do not use the Krieger cover on truncated points of a subshift or impose any condition on the subshifts.
Our approach follows \cite{KS15, SS09} where we work directly with the factorial language.

\item[(b)] Instead of cutting with compacts as in \cite{KS15}, here we have to consider the quotient by operators that are ``compactly supported'' on one or more directions.

\item[(c)] The ideals are no longer characterized in terms of forbidden words, so that a dichotomy as in \cite{KS15} fails in general for $\I_F$, unless $F = [N]$ (see Example \ref{E:dich fail} and Proposition \ref{P:IF forb}).
 \label{last-item}
\end{enumerate}
In Corollaries \ref{C:ext 1} and \ref{C:ext 2} we show exactly when $\N\O(\La^*)$ can be one of the extreme cases, i.e., when it coincides with the first quantization or with the quotient by the compactly supported operators.
But it can also be many quotients ``in-between".
In Corollary \ref{C:inbetween} we show that a factorial language is a product of rank one languages if and only if there is a canonical $*$-isomorphism identifying $\N\O(\La^*)$ with the tensor product of Cuntz-Pimsner algebras of its rank one languages.
This scheme can provide examples that are ``in-between'' quotients.

\section{Preliminaries}\label{S:pre}

\subsection{Notation}\label{Ss:not}

We use the notation $[N] := \{1, 2, \dots, N\}$ and $\bZ_+ = \{0, 1, \dots \}$.
The free generators of $\bZ_+^N$ for $N < \infty$ will be denoted by $\bo{1}, \dots, \bo{N}$.
We write
\[
|\un{n}| \equiv |\sum \{ n_i \Bi \mid i \in [N] \}| := \sum_{i \in [N]} n_i
\]
for the \emph{length} of $\un{n}$.
For $\mt \neq F \subseteq [N]:=\{1, \dots, N\}$ we write
\[
\un{1}_F = \sum \{ \Bi \mid i \in F \} \qand \un{1} \equiv \un{1}_{[N]} = (1, \dots, 1).
\]
We consider the lattice structure in $\bZ_+^N$ given by
\[
\un{n} \vee \un{m} := (\max\{n_i, m_i\})_{i=1}^N \qand \un{n} \wedge \un{m} := (\min\{n_i, m_i\})_{i=1}^N.
\]
We denote \emph{the support of $\un{n}$} by $\supp \un{n} := \{i \in [N] \mid n_i \neq 0\}$ and we write
\[
\un{n} \perp \un{m} \qiff \supp \un{n} \bigcap \supp \un{m} = \mt.
\]
Thus $\un{n} \perp F$ means that $\supp \un{n} \bigcap F = \mt$.
We will be making use of the alternating sums, i.e.,
\[
\sum \{(-1)^{|\un{n}|} \mid \un{n} \leq \un{1}_F \} = 0 \foral \mt \neq F \subseteq [N].
\]

\subsection{C*-correspondences}

The reader should be well acquainted with the general theory of Hilbert modules and C*-correspondences.
For example, one may consult \cite{Kat04} and \cite{Lan95} which we follow for terminology.
Here we just wish to fix notation.

A \emph{C*-correspondence} $X$ over $A$ is a right Hilbert module over $A$ with a left action given by a $*$-homomorphism $\phi_X \colon A \to \L X$.
We write $\L X$ and $\K X$ for the adjointable operators and the compact operators of $X$, respectively.
It is accustomed to denote the ``rank one compact operators'' $\zeta \mapsto \xi \sca{\eta, \zeta}$ by $\theta^X_{\xi, \eta}$.
We will write $a\xi$ for $\phi_X(a)\xi$ when it is clear from the context which left action we use.

The C*-correspondence $X$ is called \emph{non-degenerate/injective} if $\phi_X$ is non-degenera\-te/in\-jective.
If $\phi_X$ is injective and $\phi_X(A) \subseteq \K X$ then we say that $X$ is \emph{regular}.
For two C*-corresponden\-ces $X, Y$ over the same $A$ we write $X \otimes_A Y$ for the stabilized tensor product over $A$.
Moreover we say that $X$ is unitarily equivalent to $Y$ if there is a surjective adjointable $U \in \L(X,Y)$ such that $\sca{U \xi, U \eta} = \sca{\xi, \eta}$ and $U (a \xi b) = a U(\xi) b$ for all $\xi, \eta \in X$ and $a,b \in A$; in this case we write $X \simeq Y$.

There are two key results for C*-correspondences that we will be using throughout the paper.
First for every $\xi \in X$ there exists an $\eta \in X$ such that
\begin{equation}\label{eq:factorization}
\xi = \eta a \qfor a = \sca{\xi, \xi}^{1/2}.
\end{equation}
For a reference see \cite[Lemma 4.4]{Lan95}.
Secondly for an ideal $I \subseteq A$ and a $k \in \K X$ we have the equivalence:
\begin{equation}\label{eq:compacts ideal}
\sca{\xi, k \eta} \in I \foral \xi, \eta \in X \; \text{ if and only if } \; k \in \K(X I).
\end{equation}
For a reference see \cite[Lemma 2.6]{FMR03} or \cite[Lemma 1.6]{Kat04}.

A \emph{representation} $(\pi,t)$ of a C*-correspondence is a left module map that preserves the inner product.
Then $(\pi,t)$ is automatically a bimodule map.
Moreover there exists a $*$-homomorphism $\psi$ on $\K X$ such that
\[
\psi(\theta^X_{\xi, \eta}) = t(\xi) t(\eta)^* \foral \theta^X_{\xi, \eta} \in \K X.
\]
If $\pi$ is injective then so is $\psi$. 
A representation $(\pi,t)$ is said to admit a gauge action $\be \colon \bT \to \Aut(\ca(\pi,t))$ if $\{\be_z\}_{z \in \bT}$ is pointwise continuous and
\[
\be_z(\pi(a)) = \pi(a) \foral a \in A
\qand
\be_z(t(\xi)) = z t(\xi) \foral \xi \in X.
\]

The \emph{Toeplitz-Pimsner algebra $\T_X$} is the universal C*-algebra with respect to the representations of $X$.
The \emph{Cuntz-Pimsner algebra $\O_X$} is the universal C*-algebra with respect to the \emph{covariant} representations of $X$. More precisely, we say that $(\pi,t)$ is covariant if
\[
\pi(a) = \psi (\phi_X(a)) \foral a \in J_X,
\]
where the ideal $J_X:=\ker \phi_X \cap \phi^{-1}(\K X)$ is the largest ideal on which the restriction of $\phi_X$ is injective with image into the compacts.

The Gauge Invariant Uniqueness Theorem (GIUT) was initiated by an Huef and Raeburn for Cuntz-Krieger algebras \cite{HR97}.
Various generalizations were given by Doplicher, Pinzari and Zuccante \cite{DPZ98}, Fowler, Muhly and Raeburn \cite{FMR03}, and Fowler and Raeburn \cite{FR99}. 
A fully general version for the GIUT of $\O_X$ was eventually proven by Katsura \cite{Kat04}.
In fact, in \cite{Kak14} it is shown that $\O_X$ is co-universal in the sense that any gauge invariant quotient of $\T_X$ surjects to $\O_X$ as long as it is injective on the coefficients.

\begin{theorem}[$\bZ_+$-GIUT]
Let $X$ be a C*-correspon\-dence over $A$. Then a pair $(\pi,t)$ defines a faithful representation of the Cuntz-Pimsner algebra $\O_X$ if and only if $(\pi,t)$ admits a gauge action and $\pi$ is injective.
\end{theorem}

\subsection{Product systems}

Fix a set $\{X_{\Bi} \mid i \in [N]\}$ of C*-correspondences over $A$, one for each generator of $\bZ_+^N$.
A \emph{product system $X$} is a family $\{X_{\un{n}} : \un{n} \in \bZ_+^N\}$ of C*-correspondences over $A$ such that 
\[
X_{\un{0}} = A \qand 
X_{\bo{i}_1}^{\otimes n_1} \otimes_A \cdots \otimes_A X_{\bo{i}_k}^{\otimes n_k} \simeq X_{\un{n}} \text{ whenever } \un{n} = \sum n_j \bo{i}_j \text{ and } n_1 \neq 0.
\]
We require $n_1 \neq 0$ so that these equivalences do not force non-degeneracy of the fibers.
Consequently $X$ comes with a family of associative rules in the form of unitary equivalences
\[
u_{\un{n}, \un{m}} \colon X_{\un{n}} \otimes_A X_{\un{m}} \to X_{\un{n} + \un{m}}.
\]
We will suppress the use of the $u_{\un{n}, \un{m}}$ as much as possible by writing $\xi_{\un{n}} \xi_{\un{m}} \in X_{\un{n} + \un{m}}$ for the element $u_{\un{n}, \un{m}}(\xi_{\un{n}} \otimes \xi_{\un{m}})$.
Along with the system we have some canonical operations that respect these equivalences.
To this end we define the maps
\[
i_{\un{n}}^{\un{n} + \un{m}} \colon \L X_{\un{n}} \to \L X_{\un{n} + \un{m}}
\; \textup{ such that } \;
i_{\un{n}}^{\un{n} + \un{m}}(S) = u_{\un{n}, \un{m}}(S \otimes \id_{X_{\un{m}}}) u^{*}_{\un{n}, \un{m}}.
\]
It is clear that $i^{\un{n} + \un{m} + \un{x}}_{\un{n} + \un{m}} \, i^{\un{n} + \un{m}}_{\un{n}} = i^{\un{n} + \un{m} + \un{x}}_{\un{n}}$ and thus $i^{\un{n} + \un{m}}_{\un{n}}(\phi_{\un{n}}(a)) = \phi_{\un{n} + \un{m}}(a)$.
For convenience we shall write
\[
S \vee T := i_{\un{n}}^{\un{n} \vee \un{m}}(S) i_{\un{m}}^{\un{n} \vee \un{m}}(T) \qfor S \in \L X_{\un{n}}, T \in \L X_{\un{m}}.
\]
Following Fowler's work \cite{Fow02}, a product system is called \emph{compactly aligned} if it has the property:
\[
S \vee T \in \K X_{\un{n} \vee \un{m}} \text{ whenever } S \in \K X_{\un{n}}, T \in \K X_{\un{m}}.
\]

\subsection{Strong compactly aligned product systems} \label{Ss:sca}

In the current paper we consider product systems that satisfy an additional property.

\begin{definition}\label{D:sca}
A product system $X$ is called \emph{strong compactly aligned} if it is compactly aligned and $i_{\un{n}}^{\un{n} + \Bi}(\K X_{\un{n}}) \subseteq \K X_{\un{n} + \Bi}$ whenever $\un{n} \perp \Bi$.
\end{definition}

\begin{example}
If $\phi_\Bi(A) \subseteq \K X_{\Bi}$ for all $i \in [N]$, then $k \otimes \id_{X_{\bo{j}}} \in \K X_{\Bi + \bo{j}}$ for all $k \in \K X_{\Bi}$ and $j \in [N]$ by \cite[Proposition 4.7]{Lan95}.
Inductively, all actions $\phi_{\un{n}}$ are by compacts and \cite[Proposition 4.7]{Lan95} yields that the product system is strong compactly aligned.
We will later establish through Proposition \ref{P:sca-graph} that many higher rank graphs can be constructed so that their associated product system is strong compactly aligned without satisfying $\phi_\Bi(A) \subseteq \K X_{\Bi}$ for all $i \in [N]$. 
However there are limitations and Example \ref{E:non-sca} refers to a compactly aligned product system that is \emph{not} strong compactly aligned.
\end{example}

Every $\K X_\Bi$ admits a contractive approximate identity (c.a.i.) $(k_{\Bi, \la_i})_{\la_i \in \La_\Bi}$.
We consider the directed set $\La_{\bo{1}} \times \cdots \times \La_{\bo{N}}$ with the product order and set $k_{\Bi, \un{\la}} = k_{\Bi, \la_i}$.
Then $(k_{\Bi, \un{\la}})$ is a subnet of $(k_{\Bi, \la_i})$.
Henceforth we omit the vector form for the indices and fix the systems of nets
\[
(k_{\Bi, \la})_{\la \in \La} \text{ over the directed set } \La = \La_{\bo{1}} \times \cdots \times \La_{\bo{N}}.
\]
Fix a subset $F\subseteq [N]$. 
By applying on elementary tensors we can verify that if $i \in F$ then $i^{\un{1}_F}_{\Bi}(k_{\Bi, \la})$ converges to the identity of $\L X_{\un{1}_F}$ in the strict topology.
Therefore so does the net
\begin{equation}\label{eq:efl}
e_{F, \la}: = \prod \{ i_{\Bi}^{\un{1}_F}(k_{\Bi, \la}) \mid i \in F \}.
\end{equation}
Consequently, if $\un{n} \in \bZ_+^N$ with $\supp \un{n} \supseteq F$ then we have
\begin{equation}\label{eq:com F}
\nor{\cdot}\text{-}\lim_\la i_{\un{1}_F}^{\un{n}}(e_{F, \la}) \cdot k_{\un{n}} = k_{\un{n}}  \foral k_\un{n} \in \K X_{\un{n}}.
\end{equation}
The product defining $e_{F,\la}$ is taken in the order inherited by $1 < 2 < \cdots < N$.
However by considering different decompositions for the elementary tensors, it follows that equation (\ref{eq:com F}) holds for any order.

\begin{proposition}\label{P: sca ai}
Let $X$ be a strong compactly aligned product system.
Fix $F \subseteq [N]$ and $\un{0} \neq \un{n} \in \bZ_+^N$, and set $\un{m} = \un{n} \vee \un{1}_F$.
Then for the $e_{F,\la}$ of equation (\ref{eq:efl}) we have that
\[
\nor{\cdot}\text{-}\lim_\la i_{\un{1}_F}^{\un{m}}(e_{F, \la}) \cdot i_{\un{n}}^{\un{m}}(k_{\un{n}}) 
= 
i_{\un{n}}^{\un{m}}(k_{\un{n}})  \foral k_\un{n} \in \K X_{\un{n}}.
\]
In particular, it follows that $i_{\un{n}}^{\un{m}}(k_{\un{n}}) \in \K X_{\un{m}}$ for all $k_\un{n} \in \K X_{\un{n}}$.
\end{proposition}

\begin{proof}
Without loss of generality assume that $F = \{1, \dots, q\}$ and set $\un{x} = \bo{1} + \cdots + \bo{q}$.
Then
\[
e_{F, \la} = i_{\bo{1}}^{\un{x}}(k_{\bo{1}, \la}) \cdots i_{\bo{q}}^{\un{x}}(k_{\bo{q}, \la})
\]
for the approximate identities $(k_{\bo{i}, \la})$ in  $\K X_{\bo{i}}$.
Suppose that $\supp \un{n} = \{r, \dots, s\}$ and fix $k_\un{n} \in \K X_{\un{n}}$.
If $q \in \supp \un{n}$ then the comments preceding the statement show that 
\[
\nor{\cdot}\text{-}\lim_\la i_{\bo{q}}^{\un{n}}(k_{\bo{q}, \la}) k_{\un{n}} = k_{\un{n}}
\]
and therefore applying $i_{\un{n}}^{\un{m}}$ for $\un{m} = \un{n} \vee \un{1}_F$ yields
\[
\nor{\cdot}\text{-}\lim_\la i_{\bo{q}}^{\un{m}}(k_{\bo{q}, \la}) i_{\un{n}}^{\un{m}}(k_{\un{n}}) = i_{\un{n}}^{\un{m}}(k_{\un{n}}).
\]
If $q \notin \supp \un{n}$ then $i_{\un{n}}^{\un{n} + \bo{q}}(k_{\un{n}})$ is in $\K X_{\un{n} + \bo{q}}$ as $X$ is strong compactly aligned.
Thus by the comments preceding the statement we have that
\[
\nor{\cdot}\text{-}\lim_\la i_{\bo{q}}^{\un{n} + \bo{q}}(k_{\bo{q}, \la}) i_{\un{n}}^{\un{n} + \bo{q}}(k_{\un{n}}) = i_{\un{n}}^{\un{n} + \bo{q}}(k_{\un{n}})
\]
and therefore applying $i_{\un{n} + \bo{q}}^{\un{m}}$ yields
\[
\nor{\cdot}\text{-}\lim_\la i_{\bo{q}}^{\un{m}}(k_{\bo{q}, \la}) i_{\un{n}}^{\un{m}}(k_{\un{n}}) = i_{\un{n}}^{\un{m}}(k_{\un{n}}).
\]
Proceeding inductively we conclude that
\begin{align*}
\nor{\cdot}\text{-}\lim_\la i_{\un{1}_F}^{\un{m}}(e_{F, \la}) i_{\un{n}}^{\un{m}}(k_{\un{n}})
& =
\nor{\cdot}\text{-}\lim_\la i_{\bo{1}}^{\un{m}}(k_{\bo{1}, \la}) \cdots i_{\bo{q}}^{\un{m}}(k_{\bo{q}, \la}) i_{\un{n}}^{\un{m}}(k_{\un{n}}) 
= 
i_{\un{n}}^{\un{m}}(k_{\un{n}})
\end{align*}
which completes the proof of the first part.
For the second part we have that
\[
\un{m} = \un{n} + \sum \{ \Bi \mid 1 \leq i < \min\{q, r\} \}.
\]
As $X$ is strong compactly aligned, we get
\[
i_{\un{n}}^{\un{n} + \Bi}(k_{\un{n}}) \in \K X_{\un{n} + \Bi} \foral 1 \leq i < \min\{q, r\},
\]
and inductively it follows that $i_{\un{n}}^{\un{m}}(k_{\un{n}}) \in \K X_{\un{m}}$.
\end{proof}

\subsection{Representations} \label{Ss:CNP-relations}

Let us set the terminology for Nica-Pimsner algebras of product systems, and define the ideals that give rise to our Cuntz-Nica-Pimsner relations.

\begin{definition}
A \emph{Nica-covariant representation $(\pi, t)$} of a product system $X = \{X_{\un{n}} \mid \un{n} \in \bZ_+^N \}$ consists of a family of representations $(\pi, t_{\un{n}})$ of $X_{\un{n}}$ that satisfy the \emph{associative rule}:
\[
t_{\un{n} + \un{m}}(\xi_{\un{n}} \xi_{\un{m}}) = t_{\un{n}}(\xi_{\un{n}}) t_{\un{m}}(\xi_{\un{m}})
\]
and the \emph{Nica-covariance}:
\[
\psi_{\un{n}}(S) \psi_{\un{m}}(T) = \psi_{\un{n} \vee \un{m}} (S \vee T) \text{ whenever } S \in \K X_{\un{n}}, T \in \K X_{\un{m}}.
\]
The \emph{Toeplitz-Nica-Pimsner} algebra $\N\T(X)$ is the universal C*-algebra generated by $A$ and $X$ with respect to the representations of $X$.
We write $\pi \times t$ for the induced representation of a Nica-covariant pair $(\pi,t)$.
\end{definition}

\begin{remark}
The Fock space provides an essential example of a Nica-covariant representation.
In short let $\F(X) = \sumoplus \{ X_{\un{m}} \mid \un{m} \in \bZ_+^N\}$.
For $a \in A$ and $\xi_{\un{n}} \in X_{\un{n}}$ define
\[
\si(a) \xi_{\un{m}} = \phi_{\un{m}}(a) \xi_{\un{m}}
\qand
s(\xi_{\un{n}}) \xi_{\un{m}} = \xi_{\un{n}} \xi_{\un{m}}
\]
for all $\xi_{\un{m}} \in X_{\un{m}}$.
Then $(\si, s)$ is Nica-covariant and it is called the \emph{Fock representation} of $X$ \cite{Fow02}.
By taking the compression at the $(\un{0}, \un{0})$-entry we see that $\si$, and thus each $s_{\un{n}}$, is injective.
\end{remark}

Now we want to explore a specific subclass of representations.
For a finite $\mt \neq F \subseteq [N]$ we form the ideal 
\begin{align*}
\J_F
& =
(\bigcap_{i \in F} \ker\phi_{\Bi})^\perp \cap (\bigcap\{ \phi_{\un{n}}^{-1}( \K X_{\un{n}}) \mid \un{n} \leq \un{1} \} )
\end{align*}
with the understanding that $\phi_{\un{0}} = \id_A$.
In particular when $X$ is strong compactly aligned, we have that
\[
\bigcap\{ \phi_{\un{n}}^{-1}( \K X_{\un{n}}) \mid \un{n} \leq \un{1} \} 
= 
\bigcap\{ \phi_{\Bi}^{-1}( \K X_{\Bi}) \mid i \in [N] \}. 
\]
Furthermore we define the ideal
\[
\I_F:= \{a \in \J_F \mid \sca{X_{\un{m}}, a X_{\un{m}}} \subseteq \J_F \foral \un{m} \perp F\}.
\]
The next proposition shows that $\I_F$ is the biggest ideal in $\J_F$ that remains invariant under the ``action'' of $F^\perp$.

\begin{proposition}\label{P:rem inv IF}
Let $X$ be a product system.
If $a \in \I_F$ then $\sca{X_{\un{m}}, a X_{\un{m}}} \subseteq \I_F$ for all $\un{m} \perp F$.
\end{proposition}

\begin{proof}
Let $b = \sca{\xi_{\un{m}}, a \eta_{\un{m}}}$ for $a \in \I_F$ and $\un{m} \perp F$.
Then $b \in \J_F$ as $a \in \I_F$.
Now let $\un{n} \perp F$.
Then for $\xi_{\un{n}}, \eta_{\un{n}} \in X_{\un{n}}$ we compute
\begin{align*}
\sca{\xi_{\un{n}}, b \eta_{\un{n}}}
& =
\sca{\xi_{\un{n}}, \sca{\xi_{\un{m}}, a \eta_{\un{m}}} \eta_{\un{n}}}
=
\sca{u_{\un{n}, \un{m}}( \xi_{\un{n}} \otimes \xi_{\un{m}}), u_{\un{n}, \un{m}}(a \eta_{\un{n}} \otimes \eta_{\un{m}})}
= 
\sca{\xi_{\un{m}} \xi_{\un{n}}, a \eta_{\un{m}} \eta_{\un{n}}}.
\end{align*}
Since $\un{m} +\un{n} \perp F$ and $a \in \I_F$ we also get that $\sca{\xi_{\un{n}}, b \eta_{\un{n}}} \in \J_F$, which completes the proof.
\end{proof}

\begin{definition}
Let $X$ be a strong compactly aligned product system.
A representation $(\pi, t)$ of $X$ is called \emph{Cuntz-Nica-Pimsner} (or a \emph{CNP-repre\-sentation}) if it satisfies
\[
\sum \{ (-1)^{|\un{n}|} \psi_{\un{n}}(\phi_{\un{n}}(a)) \mid \un{n} \leq \un{1}_F \} = 0 \foral a \in \I_F,
\]
where $\psi_{\un{0}}(\phi_{\un{0}}(a)) = \pi(a)$.
The \emph{Cuntz-Nica-Pimsner algebra} $\N\O(X)$ is the universal C*-algebra with respect to the CNP-representations.
\end{definition}

In the following sections we will make precise that the CNP-representations above coincide with the ones introduced by Sims-Yeend \cite{SY11} when $X$ is strong compactly aligned.
As a consequence the universal Cuntz-Nica-Pimsner algebra will coincide with that considered in \cite{CLSV11, SY11}.

\begin{examples}
We use the following three motivating examples of strong compactly aligned product systems:
\begin{enumerate}[labelindent=18pt,labelwidth=\widthof{\ref{last-item}},leftmargin=!]
\item If every $X_{\Bi}$ is regular then each $X_{\un{n}}$ is so.
In this case we have that $\J_F = A$ and therefore trivially $\I_F = A$.
In Section \ref{S:regular} we will show that the CNP-representations are in bijection with $(\pi,t)$ such that each $(\pi, t_{\un{n}})$ is covariant for $X_{\un{n}}$ in the sense of Katsura.
Thus $\N\O(X)$ coincides with $\O_X$ defined in \cite{Fow02}.
\item Regular product systems encode a number of constructions such as the higher rank row-finite graphs without sources of \cite{RSY03}.
In Section \ref{S:hrg} we give a full description of the higher rank graphs that produce strong compactly aligned product systems (that may or may not have sources).
Moreover we show in this case what is the form of the $\J_F$ and the $\I_F$.
\item In Section \ref{S:ds} we show how we can associate a product system to a semigroup action $\al \colon \bZ_+^N \to \End(A)$.
The ideals $\I_F$ are then given by
\[
\I_F = \bigcap_{\un{n} \perp F} \al_{\un{n}}^{-1} \Big(\big( \bigcap_{i \in F} \ker\al_\Bi \big)^\perp\Big) . 
\]
Nica-Pimsner algebras related to this object were considered in \cite[Section 4]{DFK14} and we will show later how they connect with the work herein.
 \label{last-item}
\end{enumerate}
\end{examples}

\subsection{The cores of the representations}

Given a Nica-covariant representation $(\pi, t)$ and $\un{m}, \un{m}' \in \bZ_+^N$ we write
\[
\B_{[\un{m}, \un{m} + \un{m}']} := \spn\{ \psi_{\un{n}}(k_{\un{n}}) \mid k_\un{n} \in \K X_{\un{n}}, \un{m} \leq \un{n} \leq \un{m} + \un{m}'\}
\]
and
\[
\B_{(\un{m}, \un{m} + \un{m}']} := \spn\{ \psi_{\un{n}}(k_{\un{n}}) \mid k_\un{n} \in \K X_{\un{n}}, \un{m} < \un{n} \leq \un{m} + \un{m}'\}.
\]
It is not hard to see that these $*$-algebras are closed in $\ca(\pi,t)$, e.g. \cite[Lemma 36]{CLSV11}.
Moreover we write
\[
\B_{[\un{m}, \infty]} := \ol{\spn} \{ \psi_{\un{n}}(k_{\un{n}}) \mid k_\un{n} \in \K X_{\un{n}}, \un{m} \leq \un{n} \}
\]
and
\[
\B_{(\un{m}, \infty]} := \ol{\spn} \{ \psi_{\un{n}}(k_{\un{n}}) \mid k_\un{n} \in \K X_{\un{n}}, \un{m} < \un{n} \}.
\]
We refer to these sets as the \emph{cores} of the representation $(\pi,t)$.
It follows from the work of Fowler \cite[Proposition 5.10]{Fow02} that if $(\pi,t)$ is a Nica-covariant representation of a compactly aligned product system $X$ then
\begin{equation}\label{eq:Fow}
t_{\un{m}}(X_{\un{m}})^* t_{\un{n}}(X_{\un{n}}) \subseteq \ol{t_{\un{n}'}(X_{\un{n}'}) t_{\un{m}'}(X_{\un{m}'})^*}
\end{equation}
for $\un{n}' = -\un{n} + \un{n} \vee \un{m}$ and $\un{m}' = - \un{m} + \un{n} \vee \un{m}$.
This imposes that the cores are stable under multiplying by elements on orthogonal support, i.e.,
\begin{equation}\label{eq:stable}
t_{\un{n}}(X_{\un{n}})^* \cdot \B_{[\un{m}, \un{m} + \un{m}']} \cdot t_{\un{n}}(X_{\un{n}}) \subseteq \B_{[\un{m}, \un{m} + \un{m}']} 
\foral \un{n} \perp \un{m} + \un{m}'.
\end{equation}
A Nica-covariant representation $(\pi,t)$ \emph{admits a gauge action} if there is a point-norm continuous family of $*$-automorphisms $\{\be_{\un{z}}\}_{\un{z} \in \bT^N}$ such that
\[
\be_{\un{z}}(t_\un{n}(\xi_{\un{n}})) = \un{z}^{\un{n}} \, t_{\un{n}}(\xi_{\un{n}}) 
\foral \xi_{\un{n}} \in X_{\un{n}}
\qand 
\be_{\un{z}}(\pi(a)) = \pi(a) \foral a \in A.
\]
In this case $\B_{[\un{0}, \infty]}$ is the fixed point algebra of $\ca(\pi,t)$.
By universality both $\N\T(X)$ and $\N\O(X)$ admit a gauge action.

\section{Solving polynomial equations}\label{S:pol eq}

In this section we show how the ideals $\I_F$ arise as solutions of polynomial equations.
Given a Nica-covariant representation $(\pi,t)$ of $X$, we use the system of c.a.i. $(k_{\Bi,\la})$ for $\K X_{\Bi}$ to define the elements $p_{\Bi, \la}$ and the projections $p_{\Bi}$ by
\begin{equation}\label{eq:p'i}
p_{\Bi, \la}: = \psi_{\Bi}(k_{\Bi,\la})
\qand
p_\Bi:= \textup{w*-}\lim_\la p_{\Bi, \la}.
\end{equation}
The next remark is an immediate consequence of Proposition \ref{P: sca ai}.

\begin{proposition}\label{P:prod cai}
Let $X$ be a strong compactly aligned product system over $A$.
Suppose that $(\pi,t)$ is a Nica-covariant representation of $X$.
Let $p_{\Bi, \la}$ and $p_{\Bi}$ as defined in equation (\ref{eq:p'i}) and fix $\mt \neq F \subseteq [N]$.
Then 
\[
\nor{\cdot}\text{-}\lim_\la  \psi_{\un{n}}(k_{\un{n}}) \prod_{i \in F} p_{\Bi, \la}
=
\psi_{\un{n}}(k_{\un{n}}) \prod_{i \in F} p_{\Bi}
\foral k_{\un{n}} \in \K X_{\un{n}}.
\]
\end{proposition}

\begin{proof}
Recall the definition of the $e_{F,\la}$ from equation (\ref{eq:efl}).
By Nica covariance we have that $\psi_{\un{n}}(k_{\un{n}}) \prod_{i \in F} p_{\Bi, \la} = \psi_{\un{m}}(k_{\un{n}} \vee e_{F, \la})$ for $\un{m} = \un{n} \vee \un{1}_F$.
But Proposition \ref{P: sca ai} yields that
\[
\nor{\cdot}\text{-}\lim_\la k_\un{n} \vee e_{F, \la} 
= 
\nor{\cdot}\text{-}\lim_\la i_{\un{n}}^{\un{m}}(k_{\un{n}}) \cdot i_{\un{1}_F}^{\un{m}}(e_{F, \la})
=
i_{\un{n}}^{\un{m}}(k_{\un{n}}).
\]
Therefore the net $(\psi_{\un{n}}(k_{\un{n}}) \prod_{i \in F} p_{\Bi, \la})_\la$ converges in norm to $\psi_{\un{m}}(i_{\un{n}}^{\un{m}}(k_{\un{n}}))$.
This means that its w*-limit $\psi_{\un{n}}(k_{\un{n}}) \prod_{i \in F} p_{\Bi}$ is also the norm limit.
\end{proof}

From now on we fix a Nica-covariant representation $(\pi,t)$ so that $\pi$ is injective.
In order to obtain the GIUT we will have to solve equations of the form
\begin{equation}\label{eq:in com}
\pi(a) \in \B_{(\un{0}, m \cdot \un{1}_F]} \text{ for } F \subseteq [N], m \in \bZ_+.
\end{equation}
Due to the structure of the cores, an element $\pi(a)$ satisfies equation (\ref{eq:in com}) if and only if there are $k_{\un{n}} \in \K X_{\un{n}}$ such that
\begin{equation}\label{eq:in}
\pi(a) + \sum \{ \psi_{\un{n}}(k_{\un{n}}) \mid \un{0} \neq \un{n} \leq m \cdot \un{1}_F \} = 0.
\end{equation}
Our strategy is to show that equation (\ref{eq:in}) implies both that
\begin{equation}\label{eq:out com}
\phi_{\un{n}}(a) \in \K X_{\un{n}} \foral \un{n} \leq \un{1},
\end{equation}
and that 
\begin{equation}\label{eq:out prod}
\pi(a) \prod_{i \in F} (I - p_{\Bi}) = 0,
\end{equation}
under the assumption that $X$ is strong compactly aligned.
Notice here that when equation (\ref{eq:out com}) is satisfied and $[r, s]$ is any interval inside $F$ then we get
\[
\pi(a) p_{\bo{r}} \cdots p_{\bo{s}} = \psi_{\un{n}}(\phi_{\un{n}}(a)) \text{ for } \un{n} = \bo{r} + \cdots + \bo{s}.
\]
Hence equation (\ref{eq:out prod}) is rewritten as
\begin{equation}
\pi(a) + \sum \{ (-1)^{|\un{n}|} \psi_{\un{n}}(\phi_{\un{n}}(a)) \mid \un{0} \neq \un{n} \leq \un{1}_F \} = 0
\end{equation}
which gives back (\ref{eq:in}).
As the latter is independent of the order we multiplied, the product in equation (\ref{eq:out prod}) is the same for any order within $F$. 
This line of reasoning gives that equations (\ref{eq:out com}) and (\ref{eq:out prod}), together, imply equation (\ref{eq:in com}).

\begin{proposition}\label{P:out com}
Let $X$ be a strong compactly aligned product system over $A$.
Suppose that $(\pi,t)$ is a Nica-covariant representation of $X$ such that $\pi$ is injective.
If $\pi(a)$ satisfies
\begin{equation}
\pi(a) + \sum \{ \psi_{\un{n}}(k_{\un{n}}) \mid \un{0} \neq \un{n} \leq \un{m} \}= 0 \tag{\ref{eq:in}}
\end{equation}
for some $\un{m} \in \bZ_+^N$ then it satisfies
\begin{equation}
\phi_{\un{n}}(a) \in \K X_{\un{n}} \; \foral \, \un{0} \neq \un{n} \leq \un{1}.  \tag{\ref{eq:out com}}
\end{equation}
\end{proposition}

\begin{proof}
By introducing zeros, without loss of generality we may assume that $\un{m} = m \cdot \un{1} = (m, \dots, m)$.
It suffices to show that $\phi_{\Bi}(a) \in \K X_{\Bi}$ for all $i \in [N]$.
By Proposition \ref{P:prod cai} we derive
\begin{align*}
\nor{\cdot}\text{-}\lim_\la \pi(a) p_{\Bi, \la}
& = 
- \sum \{ \nor{\cdot}\text{-}\lim_\la \psi_{\un{n}}(k_{\un{n}}) p_{\Bi, \la} \mid \un{0} \neq \un{n} \leq m \cdot \un{1} \} \\
& = 
- \sum \{ \psi_{\un{n}}(k_{\un{n}}) p_{\Bi} \mid \un{0} \neq \un{n} \leq m \cdot \un{1} \}.
\end{align*}
Thus the w*-limit $\pi(a) p_{\Bi}$ of the net $(\pi(a) p_{\Bi, \la})_\la$ coincides with its norm-limit.
However the net $(\pi(a) p_{\Bi, \la})_\la$ is in $\psi_\Bi(\K X_\Bi)$ and thus $\pi(a) p_{\Bi} \in \psi_{\Bi}(\K X_{\Bi})$.
Let $k_\Bi \in \K X_\Bi$ such that $\pi(a) p_{\Bi} = \psi_\Bi(k_\Bi)$.
Then for every $\xi_\Bi \in \K X_\Bi$ we have
\begin{align*}
t_{\Bi}(\phi_{\Bi}(a) \xi_{\Bi})
=
\pi(a) t_{\Bi}(\xi_{\Bi})
=
\pi(a) (p_{\Bi} t_{\Bi} (\xi_{\Bi}))
=
\psi_{\Bi}(k_{\Bi}) t(\xi_{\Bi})
=
t_{\Bi} ( k_{\Bi} \xi_{\Bi}).
\end{align*}
As $\pi$, and thus $t_{\Bi}$, are isometric we have that $\phi_{\Bi}(a) = k_\Bi$. 
The discussion preceding Proposition \ref{P:rem inv IF} finishes the proof.
\end{proof}

\begin{proposition}\label{P:out prod}
Let $X$ be a strong compactly aligned product system over $A$.
Suppose that $(\pi,t)$ is a Nica-covariant representation of $X$.
If $\pi(a)$ satisfies
\begin{equation}
\pi(a) + \sum \{ \psi_{\un{n}}(k_{\un{n}}) \mid \un{0} \neq \un{n} \leq \un{m} \}= 0 \tag{\ref{eq:in}}
\end{equation}
for some $\un{m} \in \bZ_+^N$, then it satisfies
\begin{equation}
\pi(a) \prod_{i \in F} (I - p_{\Bi}) = 0 \tag{\ref{eq:out prod}} \, \text{ for } \, F = \supp \un{m}.
\end{equation}
\end{proposition}

\begin{proof}
Without loss of generality assume that $F = \{1,...,r\}$.
It suffices to show that
\begin{equation}\label{eq:for proof}
\psi_{\un{n}}(k_{\un{n}}) \prod_{i =1}^r (I - p_{\Bi}) = 0
\end{equation}
for every $\un{0} \neq \un{n} \leq \un{m}$.
Once this is shown then we can directly verify that
\begin{align*}
\pi(a) \prod_{i =1}^r (I - p_{\Bi})
& =
- \sum \{ \psi_{\un{n}}(k_{\un{n}}) \prod_{i =1}^r (I - p_{\Bi}) \mid \un{0} \neq \un{n} \leq \un{m} \}= 0.
\end{align*}
In order to show equation (\ref{eq:for proof}), set
\[
p := \prod_{i =1}^r (I - p_{\Bi})
\]
and fix a non-zero $\un{n} \leq \un{m}$.
By Proposition \ref{P: sca ai}, if $1 \in \supp \un{n}$ then $\psi_{\un{n}}(k_{\un{n}}) (I - p_{\bo{1}}) = 0$ and thus $\psi_{\un{n}}(k_{\un{n}}) p =0$.
Otherwise
\[
\psi_{\un{n}}(k_{\un{n}})(I - p_{\bo{1}}) = \psi_{\un{n}}(k_{\un{n}}) - \psi_{\un{n}}(k_{\un{n}})p_{\bo{1}}
\]
and the second summand is in $\psi_{\un{n} + \bo{1}}(\K X_{\un{n} + \bo{1}})$ by Proposition \ref{P:prod cai}.
Likewise in the second step, if $2 \in \supp \un{n}$ then $2 \in \supp (\un{n} + \bo{1})$.
Hence
\[
\psi_{\un{n}}(k_{\un{n}})(I-p_{\bo{2}}) = 0 \qand \psi_{\un{n}}(k_{\un{n}})p_{\bo{1}}(I - p_{\bo{2}}) = 0
\]
and so $\psi_{\un{n}}(k_{\un{n}}) p = 0$.
Otherwise we move on to consider
\begin{align*}
\psi_{\un{n}}(k_{\un{n}})(I - p_{\bo{1}})(I- p_{\bo{2}})(I - p_{\bo{3}})
& =
\psi_{\un{n}}(k_{\un{n}})(I - p_{\bo{3}})  + \psi_{\un{n}}(k_{\un{n}})p_{\bo{1}}p_{\bo{2}}(I - p_{\bo{3}}) -\\
& \hspace{2cm} 
- \psi_{\un{n}}(k_{\un{n}})p_{\bo{1}}(I - p_{\bo{3}})
- \psi_{\un{n}}(k_{\un{n}})p_{\bo{1}}(I - p_{\bo{3}}).
\end{align*}
Eventually there will be an $s \in \supp \un{n} \cap F$ that must also be in the support of every $\un{n} + \sum_{i=1}^l \Bi$ for $l \leq s$, giving
\[
\psi_{\un{n}}(k_{\un{n}})(I - p_{\bo{1}}) \cdots (I - p_{\bo{s}}) = 0
\]
which completes the proof of the claim.
\end{proof}

\begin{proposition}\label{P:in IF}
Let $X$ be a strong compactly aligned product system over $A$.
Suppose that $(\pi,t)$ is a Nica-covariant representation of $X$ such that $\pi$ is injective.
If $\pi(a) \in \B_{(\un{0}, \un{m}]}$ then $a \in \I_F$ for $F = \supp \un{m}$.
\end{proposition}

\begin{proof}
By adding zeros we may assume that $\un{m} = m \cdot \un{1}_F$ for 
\[
m := \max\{|m_i| \mid i \in F\}.
\]
By Propositions \ref{P:out com} and \ref{P:out prod} we have that $\phi_{\un{n}}(a) \in \K X_{\un{n}}$ for all $\un{n} \leq \un{1}$ and
\[
\pi(a) \prod_{i \in F} (I - p_\Bi) = 0.
\]
Moreover notice that $\pi(b) p_{\Bi} = 0$ whenever $b \in \ker \phi_{\Bi}$, so that $\pi(b)(I - p_{\Bi}) = \pi(b)$. 
Indeed, this is because
\[
\pi(b) p_{\Bi} = \text{w*-}\lim_\la \pi(b) \psi_{\Bi}(k_{\Bi, \la}) = \text{w*-}\lim_\la \psi_{\Bi}(\phi_{\Bi}(b) k_{\Bi, \la}) = 0.
\]
Therefore if $b \in \bigcap_{i \in F} \ker \phi_{\Bi}$ then
\[
0 = \pi(b) \pi(a) \prod_{i \in F} (I - p_{\Bi}) = \pi(ba).
\]
As $\pi$ is injective we obtain that $ba = 0$ which implies that $a \perp \bigcap_{i \in F} \ker \phi_{\Bi}$.
This shows that $a \in \J_F$.
It remains to show that $\sca{X_{\un{x}}, a X_{\un{x}}} \subseteq \J_F$ for every $\un{x} \perp F$.
Recall that equation (\ref{eq:out prod}) can now be written as
\[
\pi(a) + \sum \{ (-1)^{|\un{n}|} \psi_{\un{n}}(k_{\un{n}}) \mid \un{0} \neq \un{n} \leq \un{1}_F \}= 0
\]
for $k_{\un{n}} = \phi_{\un{n}}(a)$, so that $\pi(a) \in \B_{(\un{0}, \un{1}_F]}$.
However, as $\un{n} \leq \un{1}_F$ then $\un{n} \perp \un{x}$, and equation (\ref{eq:stable}) yields
\[
t_{\un{x}}(X_{\un{x}})^* \psi_{\un{n}}(\K X_{\un{n}}) t_{\un{x}}(X_{\un{x}}) \subseteq \psi_{\un{n}}(\K X_{\un{n}}).
\]
Therefore, for $\xi_{\un{x}}, \eta_{\un{x}} \in X_{\un{x}}$, we get that
\[
\pi(\sca{\xi_{\un{x}}, a \eta_{\un{x}}}) = t_{\un{x}}(\xi_{\un{x}})^* \pi(a) t_{\un{x}}(\eta_{\un{x}}) \in  \B_{(\un{0}, \un{1}_F]}.
\]
Then the first part of the proof gives that $\sca{\xi_{\un{x}}, a \eta_{\un{x}}} \in \J_F$, which completes the proof.
\end{proof}

\begin{remark}
The essential part for Proposition \ref{P:in IF} is to argue that, for appropriate $f \in \ca(\pi,t)$, the w*-limit of nets $(f p_{\Bi, \la})_\la$ is in fact a norm-limit.
Solving polynomial equations becomes easier if the $(p_{\Bi, \la})_\la$ already converges in norm to $p_{\Bi}$.
This is the case for a large number of higher rank examples; in particular when every $\K X_{\Bi}$ has a unit.
\end{remark}

\section{CNP-relations and the Gauge Invariant Uniqueness Theorem}\label{S:CNP-rep}

\subsection{CNP-representations of Sims-Yeend}\label{Ss:SY}

One of the key elements of the GIUT is to provide a CNP-rep\-resentation $(\pi,t)$ that admits a gauge action and $\pi$ is injective. 
Sims and Yeend \cite{SY11} established the embedding of $A$ in the Cuntz-Nica-Pimsner algebras by using the augmented Fock space representation.
They achieve this in much greater generality, i.e., for product systems over a quasi-lattice ordered semigroup, but their CNP relations involve infinitely many fibers. 
We will compare our CNP-relations with those of Sims-Yeend, and use their result to show that $A$ always embeds into $\N\O(X)$.
The main point is that they coincide for strong compactly aligned product systems.
First we recall the terminology from \cite{SY11}.
Let 
\[
\I_{\un{0}} := A 
\qand 
\I_{\un{n}} = \bigcap \{ \ker \phi_{\un{m}} \mid \un{0} \neq \un{m} \leq \un{n} \}
\text{ for } \un{n} \in \bZ_+^N,
\]
all of which are ideals in $A$. 
For $\un{\ell} \in \bZ_+^N$, let
\[
\widetilde{X}_{\un{\ell}} = \bigoplus \{ X_{\un{m}} \I_{\un{\ell}-\un{m}} \mid \un{m} \leq \un{\ell} \}
\]
and write $\widetilde{\phi_{\un{\ell}}}$ for the left action on $\widetilde{X}_{\un{\ell}}$. 
Consequently, for $\un{n} \leq \un{\ell}$ we obtain a $*$-homomor\-phism 
\[
\widetilde{i}_{\un{n}}^{\un{\ell}} \colon \L X_{\un{n}} \rightarrow \L \widetilde{X}_{\un{\ell}}
\; \textup{ with } \;
\widetilde{i}_{\un{n}}^{\un{\ell}} := \oplus \{ i_{\un{n}}^{\un{m}} \mid \un{n} \leq \un{m} \leq \un{\ell} \}.
\]
A Nica-covariant representation $(\pi,t)$ of $X$ will be called \emph{Sims--Yeend CNP} if:
\begin{quoting}
\emph{$\sum \{ \psi_{\un{n}}(k_{\un{n}}) \mid \un{n}\in \F \} = 0$ for any finite set $\F \subseteq \bZ_+^N$, and every choice $\{k_{\un{n}} \in \K(X_{\un{n}}) \mid  \un{n}\in \F \}$ such that $\sum \{ \, \widetilde{i}_{\un{n}}^{\un{\ell}}(k_{\un{n}}) \mid \un{n}\in \F \, \} = 0$ for large $\un{\ell}$.} 
\end{quoting}
Since $\bZ_+^N$ is an Ore semigroup, the phrase ``for large $\un{\ell}$" means that for every $\un{s}$ there exists an $\un{s} \leq \un{r} \in \bZ_+^N$ such that the above property holds for all $\un{\ell} \geq \un{r}$.

\begin{proposition} \label{P:SYCNP-CNP}
Let $X$ be a strong compactly aligned product system, and let $(\pi,t)$ be a Nica-covariant representation of $X$. 
If $(\pi,t)$ is a Sims-Yeend CNP-representation, then it is a CNP-representation.
\end{proposition}

\begin{proof}
Fix $\mt \neq F\subseteq [N]$. 
By definition, every $a\in \I_F$ satisifies $\phi_{\un{n}}(a) \in \K(X_{\un{n}})$ for all $\un{n} \leq \un{1}$. 
Hence, in order to show that
\[
\sum \{ (-1)^{|\un{n}|}\psi_{\un{n}}(\phi_{\un{n}}(a)) \mid \un{n}\leq \un{1}_F \} = 0
\]
from the Sims--Yeend CNP relations, it suffices to show that
\begin{equation*}
\sum \{ (-1)^{|\un{n}|} \widetilde{i}_{\un{n}}^{\un{\ell}}(\phi_{\un{n}}(a)) \mid \un{n}\leq \un{1}_F \} = 0 \foral \un{\ell}\geq \un{1}.
\end{equation*}
To this end we fix $\un{m}\leq \un{\ell}$ and we show that this holds on each summand $X_{\un{m}}\I_{\un{\ell}-\un{m}}$ of $\widetilde{X}_{\un{\ell}}$. 
Notice that each $\widetilde{i}_{\un{n}}^{\un{\ell}}$ acts on $X_{\un{m}}\I_{\un{\ell} - \un{m}}$ via $i_{\un{n}}^{\un{m}}$, and that
\[
(-1)^{|\un{n}|} i_{\un{n}}^{\un{m}}(\phi_{\un{n}}(a)) = (-1)^{|\un{n}|}\phi_{\un{m}}(a).
\]
Hence we need only show that
\begin{equation} \label{eq:ann-sum}
\sum \{ (-1)^{|\un{n}|}\phi_{\un{m}}(a) \mid \un{n}\leq \un{1}_F \wedge \un{m} \} = 0.
\end{equation}
Let $F' := F \cap \supp{\un{m}}$. 
If $F' \neq \emptyset$, then $\un{1}_F \wedge \un{m} \neq \un{0}$ in which case
\[
\sum \{ (-1)^{|\un{n}|} \mid \un{n} \leq \un{1}_F \wedge \un{m} \} = 0
\]
and equation \eqref{eq:ann-sum} is satisfied.
On the other hand if $F' = \emptyset$ then equation \eqref{eq:ann-sum} has only one summand, and we are left with showing that $\phi_{\un{m}}(a) = 0$ on $X_{\un{m}}\I_{\un{\ell} - \un{m}}$, when $\un{m} \perp F$. 
In this case let $\xi_{\un{m}}, \eta_{\un{m}} \in X_{\un{m}}$ and $c, d\in \I_{\un{\ell}-\un{m}}$.
Since $\un{\ell} \geq \un{1}$ and $\un{m} \perp F$, we see that $\I_{\un{\ell} - \un{m}} \subseteq \bigcap_{i\in F}\ker \phi_{\bo{i}}$. 
Since $\sca{\xi_{\un{m}}, a \eta_{\un{m}}} \in \J_F$ and $\J_F \perp \bigcap_{i\in F}\ker \phi_{\bo{i}}$ we have for $c,d \in \I_{\un{\ell} - \un{m}}$ that
\[
\sca{\xi_{\un{m}} c, a \eta_{\un{m}} d} = c^* \sca{\xi_{\un{m}}, a \eta_{\un{m}}} d = 0.
\]
Therefore $\phi_{\un{m}}(a)$ vanishes on $X_{\un{m}}\I_{\un{\ell}-\un{m}}$ and the proof is complete.
\end{proof}

\subsection{The Gauge Invariant Uniqueness Theorem}\label{Ss:GIUT}

To make a distinction, we will be denoting the universal Sims-Yeend CNP algebra by $\N\O_X$. 
By Proposition \ref{P:SYCNP-CNP} and universality of the C*-algebras, we have a canonical surjective $*$-homo\-mor\-phism $\Phi \colon \N\O(X) \rightarrow \N\O_X$ that fixes generators of the same index.

\begin{theorem}[$\bZ_+^N$-GIUT] \label{T:GIUT}
Let $(\pi,t)$ be a CNP-representation of a strongly compactly aligned product system $X$.
Then $\pi \times t$ defines a faithful representation on $\N\O(X)$ if and only if $\pi$ is injective and $(\pi,t)$ admits a gauge action.
\end{theorem}

\begin{proof}
For the forward implication we have that $(\bZ^N,\bZ_+^N)$ satisfies \cite[Equation (3.5)]{SY11}, and thus each $\widetilde{\phi}_{\un{k}}$ is injective. 
Hence, by \cite[Theorem 4.4]{SY11} $A$ embeds faithfully in $\N\O_X$ and thus in $\N\O(X)$.
Universality of $\N\O(X)$ yields the existence of a gauge action.

For the converse let $\Phi_u$ be the universal representation of $\N\O(X)$ and set $\pi_u = \Phi_u|_A$ and $t_u = \Phi_u|_X$.
Since $A$ embeds in $\N\O(X)$ the pair $(\pi_u, t_u)$ is isometric.
Suppose that $(\pi, t)$ satisfies the hypothesis but $\pi \times t$ is not faithful.
To make a distinction we use $\B$ for the cores of $(\pi_u, t_u)$ and $B$ for the cores of $(\pi,t)$.
By using the gauge action, the intersection of $\ker(\pi \times t)$ with the fixed point algebra is not trivial.
As the fixed point algebra $B_{[\un{0},\infty]}$ is an inductive limit, there is an $m \in \bZ_+$ and an $\mt \neq F \subseteq [N]$ such that
\[
\ker (\pi \times t) \bigcap \B_{[\un{0}, m \cdot \un{1}_F]} \neq (0).
\]
Therefore there are $a \in A$ and $k_{\un{n}} \in \K X_{\un{n}}$ such that
\[
f := \pi_u(a) + \sum \{ \psi_{u, \un{n}}(k_{\un{n}}) \mid \un{0} \neq \un{n} \leq m \cdot \un{1}_F\} 
\]
is a non-zero element in $\ker (\pi \times t)$.
Then $\pi(a) \in B_{(\un{0}, m \cdot \un{1}_F]}$, and Proposition \ref{P:in IF} yields that $a \in \I_F$.
Thus
\[
\pi_u(a) = - \sum\{ (-1)^{|\un{n}|} \psi_{u, \un{n}}(\phi_{\un{n}}(a)) \mid {\un{0} \neq \un{n} \leq \un{1}_F} \}.
\]
Consequently we may rewrite
\[
\sum \{ \psi_{\un{n}}(k_{\un{n}}') \mid \un{0} \neq \un{n} \leq m \cdot \un{1}_F \} = 0
\]
for
\[
k_{\un{n}}' = 
\begin{cases}
k_{\un{n}} + (-1)^{|\un{n}| + 1} \phi_{\un{n}}(a) & \textup{ when } \un{0} \neq \un{n} \leq \un{1}_F, \\
k_{\un{n}} & \textup{ otherwise}.
\end{cases}
\]
Now suppose we write $f$ in an irreducible form, meaning that 
\[
\pi(a) + \sum \{ \psi_{\un{n}}(k_{\un{n}}) \mid \un{0} \neq \un{n} \leq m \cdot \un{1}_F \} = 0
\]
where some of the $a, k_{\un{n}}$ are zero and if a $k_{\un{n}} \neq 0$ then $\psi_{\un{n}}(k_{\un{n}})$ is not in $B_{(\un{n}, m \cdot \un{1}_F]}$. 
Choose a minimal $\un{x} \in (\un{0}, m \cdot \un{1}_F]$ so that $k_{\un{x}} \neq 0$ and $\psi_{\un{x}}(k_{\un{x}})$ is not in $B_{(\un{x}, m \cdot \un{1}_F]}$.
Then for every $\xi_{\un{x}}, \eta_{\un{x}} \in X_{\un{x}}$ we still have that 
\[
t_{u,\un{x}}(\xi_{\un{x}})^* f t_{u,\un{x}}(\eta_{\un{x}}) \in \ker(\pi \times t)
\]
and therefore
\[
\pi(a') + \sum \{ \psi_{\un{n}}(k_{\un{n}}'') \mid \un{0} \neq \un{n} \leq m \cdot \un{1}_F - \un{x} \} 
=
t_\un{x}(\xi_\un{x})^* (\pi \times t)(f) t_{\un{x}}(\eta_{\un{x}})
= 0
\]
where now $a' = \sca{\xi_{\un{x}}, k_{\un{x}} \eta_{\un{x}}}$ and  likewise for the appropriate $k_{\un{n}}''$.
Following the previous arguments we see that $a' \in \I_F$ and therefore
\[
t_{\un{x}}(\xi_{\un{x}})^* \psi_{\un{x}}(k_{\un{x}}) t_{\un{x}}(\eta_{\un{x}}) = \pi(a') \in \B_{(\un{0}, m \cdot \un{1}_F - \un{x}]}.
\]
Consequently we obtain that
\[
\psi_{\un{x}}(\K X_{\un{x}}) \psi_{\un{x}}(k_{\un{x}}) \psi_{\un{x}}(\K X_{\un{x}}) \in \B_{(\un{x}, m \cdot \un{1}_F]}.
\]
By choosing an approximate identity in $\psi_{\un{x}}(\K X_{\un{x}})$ we reach the contradiction that $k_{\un{x}}$ is in $B_{(\un{x}, m \cdot \un{1}_F]}$.
Hence $k_{\un{x}} = 0$.
Inductively we eliminate all entries and arrive at the point where $f = \psi_{u, m \cdot \un{1}_F}(k)$ for some $k \in \K X_{m \cdot \un{1}_F}$.
However as $\pi$ is injective and thus so is $\psi_{m \cdot \un{1}_F}$, we derive that $k = 0$, and the proof is complete.
\end{proof}

As both $\N\O(X)$ and $\N\O_X$ satisfy the GIUT, we get that they coincide and that $\N \O_X$ has a simpler universal property.

\begin{corollary}\label{C:CNP is CNP}
Let $X$ be a strong compactly aligned product system over $A$.
Then the natural map $\Phi \colon \N\O(X) \rightarrow \N\O_X$ is a $*$-isomorphism. 
Thus, a pair $(\pi,t)$ is a CNP-representation if and only if it is a Sims-Yeend CNP-represen\-tation.
\end{corollary}

\subsection{The CNP-ideal} \label{subsec:cnp-ideal}

Next we describe the ideal of the CNP-relations.
To this end let $(\pi,t)$ be a Nica-covariant representation.
For $\mt \neq F \subseteq [N]$ set
\begin{equation}
q_F := \prod_{i \in F}(I - p_i)
\end{equation}
where the product \emph{initially} is taken in the usual order of $[N]$. 
When $a \in \I_F$ then $\phi_{\un{n}}(a) \in \K X_{\un{n}}$ for all $\un{n} \leq \un{1}$ and thus
\[
\pi(a)q_F = \pi(a) + \sum \{ (-1)^{|\un{n}|} \psi_{\un{n}}(\phi_{\un{n}}(a)) \mid \un{0} \neq \un{n} \leq \un{1}_F \}.
\]
Hence the product in $\pi(a) q_F$ is independent of the order in $F$.
At the same time $p_{\Bi} \pi(a) = \pi(a) - \psi_{\Bi}(\phi_{\Bi}(a)) = \pi(a) p_{\Bi}$ for $i\in F$. Hence, we also get that $q_F \in \pi(A)'$.

\begin{proposition}\label{P:pf reducing}
Let $(\pi,t)$ be a Nica-covariant representation of a strong compactly aligned product system $X$. Then
\[
 q_Ft_{\un{m}}(\xi_{\un{m}})
=
\begin{cases}
t_{\un{m}}(\xi_{\un{m}})q_F & \text{ if } \un{m} \perp F, \\
0 & \text{ if } \un{m} \not\perp F,
\end{cases}
\]
for all $\xi_{\un{m}} \in X_{\un{m}}$.
Consequently, for all $\un{m} \in \bZ_+^N$ we have that
\[
\pi(\I_F) q_F t_{\un{m}}(X_{\un{m}}) \subseteq
\pi(\I_F) t_{\un{m}}(X_{\un{m}}) q_F 
\]
where the product defining $q_F$ can be taken in any order in these relations.
\end{proposition}

\begin{proof}
Suppose that $F = \{1, \dots, r\}$ and fix $\xi_{\un{m}} \in X_{\un{m}}$.
First consider the case where $\un{m} \perp F$ and let $i \in F$.
Equation (\ref{eq:Fow}) and $X_{\un{m}} \otimes X_\Bi \simeq X_\Bi \otimes X_{\un{m}}$ imply
\begin{align*}
t_{\un{m}}(\xi_{\un{m}})^* \psi_\Bi(\K X_\Bi)
& \subseteq
\ol{t_\Bi(X_\Bi) t_{\un{m}}(X_{\un{m}})^* t_\Bi(X_\Bi)^*}
=
\ol{t_\Bi(X_\Bi) t_\Bi(X_\Bi)^* t_{\un{m}}(X_{\un{m}})^*}
\end{align*}
so that $t_{\un{m}}(\xi_{\un{m}})^* p_\Bi H \subseteq p_\Bi H$.
Moreover we have that
\begin{align*}
t_{\un{m}}(\xi_{\un{m}}) \psi_\Bi(\K X_\Bi)
& \subseteq
\ol{t_\Bi(X_\Bi) t_{\un{m}}(X_{\un{m}}) t_\Bi(X_\Bi)^*}
=
\ol{\psi_\Bi(\K X_\Bi) t_\Bi(X_\Bi) t_{\un{m}}(X_{\un{m}}) t_\Bi(X_\Bi)^*}
\end{align*}
so that $t_{\un{m}}(\xi_{\un{m}}) p_\Bi H \subseteq p_\Bi H$. Thus $p_\Bi$ is reducing for $t_{\un{m}}(X_{\un{m}})$.

Now suppose that $\supp \un{m} \bigcap F \neq \mt$ and let $i \in \supp \un{m} \bigcap F$.
Then $p_{\bo{i}} t_{\un{m}}(\xi_{\un{m}}) = t_{\un{m}}(\xi_{\un{m}})$ as $p_{\bo{i}}$ is a left unit for $t_{\un{m}}(X_{\un{m}}) = t_\Bi(X_{\bo{i}}) t_{\un{m} - \Bi}(X_{\un{m} - \bo{i}})$. 
Therefore
\[
(I - p_{\bo{i}}) t_{\un{m}}(\xi_{\un{m}}) = 0.
\]
Likewise we have that
\begin{equation}\label{eq:more perp}
(I - p_{\bo{i}}) \psi_{\un{n}}(\K X_{\un{n}}) t_{\un{m}}(\xi_{\un{m}}) = 0 \foral \un{n} \in \bZ_+^N.
\end{equation}
Indeed if $i \in \supp \un{n}$ then
\[
\psi_{\un{n}} (\K X_{\un{n}}) t_{\un{m}}(X_{\un{m}})
=
t_{\Bi}(X_{\Bi}) \psi_{\un{n} - \Bi} (\K X_{\un{n} - \Bi}) t_{\Bi}(X_{\Bi})^* t_{\un{m}}(X_{\un{m}}),
\]
and we proceed as before, whereas if $i \notin \supp \un{n}$ then equation (\ref{eq:Fow}) gives
\[
\psi_{\un{n}} (\K X_{\un{n}}) t_{\un{m}}(X_{\un{m}})
\subseteq
\ol{t_{\Bi}(X_{\Bi}) \psi_{\un{n}} (\K X_{\un{n}}) t_{\un{m} - \Bi}(X_{\un{m} - \Bi})}.
\]
In both cases we have $(I-p_{\bo{i}})\psi_{\un{n}}(\K X_{\un{n}}) t_{\un{m}}(\xi_{\un{m}}) = 0$.
Now we can show that $q_Ft_{\un{m}}(\xi_{\un{m}}) = 0$. 
By the above comments, it suffices to show that $q_F t_{\bo{i}}(\xi_{\bo{i}}) = 0$ for $i\in F$. 
There are two cases. 
If $i = r$ then we get
\[
q_F t_{\bo{i}}(\xi_{\bo{i}}) = (\prod_{r \neq i \in F} (I - p_{\Bi}) ) \cdot (I - p_{\bo{r}})t_{\bo{r}}(\xi_{\bo{r}}) = 0.
\]
If $i < r$ then let $H$ be the Hilbert space that $(\pi,t)$ acts on. 
For $j = i +1 \leq r$ we obtain
\begin{align*}
(I- p_{\bo{j}}) \cdots (I - p_{\bo{r}})  t_{\bo{i}}(\xi_{\bo{i}}) H
& \subseteq 
t_{\bo{i}}(X_{\bo{i}}) H \vee (\bigvee \{\psi_{\un{n}}(\K X_{\un{n}}) t_{\bo{i}}(X_{\bo{i}}) H \mid \supp \un{n} \subseteq [j, r] \}).
\end{align*}
Therefore by equation (\ref{eq:more perp}) we still get
\[
(I - p_{\bo{i}})(I- p_{\bo{j}}) \cdots (I - p_{\bo{r}})  t_{\bo{i}}(\xi_{\bo{i}}) H = \{0\}
\]
which implies that $q_F t_{\bo{i}}(\xi_{\bo{i}}) = 0$.
The arguments above can be applied when considering any order in the product defining $q_F$ in the relations, and the proof is complete.
\end{proof}


\begin{definition}
Let $(\pi,t)$ be a Nica-covariant representation of a strong compactly aligned product system $X$.
For $\mt \neq F \subseteq [N]$ we define the subspace
\[
\fI_F := \ol{\spn}\{ t_{\un{n}}(\xi_{\un{n}}) \pi(a) q_F t_{\un{m}}(\eta_{\un{m}})^* \mid a \in \I_F, \un{n}, \un{m} \in \bZ_+^N \}.
\] 
The \emph{ideal of the CNP-relations} (CNP-ideal) is defined as the algebraic sum
\[
\K(\pi,t): = \sum \{ \fI_F \mid \mt \neq F \subseteq [N] \}.
\]
\end{definition}

The next proposition justifies this terminology.

\begin{proposition}\label{P:K ideal}
Let $X$ be a strong compactly aligned product system over $A$.
Suppose that $(\pi,t)$ is a Nica-covariant representation of $X$.
If $\mt \neq F \subseteq [N]$ then the subspace $\fI_F$ is the ideal generated by $\{\pi(a) q_F \mid a \in \I_F\}$ in $\ca(\pi,t)$. 
In particular, $\K(\pi,t)$ is a norm-closed ideal in $\ca(\pi,t)$.
\end{proposition}

\begin{proof}
Fix an $\mt \neq F \subseteq [N]$. 
First we show that
\[
\pi(a) t_{\un{m}}(X_{\un{m}}) \in \ol{t_{\un{m}}(X_{\un{m}}) \pi(\I_F)} \foral a \in \I_F, \un{m} \perp F.
\]
Indeed if $a \in \I_F$ then $a^* a \in \I_F$ and thus
\[
t_{\un{m}}(X_{\un{m}})^* \pi(a)^*\pi(a) t_{\un{m}}(X_{\un{m}}) 
\subseteq
\pi(\sca{X_{\un{m}}, a^*a X_{\un{m}}})
\subseteq
\pi(\I_F).
\]
By applying equation (\ref{eq:factorization}) for the Hilbert module $X_{\un{m}}$, we get that for any $\xi_{\un{m}} \in X_{\un{m}}$ there exists an $\eta_{\un{m}} \in X_{\un{m}}$ such that
\[
\pi(a) t_{\un{m}}(\xi_{\un{m}}) = t(\eta_{\un{m}}) \pi(b) \qfor b = |\sca{\xi_{\un{m}}, a^*a \xi_{\un{m}}}|^{1/2}. 
\]
As $b^2 \geq 0$ is in $\I_F$ then so is $b$.

Next we show that $\fI_F$ is an ideal of $\ca(\pi,t)$.
It is clear that $\fI_F$ is selfadjoint and that $\fI_F \pi(A) \subseteq \fI_F$.
Recall that $q_F$ is in the commutant of $\pi(\I_F)$.
Hence Proposition \ref{P:pf reducing} gives
\begin{align*}
\pi(a) q_F t_{\un{m}}(\xi_{\un{m}})^* \pi(b) q_{F}
& =
\pi(a) q_F t_{\un{m}}(\xi_{\un{m}})^* q_F \pi(b) \\
& =
\begin{cases}
\pi(a)q_{F} t_{\un{m}}(b^* \xi_{\un{m}})^* & \text{ if } \un{m} \perp F, \\
0 & \text{ if } \un{m} \not\perp F,
\end{cases}
\end{align*}
for $a,b \in \fI_F$. 
It is also clear that multiplying an element in $\fI_F$ on the right by $t_{\un{n}}(\xi_{\un{n}})^*$ gives an element in $\fI_F$.
For arbitrary $\un{n}$ and $\un{m}$ we have that
\[
\pi(a) q_F t_{\un{m}}(\xi_{\un{m}})^* t_{\un{n}}(\xi_{\un{n}}) \in \pi(a)q_F \ol{t_{\un{n}'}(X_{\un{n}'}) t_{\un{m}'}(X_{\un{m}'})^*}
\]
for
\[
\un{n}' = - \un{n} + \un{n} \vee \un{m} \qand \un{m}' = -\un{m} + \un{n} \vee \un{m}.
\]
However in view of Proposition \ref{P:pf reducing} if $\un{n}' \not\perp F$ then
\[
\pi(a) q_F t_{\un{n}'}(X_{\un{n}'})= \{0\},
\]
whereas if $\un{n}' \perp F$ then by the comments above we get that
\[
\pi(a) q_F t_{\un{n}'}(X_{\un{n}'}) \subseteq \pi(\I_F) t_{\un{n}'}(X_{\un{n}'}) q_F \subseteq \ol{t_{\un{n}'}(X_{\un{n}'}) \pi(\I_F) q_F}.
\]
In every case we have that 
\[
\pi(a) q_F t_{\un{m}}(\xi_{\un{m}})^* t_{\un{n}}(\xi_{\un{n}}) \in \ol{t_{\un{n}'}(X_{\un{n}'}) \pi(\I_F) q_F t_{\un{m}'}(X_{\un{m}'})^*}
\]
which completes the first part of the statement, since $\fI_F$ is self-adjoint.

For the second part, we remark that if $F \supseteq F'$ then
\begin{align*}
\pi(a) q_F t_{\un{m}}(\xi_{\un{m}})^* \pi(b) q_{F'}
& =
\begin{cases}
\pi(a)q_{F \cup F'} t_{\un{m}}(b^* \xi_{\un{m}})^* & \text{ if } \un{m} \perp F', \\
0 & \text{ if } \un{m} \not\perp F',
\end{cases}
\\
& =
\begin{cases}
\pi(a)q_{F} t_{\un{m}}(b^* \xi_{\un{m}})^* & \text{ if } \un{m} \perp F', \\
0 & \text{ if } \un{m} \not\perp F',
\end{cases}
\end{align*}
for $a \in \I_F$ and $b \in \I_{F'}$.
Since $\I_F \supseteq \I_{F'}$ we obtain
\[
\pi(a) q_F t_{\un{m}}(\xi_{\un{m}})^* \cdot t_{\un{n}}(\xi_{\un{n}}) \pi(b) q_{F'} \in \fI_F
\]
for every $\un{n}$.
Hence we derive that $\fI_F \cdot \fI_{F'} \subseteq \fI_F$.
Thus $\fI_{[N]}$ is a norm-closed ideal of $\ol{\fI_F + \fI_{[N]}}$ whenever $|F| = N-1$.
Therefore
\[
\ol{\fI_F + \fI_{[N]}} = {\fI_F + \fI_{[N]}}.
\]
Likewise if $|F_1| = \dots = |F_k| = N-1$ for some $k$ then 
\[
\sum_{n=2}^k \fI_{F_n} + \fI_{[N]} 
\; \textup{ is an ideal of } \;
\{\sum_{n=1}^k \fI_{F_n} + \fI_{[N]}\}^{\ol{\phantom{oo}}}
\]
and therefore $\sum_{n=1}^k \fI_{F_n} + \fI_{[N]}$ is closed.
Induction then completes the proof.
\end{proof}

We are now in position of recovering the co-universal property of $\N\O(X)$ established in \cite[Theorem 4.1]{CLSV11}.
Our approach is different and emphasizes the explicit form of the CNP-relations.

\begin{corollary}\label{C:co-un} \cite[Theorem 4.1]{CLSV11}
Let $X$ be a strong compactly aligned product system over $A$.
If $(\pi,t)$ is a Nica-covariant representation that admits a gauge action and $\pi$ is injective then there is a canonical $*$-epimorphism $\ca(\pi,t) \to \N\O(X)$ that fixes the generators of the same index.
Consequently $\ca(\pi, t)/ \K(\pi,t)$ is $*$-isomorphic to $\N\O(X)$.
\end{corollary}

\begin{proof}

Let $(\pi,t)$ be a Nica-covariant representation that admits a gauge action and $\pi$ is injective.
To allow comparisons we denote the Fock representation by $(\si, s)$ and write $\psi^t$ and $\psi^s$ for the induced representations on the compacts $\K X_{\un{n}}$.
Denote also $q: = \pi \times t$ and $q':= \Psi \circ (\sigma \times s)$ where $\Psi : \N \T(X) \rightarrow \N \O(X)$ is the natural quotient map by the CNP ideal.

It will suffice to show that $\ker q \subseteq \ker q' = \K(\si, s)$. 
This way, the map $q'$ factors through $\ca(\pi,t)$ to a $\bT^N$-equivariant map $\ca(\pi,t) \rightarrow \N\O(X)$ which is still injective on $A$, and the GIUT finishes the proof.
As $(\pi, t)$ admits a $\bT^N$-gauge action $\beta$ 
and $\ca(\si, s)^\be$ is the inductive limit of the $\B_{[\un{0}, m \cdot \un{1}]}$ with $m \in \bZ_+$, it suffices to show that
\[
(\ker q)^\be \bigcap \B_{[\un{0}, m \cdot \un{1}]} \subseteq \K(\si, s) \foral m \in \bZ_+.
\]
We will do this inductively.
As $\pi$ is injective there is nothing to show for $m=0$.
So take $m=1$ and fix
\[
\ker q \ni f = \si(a_{\un{0}}) + \sum \{ \psi_{\un{n}}^s(k_{\un{n}}) \mid 0 \neq \un{n} \leq \un{1} \}.
\]
By solving the equation $q(f) = 0$ as in Section \ref{S:pol eq} we derive that
\[
\pi(a_{\un{0}}) \prod_{i \in [N]} (I - q(p_{\Bi})) = 0,
\]
where $p_{\Bi}$ are the projections associated to $(\si, s)$.
This gives that $\si(a_{\un{0}}) \prod_{i \in [N]} (I - p_{\Bi}) \in \ker q$ and also that $a_{\un{0}} \in \I_{[N]}$.
Hence
\[
\si(a_{\un{0}}) \prod_{i \in [N]} (I - p_{\Bi}) = \si(a_{\un{0}}) q_{[N]} \in (\ker q)^\be \bigcap \K(\si, s).
\]
For convenience set 
\[
f_{\un{0}} := \si(a_{\un{0}}) \prod_{i \in [N]} (I - p_{\Bi}).
\]
Then $f - f_{\un{0}} \in (\ker q)^\be$ and we can group compacts of the same index to have an equation of the form
\[
g:= f - f_{\un{0}} = \sum \{ \psi_{\un{n}}^s(k_{\un{n}}') \mid 0 \neq \un{n} \leq \un{1} \}
\]
for appropriate $k_{\un{n}}' \in \K(X_{\un{n}})$.
That is, the element $g \in (\ker q)^\be$ can be written as having its $\un{0}$-summand equal to zero.
Suppose now that $\un{m} \leq \un{1}$ is a minimal position such that $\psi_{\un{m}}^s(k_{\un{m}}') \neq 0$.
For $\xi_{\un{m}}, \eta_{\un{m}} \in X_{\un{m}}$ we then obtain
\[
s_{\un{m}}(\xi_{\un{m}})^* g s_{\un{m}}(\eta_{\un{m}}) 
= 
\si(b_{\un{0}}) + \sum \{ \psi_{\un{n}}^s(k_{\un{n}}'') \mid \un{n} \leq \un{1} - \un{m} \ \}
\in (\ker q)^\be
\]
where $b_{\un{0}} = \sca{\xi_{\un{m}}, k_{\un{m}}' \eta_{\bo{j}}}$.
Proceeding as before we have 
\[
\pi(b_{\un{0}}) \prod_{i \in [N] \setminus \supp{\un{m}}} (I - q(p_{\Bi})) = 0.
\]
Once more this gives that 
\[
\si(b_{\un{0}}) \prod_{i \in [N] \setminus \supp{\un{m}}} (I - p_{\Bi}) 
\in 
(\ker q)^\be \bigcap \K(\si, s).
\]
At the same time, for $\xi_{\un{m}}', \eta_{\un{m}}' \in X_{\un{m}}$, denote $k_{1, \un{m}} = \theta^{\un{m}}_{\xi_{\un{m}}, \xi_{\un{m}}'}$ and $k_{2, \un{m}} = \theta^{\un{m}}_{\eta_{\un{m}}, \eta_{\un{m}}'}$, so that
\begin{align*}
0 & =
t_{\un{m}}(\xi_{\un{m}}') \bigg( \pi(b_{\un{0}}) \prod_{i \in [N] \setminus \supp{\un{m}}} (I - q(p_{\Bi})) \bigg)  t_{\un{m}}(\eta_{\un{m}}')^* \\
& =
t_{\un{m}}(\xi_{\un{m}}') t_{\un{m}}(\xi_{\un{m}})^* \psi^t_{\un{m}}(k_{\un{m}}) t_{\un{m}}(\eta_{\un{m}}) t_{\un{m}}(\eta_{\un{m}}')^* \prod_{i \in [N] \setminus \supp{\un{m}}} (I - q(p_{\Bi})) \\
& =
\psi_{\un{m}}^t(k_{1, \un{m}} \cdot k_{\un{m}} \cdot k_{2, \un{m}}) \prod_{i \in [N] \setminus \supp{\un{m}}} (I - q(p_{\Bi}))
\end{align*}
where we used Proposition \ref{P:pf reducing} to commute $t_{\un{m}}(\eta_{\un{m}}')$ with $\prod_{i \in [N] \setminus \supp{\un{m}}} (I - q(p_{\Bi}))$.
As this holds for elementary compacts, it still holds when replacing $k_{1, \un{m}}$ and $k_{2, \un{m}}$ with a c.a.i. in $\K X_{\un{m}}$.
Taking limits yields
\[
f_{\un{m}}:= \psi_{\un{m}}^s(k_{\un{m}}) \prod_{i \in [N] \setminus \supp{\un{m}}} (I - p_{\Bi}) \in  (\ker q)^\be \bigcap \K(\si, s).
\]
Hence, we get that
\[
f - (f_{\un{0}} + f_{\un{m}}) \in (\ker q)^\be \text{ with } f_{\un{0}}, f_{\un{m}} \in (\ker q)^\be \bigcap \K(\si, s).
\]
Now we see that this element does not contain a summand on $\un{0}$ and $\un{m}$.
We proceed inductively along the minimal remaining positions each time to find
\[
f_{\un{n}} \in (\ker q)^\be \bigcap \K(\si, s) \foral \un{0} \leq \un{n} < \un{1}
\]
such that
\[
g_{\un{1}} : = f - \sum\{f_{\un{n}} \mid \un{0} \leq \un{n} < \un{1} \} \in (\ker q)^\be.
\]
However $g_{\un{1}} \in \psi^s_{\un{1}}(\K X_\un{1})$ and $\psi^t_{\un{1}} = q \psi_{\un{1}}^s$ is injective as $\pi$ is injective.
Therefore $g_{\un{1}} = 0$ and we get that
\[
f = \sum\{f_{\un{n}} \mid \un{0} \leq \un{n} < \un{1} \} \in \K(\si, s).
\]
We conclude that
\[
(\ker q)^{\be} \cap \B_{[\un{0}, \un{1}]} \subseteq \K(\sigma, s).
\]
We use this to show that 
\[
(\ker q)^{\be} \cap \B_{[\un{0}, m \cdot \un{1}]} \subseteq \K(\sigma, s) \foral m \in \bN.
\]
We will proceed by adding one generator each time.
That is, assuming we have that $(\ker q)^{\be} \cap \B_{[\un{0}, m \cdot \un{1}]} \subseteq \K(\sigma, s)$ for a given $m$, we will prove inductively the inclusions
\[
(\ker q)^{\be} \cap \B_{[\un{0},m \cdot \un{1} + \bo{1} + ... +\bo{k}]} \subseteq \K(\sigma, s) \qfor k=1, \dots, N,
\]
so that for $k=N$ we will obtain $(\ker q)^{\be} \cap \B_{[\un{0}, (m+1) \un{1}]} \subseteq \K(\sigma, s)$.
To this end suppose that $(\ker q)^\be \cap  \B_{[\un{0}, \un{k}]} \subseteq \K(\sigma, s)$, for some $\un{k} \geq \un{1}$ and let $j \in [N]$; we will show that
\[
(\ker q)^\be \cap  \B_{[\un{0}, \un{k} + \bo{j}]} \subseteq \K(\sigma, s).
\]
Let $f \in (\ker q)^\be \cap  \B_{[\un{0}, \un{k} + \bo{j}]}$ and fix the set
\[
\F := \{ \un{n} \leq \un{k} \mid \bo{j} \notin \supp \un{n}\}.
\]
By moving along minimal elements as before, we derive that there are $f_{\un{n}}$, for $\un{n} \in \F$, such that
\[
g:= f - \sum \{ f_{\un{n}} \mid \un{n} \in \F\} \in (\ker q)^\be 
\qand 
f_{\un{n}} \in (\ker q)^\be \bigcap \K(\si, s) \foral \un{n} \in \F.
\]
By construction we have that $g$ is in $\B_{[\bo{j},\un{k} + \bo{j}]}$.
Therefore the c.a.i. of $\psi^s_{\bo{j}}(\K X_{\bo{j}})$ defines an approximate identity for $g$.
Moreover we have that
\begin{align*}
s_{\bo{j}}(X_{\bo{j}})^* g s_{\bo{j}}(X_{\bo{j}}) 
& \subseteq 
s_{\bo{j}}(X_{\bo{j}})^* \left((\ker q)^\be \bigcap \B_{[\un{0}, \un{k} + \bo{j}]}\right) s_{\bo{j}}(X_{\bo{j}})
\subseteq 
(\ker q)^\be \bigcap \B_{[\un{0}, \un{k}]}.
\end{align*}
Consequently we get that
\[
\psi^s_{\bo{j}}(\K X_{\bo{j}}) g \psi^s_{\bo{j}}(\K X_{\bo{j}}) \subseteq (\ker q)^\be \bigcap \K(\si, s).
\]
By applying a c.a.i. of $\psi^s_{\bo{j}}(\K X_{\bo{j}})$ (which is a c.a.i. for $g$) we obtain that 
\[
g \in (\ker q)^\be \bigcap \K(\si, s),
\]
and therefore
\[
f = g + \sum \{ f_{\un{n}} \mid \un{n} \in \F\} \in (\ker q)^\be \bigcap \K(\si, s).
\]
Hence, proceeding inductively we have that $(\ker q)^{\be} \cap \B_{[\un{0}, m\cdot \un{1}]} \subseteq \K(\sigma, s)$ for every $m\in \bN$, so that $(\ker q)^{\be} \subseteq \K(\sigma, s)$, as required.
\end{proof}

\section{Regular product systems}\label{S:regular}

Recall that a product system is called \emph{regular} if every $X_{\un{n}}$ is regular, and that every regular product system is automatically strong compactly aligned.
Fowler \cite{Fow02} introduced $\O_X$ as the universal C*-algebra for the representations $(\pi,t)$ such that $\pi(a) = \psi_{\un{n}}(\phi_{\un{n}}(a))$ for every $a\in A$ and $\un{n}\in \bZ_+^N$.
It is shown in \cite[Proposition 5.1]{SY11} that $\N\O_X$ coincides with $\O_X$ when $X$ is regular, and thus so does $\N\O(X)$ by Theorem \ref{T:CNP is CNP}.
Let us provide a direct proof that $\N\O(X) \simeq \O_X$ by using the alternating sums machinery.
En passant, we show that covariance needs to be checked only for $\un{n} \leq \un{1}$.

\begin{corollary}
Let $X$ be a regular product system over $\bZ_+^N$. 
Then $\N\O(X)$ is $*$-isomorphic to $\O_X$. In particular, a representation $(\pi,t)$ is covariant in the sense of \cite{Fow02} if and only if each $(\pi,t_{\un{n}})$ is covariant for all $\un{n} \leq 1$.
\end{corollary}

\begin{proof}
Suppose that $(\pi,t)$ is a representation such that $(\pi, t_{\un{n}})$ is covariant for every $\un{n} \in \bZ_+^N$.
By \cite[Proposition 5.4]{Fow02} we have that $(\pi,t)$ is automatically Nica-covariant. 
Then for any $\emptyset \neq F \subseteq [N]$ and $a \in A$ we check that
\begin{align*}
\sum \{ (-1)^{|\un{n}|} \psi_{\un{n}}(\phi_{\un{n}}(a)) \mid \un{n} \leq \un{1}_F \}
& =
\sum \{ (-1)^{|\un{n}|} \pi(a) \mid \un{n} \leq \un{1}_F \} 
= 0.
\end{align*}
Since $X$ is regular we have that $\I_F = A$ for all $F \subseteq [N]$ and so $(\pi,t)$ is a CNP-representation.

Conversely let $(\pi,t)$ be a CNP-representation.
First we show that $(\pi, t_{\un{n}})$ is covariant for every $\un{n} \leq \un{1}$.
For $F = \{i\}$ regularity implies that $\I_F = A$ and thus $(\pi, t_\Bi)$ is covariant for $X_{\Bi}$.
Now for $F = \{i, j\}$ we have that the CNP representation satisfies
\[
\pi(a) - \psi_{\Bi}(\phi_{\Bi}(a)) - \psi_{\bo{j}}(\phi_{\bo{j}}(a)) + \psi_{\Bi + \bo{j}}(\phi_{\bo{i} + \bo{j}}(a)) = 0.
\]
As $(\pi, t_\Bi)$ and $(\pi, t_\bo{j})$ are covariant for all $a \in A$ we get that
\begin{align*}
0
& = \pi(a) -  \psi_{\Bi}(\phi_{\Bi}(a)) - \psi_{\bo{j}}(\phi_{\bo{j}}(a)) + \psi_{\Bi + \bo{j}}(\phi_{\bo{i} + \bo{j}}(a))
 =
\pi(a) - \pi(a) - \pi(a) + \psi_{\Bi + \bo{j}}(a)
\end{align*}
and therefore $\pi(a) = \psi_{\Bi + \bo{j}}(a)$.
For the inductive step suppose that $(\pi, t_{\un{n}})$ is covariant for all $|\un{n}| \leq k$ and let $F$ be of cardinality $k+1$.
We have to show that $(\pi, t_{\un{1}_F})$ is covariant for $X_{\un{1}_F}$.
Fix $\un{x} = \un{1}_F - \bo{i}$ for $i \in F$ and compute
\begin{align*}
0 
& = 
\sum \{ (-1)^{|\un{n}|} \psi_{\un{n}}(\phi_{\un{n}}(a)) \mid \un{n} \leq \un{1}_F \} \\
& =
(-1)^{k} \psi_{\un{x}}(\phi_{\un{x}}(a)) + (-1)^{k+1} \psi_{\un{1}_F}(\phi_{\un{1}_F}(a)) + \sum \{ (-1)^{|\un{n}|} \pi(a) \mid \un{1}, \un{x} \neq \un{n} \lneq \un{1}_{F} \} \\
& =
(-1)^{k} \psi_{\un{x}}(\phi_{\un{x}}(a)) + (-1)^{k+1}\psi_{\un{1}}(\phi_{\un{1}_F}(a)).
\end{align*}
Thus we derive
\[
\psi_{\un{1}_F}(\phi_{\un{1}_F}(a)) = \psi_{\un{x}}(\phi_{\un{x}}(a)) = \pi(a).
\]
This finishes the part that $(\pi, t_\un{n})$ is covariant for $\un{n}\leq \un{1}$. 
Now let $\un{m} \in \bZ_+^N$ such that $\un{m} \not\leq \un{1}$ and suppose we have shown that $(\pi, t_{\un{x}})$ is covariant for all $\un{x} < \un{m}$. 
Set $\un{n} = \sum \{ \Bi \mid {i \in \supp \un{m}} \}$ and fix a c.a.i. $(a_\la)_\la$ for $A$. 
By applying on the elementary tensors in $X_{\un{m}} \simeq X_{\un{n}} \otimes X_{\un{m} - \un{n}}$ and using the Nica-covariance we get that $\phi_{\un{n}}(a_\la) \vee \phi_{\un{m}-\un{n}}(a) = \phi_{\un{m}}(a_\la a)$ and hence $\lim_\la \phi_{\un{n}}(a_\la) \vee \phi_{\un{m}-\un{n}}(a) = \phi_{\un{m}}(a)$.
Therefore we obtain
\begin{align*}
\pi(a) 
& = 
\lim_\la \pi(a_\la) \pi(a) 
 = 
\lim_\la \psi_{\un{n}}(\phi_{\un{n}}(a_\la)) \psi_{\un{m}-\un{n}}(\phi_{\un{m}-\un{n}}(a)) \\
& = 
\lim_\la \psi_{\un{m}} (\phi_{\un{n}}(a_\la) \vee \phi_{\un{m}-\un{n}}(a)) 
 = 
\psi_{\un{m}}(\phi_{\un{m}}(a))
\end{align*}
which finishes the proof.
\end{proof}

\section{C*-dynamical systems}\label{S:ds}

The second author with Davidson and Fuller considered C*-dynamical systems over $\bZ_+^N$ in \cite[Section 4]{DFK14}.
We will show how the Cuntz-Nica-Pimsner algebra of \cite{DFK14} relates to one coming from a product system.
To this end let $\al \colon \bZ_+^N \to \End(A)$ be a semigroup action.
For $\un{n} \in \bZ_+^N$ let $X_{\un{n}}$ be the closed linear subspace of $A$ generated by $\al_{\un{n}}(A) A$.
It becomes a C*-correspondence over $A$ when endowed with
\[
\sca{\xi, \eta} = \xi^* \eta \qand a \cdot \xi \cdot b = \al_{\un{n}}(a) \xi b
\]
for $\xi, \eta \in X_{\un{n}}$ and $a, b \in A$.
The unitary equivalence $X_{\un{n}} \otimes X_{\un{m}} \simeq X_{\un{n} + \un{m}}$ is induced by the $A$-balanced bilinear map $(\xi, \eta) \mapsto \al_{\un{m}}(\xi) \eta$.
The induced operator is surjective since
\[
\al_{\un{m}}(\al_{\un{n}}(A) A) \cdot \al_{\un{m}}(A) A = \al_{\un{n}+ \un{m}}(A) \al_{\un{m}}(A) A \subseteq \al_{\un{n} + \un{m}}(A) A
\]
and, if $(a_\la)$ is an approximate identity in $A$, then
\[
\al_{\un{n} + \un{m}}(a) b
=
\lim_\la \al_{\un{m}}(\al_{\un{n}}(a)) \al_{\un{m}}(a_\la) b \in 
\ol{\al_{\un{n}+ \un{m}}(A) \al_{\un{m}}(A) A}.
\]
Commutativity shows that the family $\{ X_{\un{n}} \mid n \in \bZ_+^N\}$ forms a product system $X$.
In particular the left action on every $X_{\Bi}$ is by compacts and thus $X$ is strong compactly aligned.
Then the ideals $\I_F$ coincide with those given in \cite{DFK14} by
\[
\I_F = \bigcap_{\un{n} \perp F} \al_{\un{n}}^{-1} \Big(\big( \bigcap_{i \in F} \ker\al_\Bi \big)^\perp\Big) . 
\]
The representations in \cite{DFK14} are given by $(\pi, V)$ where $V$ is an isometric Nica-covariant representa\-tion of $\bZ_+^N$, $\pi(a) V_{\un{n}} = V_{\un{n}} \pi\al_{\un{n}}(a)$.
The CNP-representations are defined in \cite{DFK14} as those pairs $(\pi,V)$ that in addition satisfy
\[
\pi(a) \prod \{I - V_{\Bi} V_{\Bi}^* \mid i \in F\} = 0 \foral a \in \I_F.
\]
Therefore the representations of $\al \colon \bZ_+^N \to \End(A)$ are representations of the induced product system.
From this point on there is a small subtlety between the Nica-Pimsner alebras of $X$ and those in \cite{DFK14}.
To keep comparisons, the universal objects are denoted by $\N\T(A,\al)$ and $\N\O(A,\al)$ in \cite{DFK14}, and it is shown that they satisfy the GIUT.
However $\N\T(A,\al)$ is defined to be generated by $V_{\un{n}} \pi(a)$ for $a \in A$ rather than for $a \in \al_{\un{n}}(A) A$.
Hence in general we have 
\[
\N\T(X) \subseteq \N\T(A,\al) \qand \N\O(X) \subseteq \N\O(A,\al).
\]
Although the representations in \cite{DFK14} may not be all the Nica-covariant representations of $X$, they are enough to norm $\N\T(X)$ and $\N\O(X)$.
This is just an application of the GIUT.
Of course when the $\al_{\un{n}}$ are non-degenerate in the sense that $\alpha_{\un{n}}(A)A = A$, then $X_{\un{n}} = A$ and the C*-algebras coincide.

If, in addition, every $\al_{\Bi}$ is injective then $X$ is regular and $\N\O(A,\al) \simeq \O_X$.
In fact, the latter is a C*-crossed product over the minimal automorphic extension of $\al$ given by the direct limit process
\[
\xymatrix{
A \ar[rr]^{\al_{\un{n}}} \ar[d]^{\al_{\un{m}}} & & A \ar[d]^{\al_{\un{m}}} \ar@{.>}[rr] & & A_{\infty} \ar[d]^{\al_{\infty, \un{m}}} \\
A \ar[rr]^{\al_{\un{n}}} & & A \ar@{.>}[rr] & & A_{\infty}.
}
\]

\begin{example}
One-dimensional dynamical systems over commutative C*-algebras $C_0(X)$ are topological graphs in the sense of Katsura \cite{Kat03}.
More generally, a product system that arises from $\al \colon \bZ_+^N \to \End(C_0(X))$ coincides with the topological higher graph of Yeend \cite{Yee07}.
The reader is addressed to \cite[Section 5.3]{CLSV11} for an excellent exposition that connects topological graphs with their product systems.

By duality $\al$ transforms to a $\bZ_+^N$-action on $X$ by proper continuous maps.
Afsar, an Huef and Raeburn \cite{AHR17} considered product systems arising from surjective local homeomorphisms that $*$-commute.
In this case the induced $\O(X)$ of \cite{AHR17} coincides with $\O_X$ of \cite{Fow02} and $\N\O(A,\al)$ of \cite{DFK14, Kak15} as the induced $\al_{\Bi}$ are injective.
In fact they coincide with the C*-crossed product of the minimal automorphic extension.
\end{example}

\section{Higher rank graphs}\label{S:hrg}

We will require some terminology on higher rank graphs and their associated product systems. For more details the reader is addressed to \cite{RS03, RSY03, RSY04}.

Let $G= (V,E,r,s)$ be a directed graph, and partition the edge set $E=E_1 \cup \cdots \cup E_N$ such that each edge carries a unique colour from a selection of $N$ colours. 
Denote by $E^{\bullet}$ the collection of all paths in $G$. 
We may then define a multi-degree function $d \colon E^{\bullet} \rightarrow \bZ_+^N$ by $d(\lambda) = (n_1, \dots, n_N)$, where $n_i$ is the number of edges in $\lambda$ from $E_i$.

A \emph{higher rank $N$-structure} on $G$ is an equivalence relation $\sim$ on $E^{\bullet}$ such that for all $\lambda \in E^{\bullet}$ and $\un{m},\un{n} \in \bZ_+^N$ with $d(\lambda) = \un{m} + \un{n}$, there exist unique $\mu,\nu \in E^{\bullet}$ with $s(\lambda) = s(\nu)$ and $r(\lambda) = r(\mu)$, such that $d(\mu)=\un{n}$ and $d(\nu) = \un{m}$ and $\lambda \sim \mu \nu$. 
It is important to note that paths are read from right to left to comply with operator multiplication.
We denote $\La := E^{\bullet} / \sim$ and keep denoting by $d$ the induced multi-degree map on $\La$. 
It is usual to still denote by $\la, \mu$ etc. the elements of $\La$.
In this way the pair $(\La,d)$ is a \emph{higher rank graph} as in \cite[Definition 2.1]{RSY03}. 

For each $\un{n}\in \bZ_+^N$ we write $\La^{\un{n}} := \{\lambda \in \La  \mid d(\lambda) = \un{n} \}$.
For $\lambda \in \La$ and $\S \subseteq \La$, we define $\lambda \S := \{\lambda \mu \mid \mu\in \S \} $ and $\S \lambda := \{\mu \lambda \mid \mu \in \S \}$.
For $\la, \mu \in \La$ let
\[
\La^{\min}(\la, \mu) := \{(\al, \be) \mid \la \al = \mu \be, d(\la \al) = d(\la) \vee d(\mu) = d(\mu\be) \}
\]
be the set of \emph{minimal common extensions of $\la$ and $\mu$.} 
The higher rank graph $(\La,d)$ is called \emph{finitely aligned} if $|\La^{\min}(\la, \mu)|<\infty$ for any $\lambda,\mu \in \La$.
Given a vertex $v\in \La^{\un{0}}$, a subset $\S \subseteq v \La$ is called \emph{exhaustive} if for every $\la \in v\La$ there is $\mu \in \S$ such that $\La^{\min}(\la, \mu) \neq \emptyset$. 

A set of partial isometries $\{T_{\lambda}\}_{\lambda \in \La}$ for a finitely aligned higher rank graph $(\La,d)$ is called a \emph{Toeplitz-Cuntz-Krieger $\La$-family} if
\begin{enumerate}
\item[(P)]
$\{T_v\}_{v\in \La^{\un{0}}}$ is a collection of pairwise orthogonal projections;
\item[(HR)]
$T_{\la} T_{\mu} = \de_{s(\la), r(\mu)} T_{\la \mu}$ for all $\la, \mu \in \La$; and
\item[(NC)]
$T_{\la}^* T_{\mu} = \sum_{(\alpha,\beta)\in \La^{\min}(\la, \mu)} T_{\alpha} T_{\beta}^*$ for all $\la, \mu \in \La$.
\end{enumerate}
It is called a \emph{Cuntz-Krieger $\La$-family} if it additionally satisfies
\begin{enumerate}
\item[(CK)]
$\prod_{\lambda \in \S}(T_v - T_{\lambda} T_{\lambda}^*) = 0$ for every $v\in \La^{\un{0}}$ and all non-empty finite exhaustive sets $\S \subseteq v\La$.
\end{enumerate}
The C*-algebra $\ca(\La)$ is the universal one with respect to the Cuntz-Krieger $\La$-families.
In \cite[Theorem 4.2]{RSY04} it was proven that $\ca(\La)$ satisfies the GIUT theorem.

Every higher rank graph $(\La,d)$ has a natural product system $X(\La)$ associated to it. 
In short for each $\un{n} \in \bZ_+^N$ we put a pre-Hilbert $c_0(\La^{\un{0}})$-module structure on $c_0(\La^{\un{n}})$ via the formulas
\[
\sca{\xi, \eta}(v) := \sum_{s(\lambda) = v}\overline{\xi(\lambda)}\eta(\lambda) \qand (a\cdot \xi \cdot b)(\lambda) := a(r(\lambda))\xi(\lambda)b(s(\lambda)).
\]
The completion $X_{\un{n}}(\La)$ of these pre-Hilbert modules then gives a product system where the identification $X_{\un{m}}(\La) \otimes X_{\un{n}}(\La) \cong X_{\un{m}+\un{n}}(\La)$ is given by
\[
\delta_{\mu} \otimes \delta_{\nu} =
\begin{cases}
\delta_{\mu \nu} & \textup{ when $s(\mu) = r(\nu)$},\\
0 & \textup{ if $s(\mu) \neq r(\nu)$}.
\end{cases}
\] 
It is shown in \cite[Theorem 5.4]{RSY03} that $X(\La)$ is compactly aligned if and only if $\La$ is finitely aligned. 
In this case, the Nica-covariant representations $(\pi,t)$ of $X(\La)$ are in bijection with the Toeplitz-Cuntz-Krieger $\La$-families.
The correspondence is given by $T_{\lambda} = t_{\un{n}}(\delta_{\lambda})$.

\begin{proposition} \label{P:sca-graph}
Let $(\La,d)$ be a finitely aligned higher $N$-rank graph.
Then $X(\La)$ is strong compactly aligned if and only if for every $\lambda \in \La$ and every $\bo{i} \perp d(\la)$ there are finitely many edges $e \in d^{-1}(\bo{i})$ such that $\La^{\min}(\la, e) \neq \emptyset$. 
\end{proposition}

\begin{proof}
We identify each $\delta_{\nu}$ with $\nu \in \La$.
Since $X(\La)$ is compactly aligned, for $\la, \mu \in \La$ we have that
\[
\theta_{\la, \la} \vee \theta_{\mu, \mu} 
= 
\sum \{ \theta_{\la \al, \mu \be} \mid (\al, \be) \in \La^{\min}(\la, \mu) \} 
\]
is a finite sum.
On the other hand we have $I_{\L X_{\bo{i}}} = \textup{s*-}\sum_{\mu \in X_{\bo{i}}} \theta_{\mu, \mu}$, where the sum is taken in the s*-topology.
Therefore, if $d(\la) = \un{n}$ and $i\notin \supp \un{n}$ then
\begin{align*}
i_{\un{n}}^{\un{n} + \bo{i}}(\theta_{\la, \la}) 
& = 
\textup{s*-}\sum_{\mu \in X_{\bo{i}}} \theta_{\la, \la} \vee \theta_{\mu, \mu} 
 =
\textup{s*-}\sum_{\mu \in X_{\bo{i}}} \sum \{ \theta_{\la \al, \mu \be} \mid (\al, \be) \in \La^{\min}(\la, \mu) \}.
\end{align*}
If the latter is in $\K (X_{\un{n} + \bo{i}}(\La))$ then the sum has to be finite. 

Conversely, if $\lambda,\nu \in \La^{\un{n}}$ and the sum for $i_{\un{n}}^{\un{n} + \bo{i}}(\theta_{\la, \la})$ is finite, then each $i_{\un{n}}^{\un{n}+\bo{i}}(\theta_{\lambda,\mu}) = i_{\un{n}}^{\un{n}+\bo{i}}(\theta_{\lambda,\lambda}) i_{\un{n}}^{\un{n}+\bo{i}}(\theta_{\lambda,\mu})$ is given by a finite sum. 
Since elements of the form $\theta_{\la, \mu}$ span a dense set in $\K (X_{\un{n}}(\La))$ we get that $i_{\un{n}}^{\un{n}+\bo{i}}(k) \in \K (X_{\un{n}+\bo{i}}(\La))$ for every $k\in \K (X_{\un{n}}(\La))$ and $i\notin \supp \un{n}$. 
Hence $X(\La)$ is strong compactly aligned.
\end{proof}

We use the following terminology for the higher rank graphs that produce strong compactly aligned product systems.

\begin{definition}
A finitely aligned higher $N$-rank graph $(\La,d)$ is called \emph{strong finitely aligned} if for every $\lambda \in \La$ and every $\bo{i} \perp d(\la)$ there are finitely many edges $e \in d^{-1}(\bo{i})$ such that $\La^{\min}(\la, e) \neq \emptyset$. 
\end{definition}

\begin{example} \label{E:non-sca}
Proposition \ref{P:sca-graph} allows us to build examples of compactly aligned product systems that are not strong compactly aligned. 
Figure 1 yields such an example.
Here we let $\{g_n \mid n \in \bN\}$ be the edges with range $w$ whose sources are all distinct, by $\{h_n \mid n \in \bN\}$ the edges whose range is $v$ with sources all distinct, and by $\{f_n \mid n \in \bN\}$ the edges such that $s(f_n) = s(g_n)$ and $r(f_n) = s(h_n)$. 
If we make the identification $e g_n = h_n f_n$ we would obtain a higher $2$-rank graph.
Since $\La$ is finitely aligned then $X(\La)$ is compactly aligned. 
However, by Proposition \ref{P:sca-graph}, it is not strong compactly aligned as $|\La^{\min}(e,h_n)| = 1$ for all $n \in \bN$.
\begin{figure}[H]
\begin{align*}
\xymatrix@R=1em@C=4em{
\bullet \ar@{-2>}[dddr]^{h_3}_{\dots} & & & & & \bullet \ar@{.2>}[lllll]_{\vdots}^{f_3} \ar@{-2>}[dddl]_{g_3}^{\dots} \\
& \bullet \ar@{-2>}[dd]^{h_2} & & & \bullet \ar@{.2>}[lll]^{f_2}  \ar@{-2>}[dd]_{g_2} & \\
& & \bullet \ar@{-2>}[dl]^{h_1} & \bullet \ar@{.2>}[l]^{f_1} \ar@{-2>}[dr]_{g_1} & & \\
& \underset{v}{\bullet} &  & & \underset{w}{\bullet} \ar@{.2>}[lll]^{e} &
}
\end{align*}
\vspace{1pt}
\begin{quote}
{\footnotesize Figure 1. A finitely aligned higher rank graph whose product system is not strong compactly aligned.}
\end{quote}
\end{figure}
\end{example}

\begin{example}
On the other hand, Proposition \ref{P:sca-graph} also allows us to build examples of strong finitely aligned higher rank graphs that are not row-finite. First of all, it is clear that all rank one directed graphs are strong finitely aligned. 

If we want a non-trivial rank two example, the only thing we need to pay attention to is not to have edges $e$ that interact with infinitely many other edges $f$ that have different colours from $e$.
In Figure 2 we have such an example.
\begin{figure}[H]
\begin{align*}
\xymatrix@R=1em@C=4em{
\bullet \ar@{=2>}[dd]^{(\infty)} & & & \bullet \ar@{.2>}[lll]_{f_0}  \ar@{-2>}[dd]^{e_0} \\
& & & \\
\underset{v}{\bullet} &  & & \bullet \ar@{:2>}[lll]^{(\infty)}
}
\end{align*}
\vspace{1pt}
\begin{quote}
{\footnotesize Figure 2. A higher rank graph that is not row-finite but whose product system is strong compactly aligned.}
\end{quote}
\end{figure}

More precisely, let $\{e_i\}$ be infinitely many bold edges received by $v$ and $\{f_i\}$ be infinitely many dotted edges received by $v$.
Set the higher rank structure $e_i f_0 = f_i e_0$.
Then for a fixed $i$ the set $\Lambda^{\min}(e_i, f_j)$ is empty unless $j = i$, for which it is of size one. 
Thus by Proposition \ref{P:sca-graph} the graph is strong finitely aligned.
\end{example}

Our next goal is to show that the relation (CK) simplifies for strong finitely aligned higher rank graphs. 
To this end notice that the ideals $\J_F$ and $\I_F$ of $c_0(\La^{\un{0}})$ are generated by vertex projections $\delta_v$. 
By definition $\delta_v \in \J_F$ if and only if $v \La^{\bo{i}} \neq \emptyset$ for at least one $i\in F$, and $|v\La^{\bo{i}}| < \infty$ for all $i\in [N]$.
On the other hand the requirement that $\sca{X_{\un{m}}(\La), \delta_v X_{\un{m}}(\La)} \subseteq \J_F$ for $\un{m} \perp F$ is equivalent to requiring that any path $\lambda \in \La$ with $d(\lambda) \perp F$ and $r(\lambda) = v$ satisfies $\delta_{s(\lambda)} \in \J_F$. Note that the former is automatic given the latter, when $\lambda = v$ is a vertex. We gather this information in the following definition.

\begin{definition}
Let $(\La,d)$ be a strong finitely aligned $N$-rank graph. 
A vertex $v \in \La^{\un{0}}$ is called \emph{$F$-tracing} if for every path $\lambda \in \La$ with $d(\lambda) \perp F$ and $r(\la) = v$ we have $s(\lambda)\La^{\bo{i}} \neq \emptyset$ for at least one $i\in F$, and $|s(\lambda)\La^{\bo{i}}| < \infty$ for all $i\in [N]$.
\end{definition}

By definition, the $F$-tracing vertices span a dense subset of $\I_F$.
Thus by Proposition \ref{P:rem inv IF}, if $\la \in \La$ satisfies $d(\la) \perp F$ and $r(\la)$ is $F$-tracing, then $s(\la)$ is also $F$-tracing.
We give an alternative of the (CK)-condition for strong finitely aligned higher $N$-rank graphs, which we call (CK').
It appears that (CK') is simpler than (CK) as, essentially, it depends only on vertices.
We note that the proof of the following does not depend on the GIUT from \cite{RSY04}.

\begin{theorem}\label{T:cnp is ck}
Let $(\La,d)$ be a strong finitely aligned higher $N$-rank graph. 
Suppose that $(\pi,t)$ is a Nica-covariant representation of $X(\La)$ corresponding to a Toeplitz-Cuntz-Krieger $\La$-family $\{T_{\lambda}\}_{\lambda \in \La}$. 
Then the following are equivalent:
\begin{enumerate}
\item $(\pi,t)$ is a CNP-representation;
\item $\{T_\la\}_{\la \in \La}$ satisfies the \textup{(CK')}-condition:
$\prod \{ T_v - T_{\mu}T_{\mu}^* \mid \mu \in v\La^{\Bi}, i \in F\} = 0$ for every $F$-tracing vertex $v$ and every non-empty $F\subseteq [N]$. 
\item $\{T_\la\}_{\la \in \La}$ is a Cuntz-Krieger $\La$-family.
\end{enumerate}
Consequently, if $\{T_{\la}\}_{\la \in \La}$ satisfies \emph{(CK')}, admits a gauge action and $T_v\neq 0$ for $v\in \La^0$, then it defines a faithful representation of $\ca(\La)$.
\end{theorem}

\begin{proof}

[(i) $\Leftrightarrow$ (ii)]:
Fix $\mt \neq F\subseteq [N]$ and let $v$ be a vertex such that $\delta_v \in \I_F$. 
Then, for any $\un{n} \leq \un{1}_F$ we have
\[
\psi_{\un{n}}(\phi_{\un{n}}(\delta_v)) 
= 
\sum \{ T_{\lambda}T_{\lambda}^* \mid \lambda \in v\La^{\un{n}} \}.
\]
As the $F$-tracing vertices span a dense subset of $\I_F$ it suffices to show the second equality in
\begin{align*}
\sum \{ (-1)^{|\un{n}|} \psi_{\un{n}}(\phi_{\un{n}}(\delta_v)) \mid \un{n} \leq \un{1}_F \}
& =
\sum \{ (-1)^{|d(\lambda)|} T_{\lambda}T_{\lambda}^* \mid \lambda \in v \La^{\un{1}_F} \}  \\
& = 
\prod \{T_v - T_{\mu}T_{\mu}^* \mid \mu \in v\La^{\Bi}, i \in F \}.
\end{align*}
To this end we shall use that $T_v - T_\mu T_\mu^* = T_v (I - T_\mu T_\mu^*)$ if $r(\mu) = v$ and that
\[
\prod \{T_v - T_{\mu}T_{\mu}^* \mid \mu \in v\La^{\Bi}, i \in F \} 
= 
\prod \{T_v - \sum_{\mu\in v\La^{\bo{i}}}T_{\mu}T_{\mu}^* \mid i \in F \}.
\]
Let $\mu \in v \La^{\un{1}_{H}}$ and $\nu \in v \La^{\un{1}_{H'}}$ for $H \neq H'$ subsets of $F$. 
By \cite[Lemma 2.7]{RSY04} we see that
\begin{align} \label{eq:multiply}
T_{\mu}T_{\mu}^*T_{\nu}T_{\nu}^* = \sum_{(\alpha,\beta) \in \La^{\min}(\mu,\nu)}T_{\mu\alpha}T_{\mu\alpha}^*.
\end{align}
Due to the unique factorization property, no element on the right hand side of equation \eqref{eq:multiply} will appear for distinct pairs $(\mu,\nu), (\mu',\nu') \in v \La^{\un{1}_H} \times v \La^{\un{1}_{H'}}$.
Hence a repeated application of equation \eqref{eq:multiply} yields
\[
\prod \{T_v - \sum_{\mu\in v\La^{\bo{i}}}T_{\mu}T_{\mu}^* \mid, i \in F \} 
= 
\sum \{ (-1)^{|d(\lambda)|} T_{\lambda}T_{\lambda}^* \mid \lambda \in v \La^{\un{1}_F} \}.
\]

\noindent
[(ii) $\Leftrightarrow$ (iii)]:
Let $\mt \neq F\subseteq [N]$ and let $v$ be an $F$-tracing vertex. 
First we show that the finite set $\{\mu \in v\La^{\Bi} \mid i \in F\}$ is exhaustive.
To this end let $\lambda \in v \La$ be some path. 
If $d(\lambda) \cap F \neq \emptyset$, then by unique factorization we may write $\lambda = \mu \nu$ where $\mu \in v\La^{\Bi}$ for some $i \in F$.
Trivially $(v, \nu)$ is a common extension of $\la$ and $\mu$ so that $\La^{\min}(\la, \mu) \neq \emptyset$.
Otherwise, if $d(\lambda) \cap F = \emptyset$, as $v$ is $F$-tracing, there exists some $i\in F$ such that $s(\lambda)\La^{\bo{i}} \neq \emptyset$.
Hence we may find $e \in \La^{\bo{i}}$ for which $\lambda e$ is a path in $v \La$.
Notice that $d(\la e) = d(\la) + \Bi$.
By unique factorization we can find a $\mu \in \La^{\bo{i}}$ such that $\la e = \mu \la'$ for some $\la' \in \La$.
Therefore there is a $\mu \in v\La^{\Bi}$ such that $\La^{\min}(\la, \mu) \neq \emptyset$.

Since $\{\mu \in v\La^{\Bi} \mid i \in F\}$ is finite and exhaustive, if $\{T_{\la}\}_{\la \in \La}$ is a Cuntz-Krieger $\La$-family then
\[
\prod \{T_v - T_{\mu}T_{\mu}^* \mid \mu \in v\La^{\Bi}, i \in F \} = 0,
\]
so that $(\pi,t)$ is a CNP-representation by the first part of the proof.
This shows that there exists a canonical $*$-epimorphism $\N\O(X) \rightarrow \ca(\La)$ that maps generators to generators, admits a gauge action and is injective on $c_0(\La^0)$.
Since $\N\O(X)$ satisfies a GIUT, this $*$-epimorphism must be injective.
Universality of the C*-algebras completes the proof.
\end{proof}

\begin{remark}
While the current paper was under submission, an interesting example of row-finite graphs was investigated by an Huef-Raeburn \cite{HR19}.
Therein they work out the (CK) relations of \cite{RSY03} on exhaustive sets to derive simpler relations that define the same Cuntz-Krieger algebra. Their main arguments use alternating sums as we have incorporated in this work, and in \cite[Example 1.8]{HR19} our (CK') relations can be used to obtain \cite[Corollary 1.11]{HR19} with a simpler analysis.
\end{remark}

\section{A case study: higher rank factorial languages}\label{S:mfl}

In this section we examine C*-algebras arising from factorial languages. 
Several elements pass through from the rank one case studied in \cite{BK17, KS15, LM95} and we will omit their proofs.
We require some notation.
The free semigroup $\bF_+^d$ on $d$ symbols admits the partial order
\[
\mu \geq \nu \qiff \mu = \nu w \text{ for some } w \in \bF_+^d.
\]
Consider the cartesian product $\fdn$ of $N$ free semigroups for fixed $N \in \bZ_+$.
We shall denote elements in $\fdn$ by $\umu, \unu$ etc., and we fix $\umt := (\mt, \dots, \mt)$.
In particular we shall write $\de_i(k)$ for the generator with the letter $k$ at the $i$-th coordinate.
The \emph{multilength} of $\umu = (\mu_1, \dots, \mu_N)$ is given by $|\umu| := (|\mu_1|, \dots, |\mu_N|)$.
We use $\ast$ for the coordinate-wise multiplication operation in $\fdn$, i.e.
\[
\umu \ast \unu = (\mu_1 \nu_1, \dots, \mu_N \nu_N).
\]
The semigroup $(\bF_+^d)^N$ inherits a partial ordering, given by
\[
\umu \geq \unu \qiff \mu_i \geq \nu_i \foral i \in [N].
\]
We write $\umu \vee \unu$ as the least upper bound for $\umu$ and $\unu$, when it exists.
We extend the definition of the support for elements in $\fdn$ in the sense that
\[
\supp \umu := \{i \in [N] \mid \mu_i \neq \mt\}
\]
and we write $\umu \perp \unu$ if and only if $\supp \umu \bigcap \supp \unu = \mt$.
We write $\unu \in \umu$ if $\umu = \uw \ast \unu \ast \uq$ for some $\uw, \uq \in \fdn$.

\begin{definition}
A \emph{factorial language} (FL) $\La^*$ is a subset of $(\bF_+^d)^N$ such that:
\begin{enumerate}
\item for every $i \in [N]$ there exists at least one $k \in [d]$ such that $\de_i(k) \in \La^*$;
\item if $\umu \in \La^*$ and $\unu \in \umu$ then $\unu \in \La^*$.
\end{enumerate}
\end{definition}

\begin{remark}
We allow cases where some $\de_i(k)$ are not in $\La^*$.
This is convenient for including cases of $\La^* \subseteq \bF_+^{d_1} \times \dots \times \bF_+^{d_N}$, for different $d_i$'s, under the one umbrella of $\fdn$ for $d = \max_i \{d_i\}$.
\end{remark}

Just as in the rank one case \cite{KS15}, a higher $N$-rank language $\La^*$ can be characterized by forbidden words. 
Indeed, when $\La^*$ is an FL, the \emph{whole} set of forbidden words $\fdn \setminus \La^*$ becomes a monoid in $\fdn$. 
On the other hand, let $\sca{\F}$ be the monoid in $\fdn$ generated by a subset $\F \subseteq \fdn$.
Then the set
\[
\La^*_\F:= \fdn \setminus \sca{\F}
\]
is a higher $N$-rank language. 
If $\F$ is finite then we say that $\La^*_\F$ is of \emph{finite type}.

Products of rank one FL's form higher rank FL's. 
In fact it is not hard to see that $\La^*$ is a product of one dimensional languages if and only if whenever $\umu, \unu \in \La^*$ are such that $\umu \perp \unu$, then $\umu * \unu \in \La^*$.
This follows by the fact that coordinate projections of FL's are rank one FL's.
Hence, if we forbid words $\umu = (\mu_1, \dots, \mu_N)$ that are supported on more than one coordinates, while allowing words of the form $\delta_i(\mu_i)$, we obtain FL's that are not products of lower dimensional languages. 
It is now clear that many such examples can be constructed.

\begin{example}\label{E: doubledoors}
Let $\La^* := \{(0^n,0^m), (0^k 1 0^l, 0^m), (0^n, 0^k 10^l) \mid n,m,k,l \in \bZ_+\}$.
Then $\La^* = \La^*_\F$ for 
\[
\F=\{ \ (1,1), (10^n1, \mt), (\mt, 10^n1) \ | \  n\geq 0 \ \}.
\] 
If $\La^*$ were a product of two rank one languages, then by projecting to each coordinate and multiplying we would have that $(0^k 1 0^l, 0^{k'} 10^{l'}) \in \La^*$. 
This is a contradiction as the word contains $(1,1) \in \F$.
\end{example}

For $\un{n} \in \bZ_+^N$ we define the equivalence relation $\sim_{\un{n}}$ on $\La^*$ such that $\umu \sim_{\un{n}} \unu$ if and only if
\[
\{\un{w} \mid \un{w} \ast \umu \in \La^*, |\un{w}| \leq |\un{n}| \} = \{\un{w} \mid \un{w} \ast \unu \in \La^*, |\un{w}| \leq |\un{n}| \}.
\]
We set $\Om_{\un{n}} := \La^*/ \sim_{\un{n}}$, which is finite.
Using the connecting maps 
\[
\vartheta_{\un{n} + \un{m}, \un{n}} \colon \Om_{\un{n} + \un{m}} \to \Om_{\un{n}}: [\umu]_{\un{n} + \un{m}} \mapsto [\umu]_{\un{n}} 
\]
we define the projective limit
\[
\Om := \lim_{\longleftarrow} (\Om_{\un{n}}, \vartheta_{\un{n} + \un{m}, \un{n}})
\]
which is a totally disconnected space.
In fact $\Om \simeq \La^*/\sim$ for the equivalence $\sim$ given by
\[
\umu \sim \unu \qiff \{\un{w} \mid \un{w} \ast \umu \in \La^* \} = \{\un{w} \mid \un{w} \ast \unu \in \La^* \}.
\]
We denote by $[\umu]$ the equivalence class of $\umu$ under $\sim$.

\begin{definition}
An FL $\La^*$ is called \emph{sofic} if and only if there exists an $\un{m} \in \bZ_+^N$ such that $\sim_{\un{n}}$ is equal to $\sim$ for all $\un{n} \geq \un{m}$.
Equivalently, if $\Om$ is finite. 
\end{definition}

This terminology is in accordance with the rank one case \cite{LM95}. 
Moreover we have the following proposition.

\begin{proposition}
Every FL of finite type is sofic.
More precisely, if $\La^* = \La^*_{\F}$ for some finite set $\F$, then $\Om = \Om_{\un{n}}$ for all $\un{n} \geq \max \{|\umu| \mid \umu \in \F\}$.
\end{proposition}

\begin{proof}
Set $\un{m} = \max\{ \ |\umu| \ | \  \umu\in \F \ \}$. Let $\umu \in \La^*$ such that $|\umu| \not\leq \un{m}$.
Without loss of generality we may assume that $\umu = (\mu_1, \dots, \mu_\ell, \mu_{\ell+1}, \dots, \mu_N)$ with $|\mu_i| > m_i$ for $i \in \{1, \dots, \ell\}$ and $|\mu_j| \leq m_j$ for $j \in \{\ell+1, \dots, N\}$.
Write $\umu = \umu' \ast \umu''$ such that $|\mu'_i| = m_i$ for $i \in\{1, \dots, \ell\}$ and $\mu_j' = \mu_j$ for $j \in \{\ell+1, \dots, N\}$, i.e.
\[
\umu' = (\mu_1', \dots, \mu_{\ell}', \mu_{\ell+1}, \dots, \mu_N) \qand \umu'' = (\mu_1'', \dots, \mu''_\ell, \mt, \dots, \mt).
\]
Then $|\umu'| \leq \un{m}$, and for every $\un{w} \in \La^*$ we claim that
\[
\un{w} \ast \umu' \in \La^* \qiff \un{w} \ast \umu \in \La^*.
\]
This shows that $\Om = \Om_{\un{m}}$ as it suffices to identify just the classes $[\umu]$ with $|\umu| \leq \un{m}$; consequently $\Om = \Om_{\un{n}}$ for any $\un{n} \geq \un{m}$.
To prove the claim, if $\un{w} \ast \umu \in \La^*$ then trivially $\un{w} \ast \umu' \in \La^*$.
Conversely suppose that $\un{w} \ast \umu' \in \La^*$ but $\un{w} \ast \umu \notin \La^*$.
Then there is a forbidden word $\unu$ in $\un{w} \ast \umu$ with $|\unu| \leq \un{m}$.
We write $\unu = \un{x} \ast \un{y}$ so that $\un{w} = \un{x}' \ast \un{x}$ and $\umu = \un{y} \ast \un{y}'$.
Hence $|y_i| \leq |\mu_i|$ and $|y_i| \leq |\nu_i| \leq m_i$ for all $i \in [N]$.
Since $\unu$ cannot be a subword of $\un{w} \ast \umu' \in \La^*$ there exists an $i \in \{1,...,\ell\}$ such that $y_i$ contains $\mu'_i$ as a subword.
Then we would have that $|y_i| > |\mu_i'| = m_i$,
so we have a contradiction. Therefore $\un{w} \ast \umu \in \La^*$.
\end{proof}

There is a way of constructing a sofic FL from a finite labeled higher rank graph $(\La,\fL)$.
Suppose we have partitioned its edge set $E = E_1 \cup \cdots \cup E_N$, and we have a labeling map $\fL \colon E \to \fdn$ such that $\fL(e_i) \in \{\de_i(k) \mid k \in [d]\}$.
We extend $\fL$ to a labeling from $E^{\bullet}$ to finite multi-words in $(\bF_+^d)^N$, which we still denote by $\fL$. 
Clearly $\fL$ respects the commutation imposed by $\La$ as a quotient of $E^{\bullet}$, and we get that $\fL$ is well-defined on $\La$. 
This way, $\fL(E^\bullet)$ naturally becomes an FL and satisfies $|\fL(\mu)| = d(\mu)$ for every path $\mu \in \La$. It goes without saying that different edges may carry the same label and colour.

We next construct a finite labeled higher rank graph $(\La,\fL)$ from a sofic FL $\La^*$ in $(\bF_+^d)^N$. We know that $\Omega$ is finite, and we may consider the finite colored graph $\La$ whose vertices are $\Omega$, and edges $E$ where $e$ is an edge with color $i$ and labeled $\fL(e):=\delta_i(k) \in (\bF_+^d)^N$ from $[\umu]$ to $[\delta_i(k) * \umu]$ if and only if $\delta_i(k) * \umu \in \La^*$. 
We denote $\sim$ the equivalence relation generated by $ef \sim f'e'$ if $\fL(e) \ast \fL(f) = \fL(f') \ast \fL(e')$. Then $\La = E^{\bullet} / \sim$ is a higher rank graph and $\fL$ extends to a well-defined labeling on $\La$.

Indeed if $ef$ is a path with source $[\umu]$ with colors $i\neq j$ and labels $\fL(e):= \delta_i(k)$ and $\fL(f):=\delta_j(\ell)$, then its range is $[\de_i(k) \ast \de_j(\ell) \ast \umu] = [\de_j(\ell) \ast \de_i(k) \ast \umu]$.
Therefore $\de_j(\ell) \ast \de_i(k) \ast \umu$ is also in $\La^*$, and as $i \neq j$ there exists a path $f'e'$ from $[\umu]$ to $[\de_j(\ell) \ast \de_i(k) \ast \umu]$ through $[\de_i(k) \ast \umu]$ so that $\fL(f') = \delta_j(\ell)$ and $\fL(e') = \delta_i(k)$.

\begin{definition} \label{D:fsg}
Let $\La^*$ be an FL in $(\bF_+^d)^N$. Then the labeled higher rank graph $(\La,\fL)$ constructed above is called the \emph{follower set graph} of $\La^*$.
\end{definition}

An important property of the follower set graph is that it is \emph{source-resolving}. More precisely, if $\la,\mu \in \La$ are paths emanating from the same source $[\umu]$, then $\fL(\la) = \fL(\mu)$ implies that $\la = \mu$. Indeed, we may write $\la = \la_1 ... \la_N$ and $\mu = \mu_1 ... \mu_N$, where $\la_i , \mu_i$ are comprised only of edges of color $i$. So in order to show that $\la = \mu$, it will suffice to show that $\la_i = \mu_i$ for each $i\in [N]$ when $\fL(\la_i) = \fL(\mu_i)$. This can then be verified by induction on length, where we use the fact that for two edges $e,f$ of the same color $i$ and source $[\umu]$, we have that $\fL(e) = \fL(f)$ implies $e=f$.

\begin{proposition}\label{P:fsg}
Let $\La^*$ be an FL in $(\bF_+^d)^N$. 
If $\La^*$ is sofic then it coincides with the labeled path space of its follower set graph $(\La,\fL)$.
Conversely if $\La^*$ coincides with the labeled path space of a finite labeled higher rank graph then it is sofic.
\end{proposition}

\begin{proof}
First suppose that  $\La^*$ is sofic and let $(\La,\fL)$ be its follower set graph of $\La^*$. 
We need to show that $\La^*$ coincides with the labeled path space $\fL(\La)$.
It is clear that if $\umu \in \La^*$ then it defines a path $x$ from $[\mt]$ to $[\umu]$ so that $\fL(x) = \umu$.
For the other inclusion, let $x \in \La$.
If $s(x) = [\unu]$ then by construction we get that $r(x) = [\fL(x) \ast \unu]$.
Hence $\fL(x) \ast \unu \in \La^*$ and so we obtain $\fL(x) \in \La^*$. 

Suppose now that $\La^*$ comes from a labeled higher $N$-rank graph $(\La,\fL)$ with finite vertices.
We show that $\Om = \La^*/\sim$ is finite.
To this end let $\umu \in \La^* = \fL(E^\bullet)$ and take all paths representing $\umu$.
Each one can be uniquely written as $x = x_1 \cdots x_N$ with $x_i$ composed only of edges from $E_i$.
Let
\[
C_{\umu} := \{r(x) \mid \fL(x) = \umu \}.
\]
Then $\{\un{w} \mid \un{w} \ast \umu \in \La^*\}$ contains all the labeled paths starting at some vertex in $C_{\umu}$.
Therefore if $[\umu] \neq [\unu]$ then $C_{\umu} \neq C_{\unu}$.
Thus the cardinality of $\Om$ is at most $2^{|V|}$, and hence finite.
\end{proof}

\begin{example}
Referring to the FL $\La^*$ in Example \ref{E: doubledoors} given by
$$
\La^* = \{(0^n,0^m), (0^k 1 0^l, 0^m), (0^n, 0^k 10^l) \mid n,m,k,l \in \bZ_+\},
$$ 
the follower set graph is given in Figure 3.
\begin{figure}[H]
\begin{align*}
\xymatrix@C=2cm@R=.5cm{
{}^{[(\mt,1)]}
\ar@{-2>}^<<<<<{(0, \mt)}@(l,u) 
\ar@{.2>}^<<<<<{(\mt, 0)}@(r,d) 
& & 
{}^{[(\mt,\mt)]}
\ar@{-2>}^<<<<<{(0,\mt)}@(l,u) 
\ar@{.2>}^<<<<<{(\mt,0)}@(r,d) 
\ar@{-2>}_>>>>>{(1,\mt) \hspace{2cm}}@/_.1pc/@(u,dl)[rr] 
\ar@{.2>}_>>>>>{\hspace{2cm} (\mt, 1)}@/_.1pc/@(d,ur)[ll]
& & 
{}^{[(1,\mt)]}
\ar@{-2>}^<<<<<{(0,\mt)}@(l,u) 
\ar@{.2>}^<<<<<{(\mt,0)}@(r,d) \\
& & & &
}
\end{align*}
\vspace{1pt}
\begin{center}
{\footnotesize Figure 3. Follower set graph for Example \ref{E: doubledoors}.}
\end{center}
\end{figure}
\end{example}

\subsection{Higher rank subshifts}

Next, we describe FL's that come from subshifts. 
This follows similarly to the rank one case \cite{LM95}, and we omit the details. 
In short, fix the symbol set $\Si=\{1, \dots, d\}$ and create the full $N$-rank shift
\[
\A_N := \prod_{l \in [N]} \Si^\bZ = \Si^{\bZ} \times \cdots \times \Si^{\bZ}
\]
as the product of the full rank one shifts.
For $\un{x} \in \A_N$ we write $x_i \in \Si^\bZ$ for the $i$-th coordinate and define the higher rank block
\[
\un{x}_{[\un{n}, \un{m}]} := (x_{1, [n_1, m_1]}, \dots, x_{N, [n_N, m_N]}) \in \fdn
\]
whenever $\un{n} \leq \un{m}$.
A word $\un{w} \in \fdn$ \emph{appears} in $\un{x} \in \A_N$ if $w_i$ appears in $x_i$ for every $i \in [N]$. 

\begin{definition}
Let $\F$ be a set of words in $\fdn$.
The \emph{subshift} $X_\F$ is defined by
\[
X_\F := \{\un{x} \in \A_N \mid \textup{ no $\un{w} \in \F$ appears in $\un{x}$}\}.
\]
We write 
\[
\B(X_\F) := \{ \un{w} \in \fdn \mid \exists \, \un{x} \in X_\F \textup{ such that $\un{w}$ appears in $\un{w}$} \}
\]
for the \emph{language} of a subshift $X_F$.
\end{definition}

When $X_{\F} \neq \emptyset$ it is clear that $\B(X_\F)$ is an FL that satisfies the property: for all $i\in [N]$ and $\umu \in \B(X_\F)$ there are $\delta_i(k),\delta_i(\ell) \in \B(X_\F)$ such that $\delta_i(k) * \umu * \delta_i(\ell) \in \B(X_F)$.
Conversely, every $\F \subseteq \fdn$ defines an FL $\La^*_\F$.
Then $\La^*_\F$ is the language of $X_\F$ if and only if $\La^*_\F$ satisfies the aforementioned property.

\begin{example}
Example \ref{E: doubledoors} is an FL of a higher rank subshift. 
Indeed, we may always append $(0,\mt)$ or $(\mt, 0)$ on either side of a word.
\end{example}

As in the rank one case, there exists an equivalent characterization of higher rank subshifts.
Recall that $\Si^\bZ$ is a compact metrizable space and thus so is $\A_N$.
We use the left shift $\si$ on $\Si^\bZ$ to define shifts on $\A_N$ along any of the directions, i.e.
\[
(\si_{\Bi}(\un{x}))_j :=
\begin{cases}
\si(x_i) & \textup{ if } j=i,\\
x_i & \textup{ if } j \neq i.
\end{cases}
\]
We will call the collection $\si:=\{\si_{\bo{1}}, \dots, \si_{\bo{N}}\}$ the \emph{full multi-shift} on $\A_N$.
It can be shown that $X \subseteq \A_N$ is a subshift if and only if $X$ is compact and $X = \si_{\Bi}(X)$ for all $i \in [N]$. The proof follows the same lines as in \cite[Theorem 6.1.21]{LM95}. Hence, we will often just write $X$ for a subshift without always specifying the underlying set of forbidden words $\F$.

\subsection{Encoding languages}

Our next goal is to prescribe operators that encode an FL, given by appending words. 
This is the multivariable analogue of what is done in \cite{Mat97, SS09, KS15}. 
For $\umu \in \fdn$ let the operator $T_{\umu} \in \B(\ell^2(\La^*))$ be given by
\[
T_{\umu} e_{\uw} 
:= 
\begin{cases}
e_{\umu \ast \uw} & \text{ if } \umu \ast \uw \in \La^*, \\
0 & \text{ otherwise}.
\end{cases}
\]
The mapping $T \colon \fdn \to \B(\ell^2(\La^*))$ defines a semigroup homomorphism such that $T_{\umu} = 0$ if and only if $\umu \notin \La^*$.
In particular every $T_{\umu}$ is a partial isometry with $T_{\umu}^*T_{\umu}$ and $T_{\umu}T_{\umu}^*$ the projections onto
\[
\ol{\spn}\{ e_{\uw} \mid \umu \ast \uw \in \La^*\}
\qand
\ol{\spn}\{ e_{\uw} \mid \umu \leq \uw, \uw \in \La^*\} \ \ \ \text{respectively}.
\]
Consequently every $T_{\umu}^* T_{\umu}$ commutes with every $T_{\unu}^* T_{\unu}$ and every $T_{\unu} T_{\unu}^*$.
In fact $T \colon \La^* \to \B(\ell^2(\La^*))$ is \emph{Nica-covariant} in the sense that
\begin{equation*}
T_{\umu} T_{\umu}^* T_{\unu} T_{\unu}^*
=
\begin{cases}
T_{\umu \vee \unu} T_{\umu \vee \unu}^* & \text{ if } \umu \vee \unu \in \La^*, \\
0 & \text{ otherwise}.
\end{cases}
\end{equation*}
Although $\La^*$ is not a semigroup, the relation above resembles the Nica-covariance of \cite{Nic92}, whenever multiplication is still within $\La^*$.
It is clear that the partial isometries $T_{\de_i(k)}$ and $T_{\de_j(\ell)}$ may fail to have orthogonal ranges when $i \neq j$, as $\de_i(k) \vee \de_j(\ell) = \de_i(k) \ast \de_j(\ell)$ may be in $\La^*$.

\begin{definition}
Let $\La^*$ be an FL.
Then its \emph{C*-algebra of checkers} $A$ is the unital commutative AF-algebra generated by the projections $T_{\umu}^* T_{\umu}$ with $\umu \in \La^*$, i.e., $A = \ol{\bigcup_{\un{n}} A_{\un{n}}}$ for
\[
A_{\un{n}} := \ca(T_{\umu}^* T_{\umu} \mid |\umu| \leq \un{n}).
\]
\end{definition}

Since every $T_{\de_i(k)} T_{\de_i(k)}^*$ commutes with $A$ we can define the (non-unital) $*$-endomor\-phisms of $A$
\[
\al_{\de_i(k)} \colon A \to A: a \mapsto T_{\de_i(k)}^* a T_{\de_i(k)}
\]
By iterating we obtain a partial anti-homo\-morphism $\al \colon \La^* \to \End(A)$, so that the identity $a T_{\umu} = T_{\umu} \al_{\umu}(a)$ is satisfied for all $a \in A$ and $\umu \in \La^*$. 
We call this partial anti-homo\-morphism $\al \colon \La^* \to \End(A)$ the \emph{quantized dynamics} of $\La^*$. 

\begin{proposition}
Let $\La^*$ be an FL, and let $\Om = \La^* / \sim$ and $\Om_{\un{n}} = \La^*/\sim_{\un{n}}$ for $\un{n} \in \bZ_+^N$. Then $A_{\un{n}} \simeq C(\Om_{\un{n}})$ for every $\un{n} \in \bZ_+^N$, so that $A \simeq C(\Om)$. 
Consequently, we have that
\begin{enumerate}
\item
$\La^*$ is sofic if and only if there is an $\un{m} \in \bZ_+^N$ such that $A = A_{\un{n}}$ for all $\un{n} \geq \un{m}$.
\item
Every $\al_{\un{w}} \in \End(A)$ induces a partially defined map $\varphi_{\un{w}}$ on $\Om$ by $\varphi_{\un{w}}([\umu]) = [\un{w} \ast \umu]$, with domain $\{[\umu] \mid \un{w} \ast \umu \in \La^*\}$.
\end{enumerate}
\end{proposition}

\begin{proof}
Suppose that $\{\umu_1, \dots, \umu_M\}$ is the set of words in $\La^*$ with length less than or equal to $\un{n}$. Then $A_{\un{n}}$ is the linear span of its minimal projections, given by
\[
\prod_{i =1}^r T_{\umu_{\si(i)}}^* T_{\umu_{\si(i)}} \cdot \prod_{j= r+1}^{M} (I - T_{\umu_{\si(j)}}^* T_{\umu_{\si(j)}})
\]
for permutations $\si \in S_M$ and $r \in \{1, \dots, M\}$.
A partition of $\{\umu_1, \dots, \umu_M\}$ into two disjoint sets $\{\umu_{\si(1)}, \dots, \umu_{\si(r)}\}$ and $\{\umu_{\si(r+1)}, \dots, \umu_{\si(M)}\}$ defines a unique $[\unu]_{\un{n}} = \{\mu_{\sigma(1)},...,\mu_{\sigma(r)}\}$.
Then the map $\varphi_{\un{n}}$ sending a minimal projection to the characteristic function on this $[\unu]_{\un{n}}$ gives the required $*$-isomorphism from $A_{\un{n}}$ onto  $C(\Om_{\un{n}})$.
For the second part, let $\ga_{\un{n}, \un{n} + \un{m}} \colon C(\Om_{\un{n}}) \to C(\Om_{\un{n} + \un{m}})$ be the induced map from $\vartheta_{\un{n} + \un{m}, \un{n}} \colon \Om_{\un{n} + \un{m}} \to \Om_{\un{n}}$.
We have to show that the diagram
\[
\xymatrix{
A_{\un{n}} \ar[rr]^{\id} \ar[d]^{\varphi_{\un{n}}} & & A_{\un{n} + \un{m}} \ar[d]^{\varphi_{\un{n} + \un{m}}} \\
C(\Om_{\un{n}}) \ar[rr]^{\ga_{\un{n}, \un{n} + \un{m}}} & & C(\Om_{\un{n} + \un{m}})
}
\]
is commutative.
On the one hand notice that $\ga_{\un{n}, \un{n} + \un{m}}$ maps $\chi_{[\unu]_{\un{n}}}$ to the characteristic function on the set $\{[\un{w}]_{\un{n} + \un{m}} \mid [\un{w}]_{\un{n}} = [\unu]_{\un{n}} \}$. 
Suppose without loss of generality that $\sigma = \id$ so that
\[
\varphi_{\un{n}}^{-1}(\chi_{[\unu]_{\un{n}}}) = \prod_{i =1}^r T_{\umu_{i}}^* T_{\umu_{i}} \cdot \prod_{j= r+1}^{M} (I - T_{\umu_{j}}^* T_{\umu_{j}}) =: a.
\]
Then $a$ as an element in $A_{\un{n} + \un{m}}$ is written as a linear sum of minimal projections over words of length $\un{n} + \un{m}$.
Suppose that $\{\umu'_1, \dots, \umu'_{M'}\}$ is this set.
Then the minimal projections in this linear combination of $a$ must satisfy $T_{\umu'_s}^* T_{\umu_s'} \leq T_{\umu_i}^* T_{\umu_i}$ for all $i=1, \dots, r$ and $T_{\umu'_s}^* T_{\umu_s} \geq T_{\umu_j}^* T_{\umu_j}$ for all $j= r + 1, \dots, M$.
This shows that the element $a$ is written as the sum of the characteristic functions on $[\un{w}]_{\un{n} + \un{m}}$ with $[\un{w}]_{\un{n}} = [\unu]_{\un{n}}$.
\end{proof}

\begin{remark}\label{R:fsg}
In the sofic case, the duality between $(\Om, \varphi)$ and $(A,\al)$ is given from the follower set graph.
First of all there is a bijection between the minimal projections of $A$ and the vertices of the follower set graph $(\La,\fL)$.
Hence, if $p_v$ corresponds to a vertex $v$ in $\La$ and $\mu \in \La$ is a path in $(\La,\fL)$, then
\[
\al_{\fL(\mu)}(p_v)
=
\sumoplus \{ p_{s(\nu)} \mid r(\nu) = v \AND \fL(\nu) = \fL(\mu)\}.
\]
Moreover $T_{\umu}^* T_{\umu}$ is the sum of all the minimal projections $p_v$ such that $v = s(\mu)$ for any $\mu$ satisfying $\fL(\mu) = \umu$.
\end{remark}

\subsection{Product system construction}

By using the Nica-covariance of $T$ and the algebra of checkers $A$ we can write
\[
\ca(T) = \ol{\spn} \{ T_{\umu} \, a \, T_{\unu}^* \mid  a \in A, \umu, \unu \in \La^*\}.
\]
The use of $A$ is necessary so as to obtain this special form of $\ca(T)$. 
More precisely, the linear span of the monomials $T_{\umu} T_{\unu}^*$ is not large enough to form an algebra that densely spans $\ca(T)$. 
Nevertheless this form of $\ca(T)$ allows us to identify a product system related to $\La^*$. 
For every $i \in [N]$ we define concretely the C*-correspondence
\[
X_{\Bi}(\La^*) := \ol{\spn} \{ T_{\de_i(k)} a \mid a \in A, \de_i(k) \in \La^*, k \in [d]\}
\]
with the inner product and bimobule structure given by
\[
\sca{\xi, \eta} := \xi^* \eta \qand \phi_{\Bi}(a) \xi \cdot b := a \xi b
\quad
\foral \xi ,\eta \in X_{\Bi}(\La^*), a, b \in A.
\]
Since all $X_{\Bi}(\La^*)$ are concretely defined in $\B(\ell^2(\La^*))$ we derive
\[
X_{\Bi}(\La^*) \otimes_A X_{\bo{j}}(\La^*) \cong \ol{[X_{\Bi}(\La^*) \cdot X_{\bo{j}}(\La^*)]}.
\]
However we have that
\[
T_{\de_i(k)} a T_{\de_j(l)} b 
= 
T_{\de_j(l)} T_{\de_i(k)} \al_{\delta_j(l)}(a) b 
\in \ol{[X_{\bo{j}}(\La^*) \cdot X_{\bo{i}}(\La^*)]}.
\]
Since $A$ is unital, and by symmetry, we get $\ol{[X_{\Bi}(\La^*) \cdot X_{\bo{j}}(\La^*)]} = \ol{[X_{\bo{j}}(\La^*) \cdot X_{\bo{i}}(\La^*)]}$, which coincides with the closed linear span of $\{ T_{\de_i(k) * \de_j(l)} a \mid a \in A\}$. 
The latter follows from the equality
\[
T_{\de_i(k) * \de_j(l)} a = T_{\de_i(k)}T_{\de_i(k)}^*T_{\de_i(k)} T_{\de_j(l)} a = T_{\de_i(k)} 1_A T_{\de_j(l)} a \in X_{\bo{i}}(\La^*) \cdot X_{\bo{j}}(\La^*),
\]
as $A$ is unital and $T_{\de_i(k)}$ is a partial isometry. 
By iterating, it is easy to verify the following proposition.

\begin{proposition}
Let $\La^*$ be an FL in $\fdn$.
Then the family
\[
X_{\un{n}}(\La^*) : =  \ol{\spn} \{ T_{\umu} a \mid a \in A, \umu \in \La^*, |\umu| = \un{n}\}
\]
defines a product system $X(\La^*)$ with the unitary equivalences given by
\[
X_{\un{n}}(\La^*) \otimes_A X_{\un{m}}(\La^*) \cong \ol{[X_{\un{n}}(\La^*) \cdot X_{\un{m}}(\La^*)]} .
\]
Moreover the equation $V_{\umu} = t(T_{\umu})$ defines a bijection between the Nica-covariant representations $(\pi,t)$ of $X(\La^*)$ and pairs $(\pi,V)$ such that:
\begin{enumerate}
\item $V \colon \La^* \to \B(H)$ is a Nica-covariant representation for $\La^*$, in the sense that
\[
V_{\umu} V_{\umu}^* V_{\unu} V_{\unu}^*
=
\begin{cases}
V_{\umu \vee \unu} V_{\umu \vee \unu}^* & \text{ if } \umu \vee \unu \in \La^*, \\
0 & \text{ otherwise},
\end{cases}
\]
\item $\pi(a) V_{\de_i(k)} = V_{\de_i(k)} \pi \al_{\de_i(k)}(a)$ for all $(i, k) \in [N] \times [d]$ and $a \in A$,
\item $V_{\de_i(k)}^* V_{\de_i(l)} = \de_{k, l} \pi(T_{\de_i(k)}^* T_{\de_i(k)})$ for all $i \in [N]$ and $k, l \in [d]$.
\end{enumerate}
\end{proposition}

Henceforth we will write $\N\T(\La^*)$ and $\N \O(\La^*)$ for the Nica-Toeplitz-Pimsner and the Cuntz-Nica-Pimsner algebras of $X(\La^*)$, respectively.
It is clear that $(\id, T)$ is a Nica-covariant representation for $X(\La^*)$ such that $\ca(\id, T) = \ca(T)$.
However this particular representation may not be faithful when induced on $\N\T(\La^*)$. 

\subsection{The $\I_F$-ideals of $X(\La^*)$}

Now that we have constructed the algebras associated to $\La^*$, we wish to understand the ideals of relations. 
Note first that the left action on $X(\La^*)$ is by compacts. Indeed a straightforward computation yields
\[
\phi_{\Bi}(a) = \sum_{k \in [d]} \theta^{X_{\Bi}}_{\xi_k, \eta_k} \qfor \xi_k = T_{\de_i(k)} \al_{\de_i(k)}(a) \text{ and } \eta_k = T_{\de_i(k)}.
\]
Hence, it is easy to verify the following proposition for the ideals of solutions.

\begin{proposition}
Let $\La^*$ be an FL in $\fdn$ and let $\mt \neq F \subseteq [N]$.
Then we have
\[
\bigcap_{i\in F}\ker \phi_{\Bi} = \bigcap \{ \ker \al_{\de_i(k)} \mid i \in F, k \in [d]\}
\qand
\I_F = \bigcap\{ \al_{\umu}^{-1}(\J_F) \mid \supp \umu \perp F\}.
\]
\end{proposition}

Let $(\pi,t)$ be a Nica-covariant representation of $X(\La^*)$ associated to a Nica-covariant representation $(\pi,V)$ of $\La^*$. 
Due to the product system construction and equation \eqref{eq:p'i}, we get that
\[
p_{\Bi} = \sum_{k \in [d]} V_{\de_i(k)} V_{\de_i(k)}^* \in \pi(A)'.
\]
To make a distinction we write $P_{\Bi}$ and $P_{\un{n}}$ for the projections induced by $(\id, T)$.
Consequently
\[
I - P_{\Bi} 
=
I - \sum_{k \in [d]} T_{\de_{i}(k)} T_{\de_i(k)}^*
\]
is the projection on $\ol{\spn} \{e_{\umu} \mid \umu \in \La^*, \mu_i = \mt\}$.
Recall our notation in Subsection \ref{subsec:cnp-ideal}, so that $Q_F = \prod_{i \in F} (I - P_{\Bi})$.
It follows that $Q_F \in A' \cap \ca(T)$ and that $Q_{[N]}$ is the projection $Q_{\umt}$ on $e_{\umt}$.
As the product system is built in a way that reflects the algebraic relations in $\ca(T)$ there is a connection between $Q_F$ and $\bigcap_{i \in F} \ker \phi_{\Bi}$.

\begin{proposition}\label{P:I_F Q_F}
Let $\La^*$ be an FL in $\fdn$ and let $F \subseteq [N]$.
Then
\[
\bigcap_{i\in F}\ker \phi_{\Bi} = \{a \in A \mid a = a Q_F\}
\qand
\{a \in A \mid a Q_F = 0\} \subseteq \I_F.
\]
Hence, $Q_F \in A$ if and only if $\I_F = A (I - Q_F)$.
\end{proposition}

\begin{proof}
Let $a \in \bigcap_{i \in F} \ker \phi_{\Bi}$ and $i \in F$.
Then $a T_{\de_i(k)} = T_{\de_i(k)} \al_{\de_i(k)}(a) = 0$ for every $k \in [d]$, and therefore
\[
a (I - P_{\Bi}) = a (I - \sum_{k \in [d]} T_{\de_{i}(k)} T_{\de_i(k)}^*) = a.
\]
As this holds for every $i \in F$ we get that $a Q_F = a$.
Conversely if $a Q_F = a$ then 
\[
\al_{\de_i(k)}(a) = T_{\de_i(k)}^* a T_{\de_i(k)} = T_{\de_i(k)}^* a Q_F T_{\de_i(k)} = T_{\de_i(k)}^* a Q_F (I - P_i) T_{\de_i(k)} = 0
\]
for all $(i,k) \in F \times [d]$, which shows that $a \in \bigcap_{i \in F} \ker \phi_{\Bi}$.
For the second inclusion let $a \in A$ such that $a Q_F = 0$ and let $b \in \bigcap_{i \in F} \ker \phi_{\Bi}$.
Then $b = bQ_F$ and so
\[
ab = ab Q_F = a Q_F b = 0,
\]
which gives that $a \in \J_F$.
Moreover, if $\umu$ has support perpendicular to $F$ then it follows that $Q_F T_{\umu} = T_{\umu} Q_F$.
Thus, if $aQ_F = 0$, then
\[
\al_{\umu}(a) Q_F = T_{\umu}^* a Q_F T_{\umu} = 0,
\]
and as above we get $\al_{\umu}(a) \in \J_F$ when $\supp \umu \perp F$.
Therefore we have that $a \in \I_F$.

For the last item, if $Q_F \in A$ then $\bigcap_{i \in F} \ker\phi_{\Bi} = A \cdot Q_F$.
From this it follows that $\J_F = A (I - Q_F)$ which is $\al_{\umu}$-invariant when $\umu \perp F$.
Hence $\I_F = A (I - Q_F)$.
The converse follows since $A$ is unital.
\end{proof}

When $\La^*$ is rank one then there is a dichotomy on $\ker\phi$: either $\ker\phi = \bC Q_{\mt}$ or $\ker\phi = (0)$ \cite{KS15}.
In fact it is shown in \cite{KS15} that $Q_{\mt} \in A$ if and only if for every $k \in [d]$ there is a $\mu_k$ such that $\mu_k k \notin \La^*$. 
However this dichotomy no longer holds in the higher rank context. 
There may be elements in $\bigcap_{i \in F} \ker \phi_{\Bi}$ that are not scalar multiples of the products of all
\[
T_{\umu(i, k)}^* T_{\umu(i, k)} \text{ such that if } (i, k) \in F \times [d] \text{ then }  \umu(i, k) \in \La^* \text{ while } \umu(i, k) \ast \de_{i}(k) \notin \La^*,
\]
even when such words exist.
Moreover in Example \ref{E:dich fail} we will see that it is possible to have $\bigcap_{i \in F} \ker \phi_{\Bi} \neq (0)$ for $F \neq [N]$ but $\bigcap_{i \in [N]} \ker \phi_{\Bi} = (0)$. On the positive side we have the following proposition. It shows that a dichotomy holds for $F = [N]$.

\begin{proposition}\label{P:IF forb}
Let $\La^*$ be an FL in $\fdn$ and fix $\mt \neq F \subseteq [N]$.
Then the following are equivalent:
\begin{enumerate}
\item There exists an $a \in \bigcap_{i \in F} \ker \phi_{\Bi}$ such that $a Q_{F^c} \neq 0$.
\item There exists an $a \in A$ such that $a Q_{F^c} = Q_{\umt}$.
\item For every $\de_i(k) \in \La^*$ with $(i,k) \in F \times [d]$ there exists $\umu \in \La^*$ such that $\umu \ast \de_i(k) \notin \La^*$.
\end{enumerate}
If $F = [N]$ then the above are also equivalent to:
\begin{enumerate}
\item[\textup{(iv)} ] $\bigcap_{i \in [N]} \ker \phi_{\Bi} = \bC Q_{\umt}$.
\item[\textup{(v)} ] $\bigcap_{i \in [N]} \ker \phi_{\Bi} \neq (0)$.
\item[\textup{(vi)} ] $Q_{\umt} \in A$.
\end{enumerate}
\end{proposition}

\begin{proof}

[(i) $\Rightarrow$ (ii)]: 
Without loss of generality we may assume that there exists an $n \in \bZ_+$ such that $a \in \bigcap_{i \in F} \ker \phi_{\Bi} \cap A_{n \cdot \un{1}}$.
Write $\{Q_m\}_m$ for the minimal projections of $A_{n \cdot \un{1}}$ given by products of $T_{\umu}^* T_{\umu}$ and $I- T_{\unu}^* T_{\unu}$ for $|\umu|, |\unu| \leq n \cdot \un{1}$.
We reserve $Q_0$ for the product of all the $T_{\umu}^* T_{\umu}$ with $|\umu| \leq n \cdot \un{1}$.
Thus we can write $a= \sum_m \la_m Q_m$.
Since $Q_{\umt} \leq T_{\un{\nu}}^* T_{\un{\nu}}$ for all $\un{\nu} \in \La^*$ then $Q_{\umt} Q_m = 0$ for all $m \neq 0$, and $Q_{\un{\mt}} Q_0 = Q_{\un{\mt}}$.
Proposition \ref{P:I_F Q_F} gives that $a = a Q_F$ and so 
\begin{align*}
0 \neq a Q_{F^c} 
& = 
a Q_F Q_{F^c} = a Q_{\un{\mt}} 
 =
\sum_m \la_m Q_m Q_{\un{\mt}}
= \la_0 Q_{\un{\mt}}.
\end{align*}
Hence $\la_0 \neq 0$ and $Q_{\umt} = (\la_0^{-1} a ) Q_{F^c}$.

\noindent
[(ii) $\Rightarrow$ (iii)]:
Suppose there exists an $a \in A$ with $Q_{\un{\mt}} = a Q_{F^c}$.
Then $a Q_{\un{\mt}} = a Q_{F^c} Q_{\un{\mt}} = Q_{\un{\mt}}$.
Let $a_\ka \in A_{n \cdot \un{1}}$ converge to $a$, and so $\lim_\ka a_\ka Q_{\un{\mt}} = a Q_{\un{\mt}} = Q_{\un{\mt}}$.
By using a decomposition of each $a_\ka$ with respect to the minimal projections in $A_{n \cdot \un{1}}$, we find that the $0$-coefficients $\la_\ka Q_0$ tend in norm to $Q_{\un{\mt}}$. Consequently we have that
\[
1 = \| Q_{\un{\mt}} e_{\un{\mt}} \| = \lim_\ka \| \la_\ka Q_0 e_{\un{\mt}} \| = \lim_\ka |\la_\ka| .
\]
If there were a $\de_i(k) \in \La^*$ such that $\un{\mu} \ast \de_i(k) \in \La^*$ for all $\umu \in \La^*$, then $Q_0 e_{\de_i(k)} = e_{\de_i(k)}$.
Hence we would derive the contradiction
\[
0 = \| Q_{\un{\mt}} e_{\de_i(k)} \| = \lim_\ka \| \la_\ka Q_0 e_{\de_i(k)} \| = \lim_\ka |\la_\ka|.
\]

\noindent
[(iii) $\Rightarrow$ (i)]:
For $(i,k) \in F \times [d]$, let $\umu(i,k) \in \La^*$ such that $\umu(i,k) \ast \de_i(k) \notin \La^*$.
Then the element
\begin{equation}\label{eq:con a}
a = \prod_{(i,k) \in F \times [d]} T_{\umu(i,k)}^* T_{\umu(i,k)}
\end{equation}
satisfies item (i).
Indeed notice that $a Q_{F^c} = Q_{\umt}$, since $a e_{\delta_i(k)} = 0$ for all $(i,k) \in F \times [d]$. 

\noindent
[(iii) $\Leftrightarrow$ (iv) $\Leftrightarrow$ (v) $\Leftrightarrow$ (vi)]:
These equivalences follow by applying the same arguments with the identity $I$ in the place of $Q_{F^c}$ to show that $Q_{\un{\emptyset}} \in A$, so that Proposition \ref{P:I_F Q_F} finishes the proof.
\end{proof}

\begin{example}\label{E:dich fail}
Let $\La^* \subseteq \bF_+^3 \times \bF_+^3$ on three symbols $\{0,1,2\}$ defined by the forbidden words
\[
\F = \{(1,1), (2,2), (0,2), (\mt, 02) \}.
\]
We claim that the element
\[
a := T_{(\mt,1)}^* T_{(\mt,1)} \cdot T_{(\mt, 2)}^* T_{(\mt, 2)} \cdot (I - T_{(\mt, 0)}^* T_{(\mt, 0)})
\]
is a non-trivial element in $\ker \phi_{\bo{1}}$. 
First notice that
\[
T_{(\mt,1)}^* T_{(\mt,1)} e_{(\mt,2)} = T_{(\mt,2)}^* T_{(\mt,2)} e_{(\mt,2)} = (I - T_{(\mt, 0)}^* T_{(\mt, 0)}) e_{(\mt,2)} = e_{(\mt,2)}
\]
and thus $a \neq 0$.
Since $(1,1) \notin \La^*$ we have that
\[
T_{(\mt,1)}^* T_{(\mt,1)} T_{(1, \mt)} e_{(\mu,\nu)} = 0 \foral (\mu,\nu) \in \La^*.
\]
Therefore $T_{(\mt,1)}^* T_{(\mt,1)}$ is in $\ker \al_{\de_1{(1)}}$.
Similarly we obtain that $T_{(\mt,2)}^* T_{(\mt,2)}$ is in $\ker \al_{\de_1{(2)}}$.
By construction of the forbidden set we also have that $(0\mu,\nu) \in \La^*$ if and only if $(0\mu,0\nu) \in \La^*$.
Indeed if $(0\mu, \nu) \in \La^*$ then $2 \notin \nu$.
If $(0\mu,0\nu) \notin \La^*$ then there should be a forbidden word $(0\mu',0\nu')$ in $(0\mu, 0\nu)$; this can only happen if $2 \in \nu'$ giving the contradiction $2 \in \nu$.
Conversely if $(0\mu,0\nu) \in \La^*$ then its subword $(0\mu,\nu)$ is also in $\La^*$.
Therefore we have that
\begin{align*}
T_{(\mt, 0)}^* T_{(\mt, 0)} T_{(0,\mt)} e_{(\mu, \nu)}
& =
\begin{cases}
e_{(0\mu,\nu)} & \text{ if } (0\mu,0\nu) \in \La^* \text{ and } (0\mu,\nu) \in \La^*, \\
0 & \text{otherwise},
\end{cases}
\\
& =
\begin{cases}
e_{(0\mu,\nu)} & \text{ if } (0\mu,\nu) \in \La^*, \\
0 & \text{otherwise},
\end{cases}
\\
& =
T_{(0,\mt)} e_{(\mu,\nu)}
\end{align*}
which shows that $I - T_{(\mt, 0)}^* T_{(\mt, 0)} \in \ker \al_{\de_1{(0)}}$. Thus, we have that $a\in \cap_{j=0}^2 \ker \al_{\de_1(j)} = \ker \phi_{\bo{1}}$.
Note here that $a Q_{\{2\}} = 0$. Indeed, since
\[
a Q_{\{2\}} = a (I-P_{\bo{2}}) = a(I - \sum_{j=0}^2T_{\de_2(j)}T_{\de_2(j)}^*)
\]
we see that $Q_{\{2\}}e_{(\mu,\nu)} \neq 0$ forces $\nu = \emptyset$. Then $I - T^*_{(\emptyset, 0)}T_{(\emptyset,0)} e_{(\mu, \emptyset)} = 0$ for every $\mu \in \bF_+^3$, so that $aQ_{\{2\}} = 0$ (compare this with item (i) of Proposition \ref{P:IF forb}). 

At the same time we have $\I_{[2]} = A$ since $\ker \phi_{\bo{1}} \cap \ker \phi_{\bo{2}} = (0)$. Indeed, by Proposition \ref{P:IF forb} it will suffice to show there's no word $(\mu,\nu) \in \La^*$ such that $(\mu,\nu) \ast (\mt, 0) \notin \La^*$. However, it is clear from the definition of forbidden words that $(\mu, \nu0) \notin \La^*$ if and only if $(\mu, \nu) \notin \La^*$, and we are done.
\end{example}

\subsection{Consequences of the GIUT}

In the rank one case, the Cuntz-Pimsner algebra $\O(\La^*)$ is only one of the extremes of $\ca(T)$ and $\ca(T)/\K$.
However this is no longer true in more directions.
In fact we have to consider a larger ideal than the compacts.
To this end let $\fI_Q$ be the ideal in $\ca(T)$ generated by the $Q_F$, i.e.
\[
\fI_Q := \sca{Q_F \mid \mt \neq F \subseteq [N]}.
\]
Proposition \ref{P:K ideal} implies that $\fI_Q$ contains $\K(\id, T)$, as well as the compacts $\K = \sca{Q_{\umt}}$.
For the latter notice that  $T_{\unu} Q_{\umt} T_{\umu}^*$ is the rank one operator $e_{\umu} \mapsto e_{\unu}$.
Thus there is a canonical $*$-epimorphism $\N\O(\La^*) \to \ca(T)/ \fI_Q$.
On the other hand Corollary \ref{C:co-un} asserts that $\ca(T)$ lies above $\N\O(\La^*)$ (as trivially $A \hookrightarrow \ca(T)$).
Consequently we have the following canonical $*$-epimorphisms
\[
\N\T(\La^*) \to \ca(T) \to \N\O(\La^*) \to \ca(T)/ \fI_Q.
\]
First we ask when is it possible for $\N\O(\La^*)$ to coincide with $\ca(T)$ or with $\ca(T)/ \fI_Q$, and secondly if $\N\O(\La^*)$ can take values ``in-between''.
We answer these questions in the following three applications of the GIUT.

\begin{corollary}\label{C:ext 1}
Let $\La^*$ be an FL in $\fdn$.
Then the following are equivalent:
\begin{enumerate}
\item the canonical $*$-epimorphism $\N\O(\La^*) \to \ca(T)/ \fI_Q$ is injective;
\item $\I_F = A$ for all $\mt \neq F \subseteq [N]$.
\end{enumerate}
\end{corollary}

\begin{proof}
Fix $q \colon \ca(T) \to \ca(T)/ \fI_Q$.
If item (i) holds then $q|_A$ is injective as it factors through $A \hookrightarrow \N\O(\La^*)$.
Moreover $q(a) q(Q_F) = 0$ for all $a \in A$ since $q(Q_F) = 0$. 
It then follows from Proposition \ref{P:in IF} that $A \subseteq \I_F$.
Conversely if item (ii) holds then $\fI_Q = \K(\id, T)$ so that if $a \in \K(\id,T) \cap A$, we would have that $a=0$ by Corollary \ref{C:co-un}. Thus, $A$ embeds in $\ca(T)/ \fI_Q$, and the GIUT yields $\N\O(\La^*) \simeq \ca(T)/ \fI_Q$.
\end{proof}

\begin{corollary}\label{C:ext 2}
Let $\La^*$ be an FL in $\fdn$.
Then the following are equivalent:
\begin{enumerate}
\item the canonical $*$-epimorphism $\ca(T) \to \N\O(\La^*)$ is injective;
\item $\I_{F} = \{a \in A \mid a Q_F = 0\}$ for all $\mt \neq F \subseteq [N]$.
\end{enumerate}
\end{corollary}

\begin{proof}
Suppose that item (i) holds and fix $F \subseteq [N]$. 
Then $(\id, T)$ is a CNP-re\-pre\-sentation and so $a Q_F = 0$ for all $a \in \I_F$.
Proposition \ref{P:I_F Q_F} implies that $\I_F = \{a \in A \mid a Q_F = 0\}$.
Conversely, if item (ii) holds then $(\id, T)$ defines a CNP-representation for $\La^*$.
Since $A \subseteq \ca(T)$, then the GIUT implies that $\ca(T) \simeq \N\O(\La^*)$.
\end{proof}

Next we determine the form of $\N\O(\La^*)$ when $\La^*$ is a product of rank one languages, as the tensor product of Cuntz-Pimsner algebras.
Recall that the Pimsner algebra of a rank one language is nuclear \cite{Kat04, KS15}, and so all tensor products can be taken to be minimal.
For an FL $\La^*$ in $\fdn$, denote by $\La^*_i$ its projection at the $i$-th coordinate, so that $\La^* \subseteq \La^*_1 \times \cdots \times \La^*_N$.
To make a distinction suppose that $\ca(\La^*_i) = \ca(t_{i,k} \mid k \in [d])$ and define
\[
\wh{t}_{i,k} = I \otimes \cdots \otimes I \otimes \underbrace{t_{i,k}}_{\textup{$i$-th position}} \otimes I \otimes \cdots \otimes I
\in
\ca(\La^*_1) \otimes \cdots \otimes \ca(\La^*_N).
\]
Then we may define the correspondence
\begin{equation*}
\N\O(\La^*) \ni V_{\de_i(k)} \mapsto \wh{t}_{i,k} \in \ca(t_{1,k} \mid k \in [d]) \otimes \cdots \otimes \ca(t_{N,k} \mid k \in [d]).
\end{equation*}
Moreover there is a canonical $*$-epimorphism
\[
\Phi \colon  \ca(t_{1,k} \mid k \in [d]) \otimes \cdots \otimes \ca(t_{N,k} \mid k \in [d]) \to \O(\La^*_1) \otimes \cdots \otimes \O(\La^*_N)
\]
implementing the correspondence
\begin{equation}\label{eq:map}
\N\O(\La^*) \ni V_{\de_i(k)} \mapsto \Phi(\wh{t}_{i,k}) \in \O(\La^*_1) \otimes \cdots \otimes \O(\La^*_N).
\end{equation}

\begin{corollary}\label{C:inbetween}
Let $\La^*$ be an FL in $\fdn$ and let $\La^*_1, \dots, \La^*_N$ be its projections.
Then the bijection in equation (\ref{eq:map}) induces a $*$-isomorphism between $\N\O(\La^*)$ and $\O(\La^*_1) \otimes \cdots \otimes \O(\La^*_N)$ if and only if $\La^* = \La^*_1 \times \cdots \times \La^*_N$.
\end{corollary}

\begin{proof}
First suppose that $\La^* = \La^*_1 \times \cdots \times \La^*_N$ so that $\ca(T) = \ca(t_1) \otimes \cdots \otimes \ca(t_N)$.
To fix notation write $\{\al_{i, k}\}_{k \in [d]}$ for the quantized dynamics on $A_i$ and note that $A = A_1 \otimes \cdots \otimes A_N$.
For $i \in [N]$ we then have
\[
\al_{\de_i(k)} =  \id \otimes \cdots \otimes \id \otimes \underbrace{\al_{i, k}}_{\textup{$i$-th position}} \otimes \id \otimes \cdots \otimes \id.
\]
Since the $A_i$ are commutative we get
\[
\ker \al_{\de_i(k)} = A_1 \otimes \cdots \otimes A_{i -1} \otimes \ker\al_{i,k} \otimes A_{i+1} \otimes \cdots \otimes A_N.
\]
Without loss of generality, up to rearranging terms, there's an $\ell \in [N]$ such that
\[
\ker \phi_{\bo{i}} = \bigcap_{k \in [d]} \ker \al_{i, k} \neq (0) \text{ for } i \leq \ell
\qand
\ker \phi_{\bo{i}} = \bigcap_{k \in [d]} \ker \al_{i, k} = (0) \text{ for } i \geq \ell + 1.
\]
Then by \cite[Theorem 6.1]{KS15} we derive that
\[
\O(\La^*_i) 
= 
\begin{cases}
\ca(t_{i,k} \mid k \in [d]) & \textup{if } i \leq \ell,\\
\ca(t_{i,k} \mid k \in [d]) / \K & \textup{if } i \geq \ell+1.
\end{cases}
\]
We also fix the notation
\[
q_{\mt, i} = I - \sum_{k \in [d]} t_{i,k} t_{i,k}^* \in A_i \qfor i \leq \ell.
\]
Then we have to show that
\[
\N\O(\La^*) \simeq \ca(t_1) \otimes \ca(t_\ell) \otimes \ca(t_{\ell+1})/\K \otimes \cdots \otimes \ca(t_{N})/\K.
\] 
Take the quotient map
\[
\Phi \colon \ca(T) \to \ca(t_1) \otimes \ca(t_\ell) \otimes \ca(t_{\ell+1})/\K \otimes \cdots \otimes \ca(t_{N})/\K.
\]
As $(\id, T)$ forms a Nica-Pimsner representation that admits a gauge action then so does $(\Phi|_A, \Phi T)$.
In particular $\Phi|_A$ is faithful as the tensor product of injective representations.
It suffices to show that $\Phi(a) \Phi(Q_F) = 0$ for all $a \in \I_F$, and then the GIUT finishes the proof.
First suppose that $F \subseteq \{1, \dots, \ell\}$ and without loss of generality assume that $F = \{1, \dots, m\}$ for $m \leq \ell$.
Then by definition we have
\[
Q_F = q_{\mt, 1} \otimes \cdots \otimes q_{\mt, m} \otimes I \otimes \cdots \otimes I
\]
and thus $Q_F \in A$. 
Hence Proposition \ref{P:I_F Q_F} gives that $\I_F = A (I - Q_F)$.
Thus if $a \in\ \I_F$ then
\[
\Phi(a) \Phi(Q_F) = \Phi(a (I - Q_F)) \Phi(Q_F) = 0.
\]
On the other hand suppose that $F$ contains an element from $\{\ell+1, \dots, N\}$ and without loss of generality suppose that $N \in F$.
It then follows that $\bigcap_{k \in [d]} \ker \al_{\de_{N}(k)} = (0)$ and so $\bigcap_{i \in F} \ker \phi_{\Bi} = (0)$.
Hence $\I_F = A$.
At the same time we have that $q_{\mt, N} \in \K$ and thus $\Phi(I - P_{N}) = 0$.
Therefore $\Phi(Q_F) = \Phi(Q_F) \Phi_F(I - P_{N}) = 0$, and trivially $\Phi(a) \Phi(Q_F) = 0$ for all $a \in A = \I_F$.

Conversely, suppose that the correspondence of equation (\ref{eq:map}) extends to a $*$-isomor\-phism.
Then for $\umu = (\mu_1, \dots, \mu_N) \in \fdn$ we have that:
\[
\umu \in \La^* 
\iff V_{\umu}^* V_{\umu} \neq 0 
\iff \Phi(t_{\mu_1}^* t_{\mu_1} \otimes \cdots \otimes t_{\mu_N}^* t_{\mu_N}) \neq 0.
\]
However $\Phi$ is injective on $A_1 \otimes \cdots \otimes A_N$ and thus the above is equivalent to $t_{\mu_i}^* t_{\mu_i} \neq 0$ for all $i \in [N]$, which holds if and only if $\mu_i \in \La^*_i$ for all $i \in [N]$.
Thus $\La^*$ coincides with $\La^*_1 \times \cdots \times \La^*_N$.
\end{proof}

Now let us examine the sofic case further.
Henceforth, let $\La^*$ be a sofic FL and denote by $(\La, \fL)$ its higher rank follower set graph as in Definition \ref{D:fsg}.
That is, $\La$ is the ambient higher rank graph and $\fL$ is the labeling.
By Remark \ref{R:fsg} there is a one-to-one correspondence $v \mapsto p_v$ between the vertices of $\La$ and the minimal projections in the algebra $A$ of the checkers.
We wish to connect $\N\O(\La^*)$ with $\ca(\La)$.
To allow comparisons we will denote by $\I_F(\La^*)$ the ideals of solutions for $X(\La^*)$.

\begin{proposition}\label{P:cnp fsg}
Let $\La^*$ be a sofic FL and let $(\La, \fL)$ be its higher rank follower set graph.
Let $\mt \neq F \subseteq [N]$.
If a vertex $v$ of $\La$ is $F$-tracing then its corresponding projection $p_v$ is in $\I_F(\La^*)$.
\end{proposition}

\begin{proof}
Since $v$ is $F$-tracing then there is an $i \in F$  and an edge $(e,i,k)$ with $r(e) = v$.
By construction $\al_{\de_i(k)}(p_v) \geq p_{s(e)} \neq 0$ showing that $p_v \notin \bigcap_{i \in F} \ker\phi_{\Bi}$.
Since $A$ is discrete we get that $p_v \in \J_F(\La^*)$.

Now take $\lambda \in \La$ such that $d(\la) \perp F$ and apply $\al_{\fL(\la)}$ on $p_v$.
If $c = \al_{\fL(\la)}(p_v) \neq 0$ then it corresponds to the sum of $p_{s(\nu)}$ for the paths $\nu$ ending at $v$ with $\fL(\nu) = \fL(\la)$.
For such a $\nu$ we get that $d(\nu) = d(\la)$ and thus $d(\nu) \perp F$ as well.
As $v$ is $F$-tracing, the vertex $s(\nu)$ must receive an edge $(f,i,k)$ for some $i \in F$.
The comments above show that $p_{s(\nu)} \in \J_F(\La^*)$.
As this holds for all $\nu$ we get that $c \in \J_F(\La^*)$ and thus $p_v \in \I_F(\La^*)$.
\end{proof}

We now arrive at the conclusion of this section which connects C*-algebras of higher rank graphs with those of FL's through the follower set graph construction. 
This result is in analogy with \cite{Car03} and \cite{KS15} for the single variable case.

\begin{theorem}\label{T: fsg}
Let $\La^*$ be a sofic FL and let $(\La, \fL)$ be its follower set graph.
Then there is a canonical $*$-isomorphism between $\N\O(\La^*)$ and $\ca(\La)$.
\end{theorem}

\begin{proof}
Since the follower set graph is row-finite its Cuntz-Krieger $\La$-families coincide with the CNP-representations.
Let $(\pi,V)$ correspond to a CNP-representation of $X(\La^*)$.
We define the family of operators $\{t_\la \mid \la \in \La\}$ by
\[
t_{\la}
:=
\begin{cases}
\pi(p_{v}) & \textup{ if $\la = v$ is a vertex}, \\
V_{\fL(\la)} \pi(p_{s(\la)}) & \textup{ otherwise}.
\end{cases}
\]
Since $A$ is generated by the minimal projections $\pi(p_v)$ (and they add up to the identity) we get that
\[
\ca(\pi, V)
=
\ca(V_{\umu}, \pi(p_v) \mid \umu \in \La^*, v \in \La^{\un{0}})
=
\ca(t_\la \mid \la \in \La).
\]
It suffices to show that $\{t_\la \mid \la \in \La\}$ is a Cuntz-Krieger $\La$-family.
Then Theorem \ref{T:cnp is ck} applies to give the $*$-isomorphism between $\N \O (\La^*)$ and $\ca(\La)$.

It is clear that $\{t_v \mid v \in \La^{\un{0}}\}$ is a set of pairwise disjoint projections.
For (HR), let $\la, \mu \in \La$ such that $r(\mu) = s(\la)$ and compute
\begin{align*}
t_\la t_\mu
& =
V_{\fL(\la)} \pi(p_{s(\la)}) V_{\fL(\mu)} \pi(p_{s(\mu)}) 
 =
V_{\fL(\la) \ast \fL(\mu)} \pi( \al_{\fL(\mu)}(p_{s(\la)}) p_{s(\mu)}) 
 =
V_{\fL(\la\mu)} \pi(p_{s(\la\mu)})
=
t_{\la \mu},
\end{align*}
where we used that $\alpha_{\fL(\mu)}(p_{s(\la)}) \geq p_{s(\mu)}$ and $s(\la\mu) = s(\mu)$.
For (NC), let $\la, \mu \in \La$ and consider two cases.
If $\fL(\la) \vee \fL(\mu) = \infty$ then $\La^{\min}(\la,\mu) = \emptyset$.
Therefore in this case we get
\[
t_\la^* t_\mu = p_{s(\la)} V_{\fL(\la)}^* V_{\fL(\mu)} p_{s(\mu)} = 0.
\]
Now suppose that $\fL(\la) \vee \fL(\mu) < \infty$ and let paths $x, y$ in $\La$ such that
\[
\fL(\la) \ast \fL(x) = \fL(\la) \vee \fL(\mu) = \fL(\mu) \ast \fL(y).
\]
At this point we do not claim that $\la x$ or $\mu y$ exist.
We can use Nica-covariance of $V$ to deduce that
\begin{align*}
V_{\fL(\la)}^* V_{\fL(\mu)}
& =
V_{\fL(\la)}^* V_{\fL(\la)} V_{\fL(\la)}^* V_{\fL(\mu)} V_{\fL(\mu)}^* V_{\fL(\mu)} 
 =
V_{\fL(\la)}^* V_{\fL(\la) \vee \fL(\mu)} V_{\fL(\la) \vee \fL(\mu)}^* V_{\fL(\mu)} \\
& =
\pi( T_{\fL(\la)}^* T_{\fL(\la)}) V_{\fL(x)} V_{\fL(y)}^* \pi(T_{\fL(\mu)}^* T_{\fL(\mu)}) 
=
V_{\fL(x)} \pi(a) V_{\fL(y)}^*
\end{align*}
where the last equality uses commutation up to applying $\alpha$, for
\[
a := \al_{\fL(x)}(T_{\fL(\la)}^* T_{\fL(\la)}) = T_{\fL(\la) \vee \fL(\mu)}^* T_{\fL(\la) \vee \fL(\mu)} = \al_{\fL(y)}(T_{\fL(\mu)}^* T_{\fL(\mu)}).
\]
We thus get that
\begin{align*}
t_{\la}^* t_\mu
& =
\pi(p_{s(\la)}) V_{\fL(x)} \pi(a) V_{\fL(y)}^* \pi(p_{s(\mu)}) 
 =
V_{\fL(x)} \pi(\al_{\fL(x)}(p_{s(\la)}) \cdot  a \cdot \al_{\fL(y)}(p_{s(\mu)})) V_{\fL(y)}^*.
\end{align*}
By construction we have
\begin{align*}
& \al_{\fL(x)}(p_{s(\la)}) \cdot  a \cdot \al_{\fL(y)}(p_{s(\mu)})
=\\
& \quad =
\sumoplus \{ p_{s(\nu_1)} p_{s(\nu)} p_{s(\nu_2)} \mid \nu_1, \nu_2, \nu \in \La \textup{ such that } \\
& \hspace{5cm} r(\nu_1) = s(\la), r(\nu_2) = s(\mu), \AND \\
& \hspace{5cm} \fL(\nu_1) = \fL(x), \fL(\nu_2) = \fL(y), \fL(\nu) = \fL(\la) \vee \fL(\mu) \} 
\end{align*}
Clearly $p_{s(\nu_1)} p_{s(\nu)} p_{s(\nu_2)}$ is non-zero exactly when $s(\nu_1) = s(\nu) = s(\nu_2)$.

Now there are two cases. 
First suppose that $\La^{\min}(\la,\mu) = \emptyset$.
As $(\La,\fL)$ is source-resolving, if $s(\nu_1) = s(\nu) = s(\nu_2)$ then we must have that $\nu = \la \nu_1 = \mu \nu_2$ and hence $(\nu_1,\nu_2) \in \La^{\min}(\la,\mu)$ which is a contradiction. 
Thus, all the summands above are zero giving that $t_\la^* t_\mu = 0$ as required.

Now suppose that $\La^{\min}(\la,\mu) \neq \emptyset$.
Then there are $\nu_1, \nu_2, \nu \in \La$ such that $s(\nu_1) = s(\nu) = s(\nu_2)$. Again since $(\La,\fL)$ is source-resolving, we must have that $\nu = \la \nu_1 = \mu \nu_2$ with $(\nu_1,\nu_2) \in \La^{\min}(\la,\mu)$. Indeed, we have that
\[
\fL(\la) \ast \fL(\nu_1) = \fL(\la) \ast \fL(x) = \fL(\la) \vee \fL(\mu) = \fL(\nu)
\]
and likewise for $\fL(\mu) \ast \fL(\nu_2)$.
As the degree of a path agrees with the multilength of its labeling, we have that $(\nu_1, \nu_2) \in \La^{\min}(\la,\mu)$.
Therefore we obtain that
\begin{align*}
t_{\la}^* t_\mu
& =
\sum_{\nu_1, \nu_2} \{ V_{\fL(x)} p_{s(\nu_1)} p_{s(\nu_2)} V_{\fL(y)}^* 
\mid \fL(\nu_1) = \fL(x), \ \fL(\nu_2) = \fL(y), \ (\nu_1, \nu_2) \in \La^{\min}(\la,\mu) \}
\\
& =
\sum \{t_{\nu_1} t_{\nu_2}^* \mid (\nu_1, \nu_2) \in \La^{\min}(\la, \mu)\}
\end{align*}
and thus (NC) is satisfied.
Finally we show that the family satisfies (CK') and then Theorem \ref{T:cnp is ck} will conclude that $\{V_{\fL(\la)} \pi(p_{s(\la)}) \mid \la \in \La\}$ is a Cuntz-Krieger $\La$-family.
To this end, by Proposition \ref{P:cnp fsg} it suffices to show that 
\begin{equation}\label{eq:show cnp}
\prod \{ t_v - t_{\mu} t_{\mu}^* \mid \mu \in v \La^{\Bi}\}
=
\pi(p_v) \prod_{i \in F} (I - p_{\Bi})
\end{equation}
for any $F$-tracing vertex $v$. Notice that 
\[
\prod \{ t_v - t_{\mu} t_{\mu}^* \mid \mu \in v \La^{\Bi}\}
=
t_v - \sum\{ t_{\mu} t_\mu^* \mid \mu \in v \La^\Bi \}.
\]
Since $p_v$ commutes with $p_{\Bi}$ we get
\[
\pi(p_v) \prod_{i \in F} (I - p_{\Bi})
=
\prod_{i \in F} (\pi(p_v) - \pi(p_v) p_{\Bi}).
\]
By construction we have that $p_{\Bi} = \sum_{k \in [d]} V_{\de_i(k)} V_{\de_i(k)}^*$ and so $\pi(p_v) p_{\Bi} \neq 0$ if and only if $v\La^\Bi \neq \mt$.
Thus, we need to verify that
\[
\sum\{ t_{\mu} t_\mu^* \mid \mu \in v \La^\Bi \}
=
\pi(p_v) \sum_{k \in [d]} V_{\de_i(k)} V_{\de_i(k)}^*
\]
for the $i \in F$ such that $v\La^\Bi \neq \mt$, and then equation (\ref{eq:show cnp}) will hold.
However a straightforward computation gives that
\begin{align*}
\sum\{ t_{\mu} t_\mu^* \mid \mu \in v \La^\Bi \}
& =
\sum_{k \in [d]} \sum\{ V_{\de_{i}(k)} \pi(p_{s(\mu)}) V_{\de_i(k)}^* \mid \mu \in v\La^\Bi \textup{ s.t. } \fL(\mu) = \de_i(k) \} \\
& =
\sum_{k \in [d]} V_{\de_{i}(k)} \pi(\al_{\de_i(k)}(p_v)) V_{\de_i(k)}^* 
=
\pi(p_v) \sum_{k \in [d]} V_{\de_i(k)} V_{\de_i(k)}^*,
\end{align*}
and the proof is complete.
\end{proof}

\section{Future directions}\label{S:que}

A number of results in C*-correspondences rely on GIUT and now we have simpler relations that we can use to pursue the higher rank analogues. Let us describe here some problems that can follow in subsequent works.

\begin{question}
Our setting has re-formulations in the context of other semigroups.
It is possible that an algebraic description of $\N\O(X)$ for strong compactly product systems may still be carried over by using a more general setting, such as in Sims--Yeend \cite{SY11}. 
The main difference with $\bZ_+^N$ is that for general $P$ there may be no canonical generators and thus no reduction to relations in terms of subsets of the $\un{1}$-cube.
\end{question}

\begin{question}
In \cite{Kat07}, Katsura uses his insight from \cite{Kat04} to characterize gauge invariant ideals of $\T_X$ from ideals of the coefficient algebra $A$ (see also \cite{Kak14} for gauge invariant ideals inside the ideal of Katsura's relations). 
It is plausible that such techniques can be carried over to our setting. 
An important technique to try to generalize is the tail adding technique used in \cite{KK06, DFK14}.
It was first established by Muhly-Tomforde \cite{MT04} and later extended by Kakariadis-Katsoulis \cite{KK11}. 
This technique is useful for establishing Morita equivalence with C*-algebras associated to an injective system, which then gives a correspondence between ideal lattices. 
\end{question}

\begin{question}
The \emph{tensor algebra} of Muhly-Solel \cite{MS98} is the non-involuti\-ve algebra generated by a C*-correspondence inside its Toeplitz-Pimsner algebra.
Muhly-Solel \cite{MS00} have imported the notion of strong Morita equivalence to the context of C*-correspondences.
Under injectivity of the action, they have shown that it implies strong Morita equivalence of their Pimsner and tensor algebras (in the sense of Blecher-Muhly-Paulsen \cite{BMP00}) of the tensor algebras.
The injectivity condition has been removed by Eleftherakis-Kakariadis-Katsoulis \cite{EKK16} leading to strong Morita equivalence of the tensor algebras in the stronger version of Eleftherakis \cite{Ele14}.
These proofs depend on the GIUT and may extend to the context of strong compactly aligned product systems.
\end{question}

\begin{question}
Taking motivation from Symbolic Dynamics, Muhly-Pask-Tomforde \cite{MPT08} lift the notion of strong shift equivalence to the context of regular C*-cor\-respon\-dences.
By using the GIUT they show that SSE implies strong Morita equivalence for the Cuntz-Pimsner algebras, but not for the Toeplitz-Pimsner algebras and tensor algebras.
Kakaria\-dis-Katsoulis \cite{KK12} extend this theory for shift equivalence.
It is natural to ask for the shift relations for product systems and then for the corresponding results.
\end{question}

\begin{question}
A key point in \cite{KK12} is the minimal extension of an injective C*-correspon\-dence to an essential imprimitivity bimodule.
This was established in \cite{KK11} in a similar way as to how one gets the minimal automorphic extension of an injective $*$-endomorphism.
Then it is shown in \cite{KK12} that the shift relations are stable under this extension.
The direct limit process for dynamical systems has been used effectively in \cite{DFK14} for C*-dynamical systems and expanding the theory to product systems would be desirable.
\end{question}

\begin{question}
Now that a theory of higher-rank factorial languages and their algebras has been developed, it is natural to ask for analogous results to those of Matsumoto \cite{Mat02,Mat97}, Matsumoto-Carlsen \cite{CM04}, and Kakariadis-Shalit \cite{KS15}. 
For instance, one may ask for classification up to isometric/bounded isomorphism of various nonselfadjoint operator algebras associated to higher-rank factorial languages, and their relationship to subproduct systems in the sense of Shalit-Solel \cite{SS09}.

\end{question}

\begin{acknow}
The first author is grateful to Ian Putnam for several discussions on sofic shifts and follower set graphs. The first author is also grateful for the support and hospitality of the Mathematics department of University of Victoria, for a visit during which work on this project was conducted.
The second author would like to thank the Isaac Newton Institute for Mathematical Sciences, Cambridge, for support and hospitality during the programme ``Operator algebras: subfactors and their applications'' where work on this paper was undertaken. 
\end{acknow}



\begin{thebibliography}{99}

\bibitem{AHR17} Z. Afsar, A. an Huef and I. Raeburn, 
\textit{KMS states on C*-algebras associated to a family of $*$-commuting local homeomorphisms}, 
Internat.\ J.\ Math.\ \textbf{25:8} (2014), 1450066.

\bibitem{BK17} C. Barrett and E.T.A. Kakariadis,
\textit{On the quantized dynamics of factorial languages},
Quart.\ J.\ Math.\ \textbf{69:1} (2018), 119--152.

\bibitem{BHRS02} T. Bates, J. Hong, I. Raeburn and W. Szymanski,
\textit{The ideal structure of the C*-algebras of infinite graphs}, 
Illinois J.\ Math.\ \textbf{46} (2002), 1159--1176.

\bibitem{BMP00} D.P. Blecher, P.S. Muhly  and V.I. Paulsen, 
\textit{Categories of operator modules -- Morita equivalence and projective modules}, 
Mem.\ Amer.\ Math.\ Soc.\ \textbf{143} (2000), no 681.

\bibitem{Car03} T.M. Carlsen, 
\textit{On C*-algebras associated with sofic shifts}, 
J.\ Operator Theory \textbf{49:1} (2003), 203--212.

\bibitem{CLSV11} T.M. Carlsen, N.S. Larsen, A. Sims and S.T. Vittadello,
\textit{Co-universal algebras associated to product systems, and gauge-invariant uniquenss theorems},
Proc.\ London Math.\ Soc.\ (3) \textbf{103:4} (2011), 563--600.

\bibitem{CM04} T.M. Carlsen and K. Matsumoto, 
\textit{Some remarks on the C*-algebras associated with subshifts}, 
Math.\ Scand.\ \textbf{95:1} (2004), 145--160.

\bibitem{CK80} J. Cuntz and W. Krieger,
\textit{A class of C*-algebras and topological Markov chains},
Invent.\ Math.\ \textbf{56:3} (1980), 251--268.

\bibitem{DFK14} K.R. Davidson, A.H. Fuller and E.T.A. Kakariadis,
\textit{Semicrossed products of operator algebras by semigroups},
Mem.\ Amer.\ Math.\ Soc. \textbf{247:1168} (2017).

\bibitem{DY08} K.R. Davidson and D. Yang, 
\textit{Representations of higher-rank graph algebras},
New York J.\ Math.\ {\bf 15} (2009), 169--198. 

\bibitem{DKPS10} V. Deaconu, A. Kumjian, D. Pask and A. Sims,
\textit{Graphs of C*-correspondences and Fell bundles},   
Indiana Univ.\ Math.\ J.\ \textbf{59:5} (2010), 1687--1735.

\bibitem{DPZ98} S. Doplicher, C. Pinzari and R. Zuccante, 
\textit{The C*-algebra of a Hilbert bimodule}, 
Boll.\ Unione Mat.\ Ital.\ Sez.\ B Artic.\ Ric.\ Mat.\ (8) \textbf{1} (1998), 263--281.

\bibitem{Ele14} G.K. Eleftherakis,
\textit{Stable isomorphism and strong Morita equivalence of operator algebras},
Houston Journal of Mathematics, to appear.

\bibitem{EKK16} G.K. Eleftherakis, E.T.A. Kakariadis and E.G. Katsoulis,
\textit{Morita equivalence of C*-correspondences passes to the related operator algebras.},
Israel J.\ Math.\ \textbf{222:2} (2017), 949--972.

\bibitem{Eva08} D.G. Evans, 
\textit{On the K-theory of higher-rank graph C*-algebras}, 
New York J.\ Math.\ {\bf 14} (2008), 1--31.

\bibitem{Exe97} R. Exel, 
\textit{Amenability for Fell bundles}, 
J.\ Reine Angew.\ Math.\ \textbf{492} (1997) 41--73.

\bibitem{FMY05} C. Farthing, P.S. Muhly and T. Yeend, 
\textit{Higher-rank graph C*-algebras: an inverse semigroup and groupoid approach}, 
Semigroup Forum {\bf 71} (2005), 159--187.

\bibitem{Fow02} N.J. Fowler, 
\textit{Discrete product systems of Hilbert bimodules}, 
Pacific J.\ Math.\ \textbf{204:2} (2002), 335--375.

\bibitem{FR99} N.J. Fowler and I. Raeburn, 
\textit{The Toeplitz algebra of a Hilbert bimodule}, 
Indiana Univ.\ Math.\ J. \textbf{48} (1999), 155--181.

\bibitem{FMR03} N.J. Fowler, P.S. Muhly and I. Raeburn,
\textit{Representations of Cuntz-Pimsner algebras},
Indiana Univ.\ Math.\ J. \textbf{52:3} (2003), 569--605. 

\bibitem{Hop05} A. Hopenwasser, 
\textit{The spectral theorem for bimodules in higher rank graph C*-algebras}, 
Illinois J.\ Math.\ {\bf 49:3} (2005), 993--1000.

\bibitem{HR97} A. an Huef and I. Raeburn, 
\textit{The ideal structure of Cuntz-Krieger algebras}, 
Ergodic Theory Dynam.\ Systems \textbf{17} (1997), 611--624.

\bibitem{HR19} A. an Huef and I. Raeburn, 
\textit{Equilibrium states on the Toeplitz algebras of small higher-rank graphs},
preprint available at arXiv:1905.01001.

\bibitem{Kak14} E.T.A. Kakariadis,
\textit{A Note on the Gauge Invariant Uniqueness Theorem for C*-correspondences},
Israel J.\ Math.\ {\bf 215:2} (2016), 513--521. 

\bibitem{Kak15} E.T.A. Kakariadis,
On Nica-Pimsner algebras of C*-dynamical systems over $\mathds{Z}_+^n$.
Internat.\ Math.\ Res.\ Not.\ \textbf{4} (2017), 1013--1065.

\bibitem{KK11} E.T.A. Kakariadis and E.G. Katsoulis, 
\textit{Contributions to the theory of C*-correspondences with applications to multivariable dynamics}, 
Trans.\ Amer.\ Math.\ Soc. \textbf{364:12} (2012), 6605--6630.

\bibitem{KK12} E.T.A. Kakariadis and E.G. Katsoulis, 
\textit{C*-algebras and Equivalences for C*-correspondences},  
J.\ Funct.\ Anal.\ \textbf{266:2} (2014), 956--988.

\bibitem{KS15} E.T.A. Kakariadis and O.M. Shalit,
\textit{Operator algebras of monomial ideals in noncommuting variables}, 
J.\ Math.\ Anal.\ Appl.\ \textbf{472:1} (2019), 738--813.

\bibitem{KK06} E.G. Katsoulis and D.W. Kribs, 
\textit{Tensor algebras of C*-correspondences and their C*-envelopes}, 
J.\ Funct.\ Anal.\ \textbf{234:1} (2006), 226--233.

\bibitem{Kat03} T. Katsura,
\textit{A construction of C*-algebras from C*-correspondences},
Advances in quantum dynamics (South Hadley, MA, 2002), 173--182, Contemp.\ Math., \textbf{335}, Amer.\ Math.\ Soc., Providence, RI, 2003.

\bibitem{Kat04} T. Katsura,
\textit{On C*-algebras associated with C*-correspondences},
J.\ Funct.\ Anal.\ \textbf{217:2} (2004), 366--401.

\bibitem{Kat07} T. Katsura,
\textit{Ideal structure of C*-algebras associated with C*-correspondences}, 
Pacific J.\ Math.\ \textbf{230:1} (2007), 107--145.

\bibitem{KP06} D.W. Kribs and S.C. Power, 
\textit{The analytic algebras of higher rank graphs}, 
Math.\ Proc.\ R.\ Ir.\ Acad.\ {\bf 106A:2} (2006), 199--218.

\bibitem{KP00} A. Kumjian and D. Pask, 
\textit{Higher rank graph C*-algebras}, 
New York J.\ Math.\ {\bf 6} (2000), 1--20.

\bibitem{KP03} A. Kumjian and D. Pask, 
\textit{Actions of $\bZ^k$ associated to higher rank graphs}, 
Ergodic Theory Dynam.\ Systems {\bf 23:4} (2003), 1153--1172.

\bibitem{Lan95} E.C. Lance,
\textit{Hilbert C*-modules. A toolkit for operator algebraists},
London Mathematical Society Lecture Note Series, \textbf{210}. 
Cambridge University Press, Cambridge, 1995.

\bibitem{LM95} D. Lind and B. Marcus,
\textit{An Introduction to Symbolic Dynamics and Coding},
Cambridge University Press, Cambridge, 1995.

\bibitem{Mat97} K. Matsumoto,
\textit{On C*-algebras associated with subshifts},
Internat.\ J.\ Math.\ \textbf{8:3} (1997), 357--374.

\bibitem{Mat02} K. Matsumoto, 
\emph{C*-algebras associated with presentations of subshifts}, 
Doc.\ Math.\ \textbf{7} (2002), 1--30. 

\bibitem{MPT08} P.S. Muhly, D. Pask and M. Tomforde, 
\textit{Strong Shift Equivalence of C*-correspondences}, 
Israel J.\ Math.\ \textbf{167} (2008), 315--345.

\bibitem{MS98} P.S. Muhly and B. Solel, 
\textit{Tensor algebras over C*-correspondences: representations, dilations and C*-envelopes} J. 
Funct.\ Anal.\ \textbf{158:2} (1998), 389--457.

\bibitem{MS00} P.S. Muhly and B. Solel, 
\textit{On the Morita Equivalence of Tensor Algebras}, 
Proc.\ London Math.\ Soc.\ (3) {\bf 81:1} (2000), 113--168.

\bibitem{MT04} P.S. Muhly and M. Tomforde, 
\textit{Adding tails to C*-correspondences}, 
Doc.\ Math.\ \textbf{9} (2004), 79--106.

\bibitem{Nic92} A. Nica,
\textit{C*-algebras generated by isometries and Wiener-Hopf operators},
J.\ Operator Theory \textbf{27:1} (1992), 17--52.

\bibitem{Pim97} M.V. Pimsner,
\textit{A class of C*-algebras generalizing both Cuntz-Krieger algebras and crossed products by $\bZ$},
Fields Inst. Commun., \textbf{12}, Amer.\ Math. Soc., Providence, RI, 1997, 189--212.

\bibitem{PZ05} I. Popescu and J. Zacharias, 
\textit{E-theoretic duality for higher rank graph algebras}, 
K-Theory {\bf 34:3} (2005), 265--282.

\bibitem{RS03} I. Raeburn and A. Sims,
\textit{Product systems of graphs and the Toeplitz algebras of higher-rank graphs}, 
J.\ Operator Theory {\bf 53} (2005), 399-429.

\bibitem{RSY03} I. Raeburn, A. Sims, and T. Yeend, 
\textit{Higher-rank graphs and their C*-algebras}, 
Proc.\ Edinb.\ Math.\ Soc.\ (2) {\bf 46} (2003), 99--115.

\bibitem{RSY04} I. Raeburn, A. Sims, and T. Yeend, 
\textit{The C*-algebras of finitely aligned higher-rank graphs},
J.\ Funct.\ Anal.\ \textbf{213:1} (2004), 206--240.

\bibitem{RS04} I. Raeburn and W. Szymanski,
\textit{Cuntz-Krieger algebras of infinite graphs and matrices}
Trans.\ Amer.\ Math.\ Soc.\ \textbf{356} (2004)  39--59.

\bibitem{RS96} A.G. Robertson and T. Steger, 
\textit{C*-algebras arising from group actions on the boundary of a triangle building},
Proc.\ London Math.\ Soc.\ (3) {\bf 72:3} (1996), 613--637.

\bibitem{RS99} A.G. Robertson and T. Steger, 
\textit{Affine buildings, tiling systems and higher rank Cuntz-Krieger algebras},
J.\ Reine Angew.\ Math.\ {\bf 513} (1999), 115--144. 

\bibitem{SS09} O.M. Shalit and B. Solel,
\textit{Subproduct Systems},
Doc.\ Math.\ \textbf{14} (2009), 801--868.

\bibitem{SY11} A. Sims and T. Yeend,
\textit{Cuntz-Nica-Pimsner algebras associated to product systems of Hilbert bimodules},
J.\ Operator Theory {\bf 64:2} (2010), 349--376.

\bibitem{Yee07} T. Yeend,
\textit{Groupoid models for the C*-algebras of topological higher-rank graphs},
J.\ Operator Theory \textbf{57:1} (2007), 95--120.


\end{thebibliography}
\end{document}